\documentclass[preprint,12pt]{elsarticle}

\usepackage{amssymb}
\usepackage{amsmath}
\usepackage{amstext}
\usepackage{caption}
\usepackage{amsthm}
\usepackage{amsfonts}
\usepackage{algorithm}
\usepackage{amssymb}
\usepackage{array}
\usepackage{algpseudocode}
\usepackage{graphicx, subcaption}
\usepackage{xcolor}
\usepackage{mathtools}
\usepackage{booktabs}
\newtheorem{theorem}{Theorem}
\newtheorem{remark}{Remark}

\def\bx{{\mathbf{x}}}
\def\bn{{\mathbf{n}}}
 
\def\W11{{{\mathop{W^{1,1}_{2}}\limits^{\bullet\,\,\,\,\,\,\,}}}}

\journal{Journal of Computational Physics}
\journal{Elsevier}

\begin{document}

\begin{frontmatter}



\title{Improved randomized neural network  methods  with boundary processing for solving elliptic equations}

\author[label1]{Zhou, Huifang}  
 \ead{zhouhuifang@jlu.edu.cn}
\affiliation[label1]{organization={School of Mathematics,  Jilin University},
	city={Changchun},
	postcode={130012}, 
	state={Jilin Province},
	country={P.R. China}} 

 \author{Sheng, Zhiqiang\corref{cor1}\fnref{label2,label3}}       
  \ead{sheng\_zhiqiang@iapcm.ac.cn}     
 \cortext[cor1]{Corresponding author.}
\affiliation[label2]{organization={Laboratory of Computational Physics, Institute of Applied Physics and Computational Mathematics},
	postcode={100088}, 
	state={Beijing},
	country={P.R. China}} 

 \affiliation[label3]{organization={HEDPS, Center for Applied Physics and Technology, and College of Engineering, Peking University},
 	postcode={100871}, 
 	state={Beijing},
 	country={P.R. China}}

\begin{abstract}
 We present two improved randomized neural network   methods, namely    RNN-Scaling and RNN-Boundary-Processing (RNN-BP) methods, for solving elliptic equations such as  the Poisson equation and the biharmonic equation.   The RNN-Scaling method modifies the optimization objective by increasing the weight of    boundary equations, resulting in a more accurate  approximation. We propose the boundary processing techniques on the  rectangular  domain that enforce the RNN method to  satisfy the non-homogeneous Dirichlet and clamped boundary conditions exactly. We further prove that the RNN-BP method is exact  for some solutions with specific forms  and  validate it  numerically.   Numerical experiments demonstrate that the RNN-BP method is the most accurate among the  three methods, the error is reduced by 6 orders of magnitude for some tests. 
\end{abstract}




\end{frontmatter}

 \section{Introduction}
 Elliptic partial differential equations (PDEs) model steady-state conditions in various physical phenomena,  including electrostatics, gravitational fields, elasticity, phase-field models, and image processing \cite{Barrett1999a, Henn2005, Ventsel2001}. For instance, the Poisson equation characterizes the distribution of a scalar field based on boundary conditions and interior sources. In contrast,  the biharmonic equation is employed to model    phenomena such as the deflection of elastic plates and the flow of incompressible, inviscid fluids.
 Solving these equations is essential for understanding and predicting system behavior in various applications. 
 The traditional numerical methods for solving  the elliptic equations, such as  the finite difference methods \cite{Ben-Artzi2009, Bialecki2012, Mohanty2000, Xu2023},    finite element methods \cite{Ming2006, Ming2007, Park2013, Zhang2020b},   finite volume methods \cite{LePotier2005a, Sheng2008, Sheng2016},  spectral methods \cite{Bialecki2010, Doha2008, Mai-Duy2007}, and the weak Galerkin finite element  method \cite{Cui2020, Mozolevski2003,  Zhang2015}  have been well studied and widely used. 
 However, these methods often require   careful discretization to obtain numerical solutions with high accuracy.
 Moreover, they may face challenges in handling mesh generation on  complex domains and boundary conditions.
 
 In recent years,  deep neural network (DNN) methods  have been greatly developed  in various fields, such as image recognition, natural language processing, and scientific computing. 
 One area where DNN has shown promise is in solving PDEs, including  elliptic  equations. The DNN-based method transforms the process of solving PDEs into optimization problems  and utilizes gradient backpropagation to adjust the network parameters and minimize the residual error of the PDEs.  Several effective  DNN-based methods include the Physics-Informed Neural Networks (PINNs) \cite{Raissi2019}, the deep Galerkin method   \cite{Sirignano2018}, the deep Ritz method \cite{E2018}, and the deep mixed residual method \cite{Lyu2022},  among others \cite{Dong2021a, Zang2020}.  The main difference between these methods lies in the construction  of the loss function.  
 
 PINNs offer a promising approach for solving  various types of PDEs.  However, they still have   limitations.   One major limitation   is the relatively low accuracy of the solutions \cite{Jagtap2020}, the absolute error rarely goes below the level of $10^{-3}$ to $10^{-4}$. The accuracy at  such levels is less than satisfactory for scientific computing and in some cases, they may fail to converge.  Another   limitation is that PINNs require a  high computational cost  and training time, making them less practical for large-scale or complex problems.  PINNs require substantial resources to integrate the PDEs into the  training process, especially for the problems with high-dimensional PDEs or those requiring fine spatial and temporal resolutions.

 RNN  has recently attracted increasing attention for its application in solving partial differential equations.
 The weights and biases  of the RNN method  are randomly generated and  fixed and don't need to be trained. The optimization problem of PINNs is usually a complicated nonlinear optimization problem, and a great number of training steps are required. For the RNN method, the resulting optimization problem is a least squares problem, which can be solved  without training steps. 
 
 For deep neural networks, the exact imposition of boundary and initial conditions is crucial for the training speed and accuracy of the model, which may accelerate the convergence of the training process and improve overall accuracy. 
 For instance, the inexact enforcement of boundary and initial conditions  severely affects the convergence and accuracy of PINN-based methods \cite{sukumar2022}. Recently, many methods have been developed for the exact imposition of   Dirichlet and Neumann boundary conditions, which leads to more efficient and accurate training.  The main approach is to divide the numerical approximation into two parts:
 a deterministic function satisfying the boundary condition and a trainable function with the homogeneous condition. This idea was first proposed by Lagaris et al. in \cite{Lagaris1998,Lagaris2000}.
 The exact enforcement of boundary conditions is applied in the deep Galerkin method and deep Ritz method for elliptic problems in  \cite{Chen2020}.	The deep mixed residual method is employed in \cite{Lyu2021,Lyu2022}   for solving  PDEs, they satisfy the  Neumann boundary condition exactly by transforming the boundary condition into a  homogeneous Dirichlet boundary.  In \cite{liu2023a}, the authors propose  a  gradient-assisted PINN for solving nonlinear biharmonic equations,  introducing gradient auxiliary functions to transform the  clamped or simply supported boundary conditions into the Dirichlet boundary conditions and then constructing   composite functions to  satisfy these Dirichlet boundary conditions  exactly. 
 However, introducing  the gradient-based auxiliary function or additional  neural networks into a model    leads to an increase in computation and may  introduce additional errors. 
In \cite{Dong2021a}, the authors  use  the universal approximation property of DNN  and specifically designed periodic layers to ensure that the DNN's solution  satisfies the specified   periodic boundary conditions, including both $C^{\infty}$ and $C^k$ conditions. 
The PINN method proposed in \cite{sukumar2022} exactly satisfies the  Dirichlet, Neumann, and Robin boundary conditions on complex geometries. The main idea is to utilize R-functions and mean value potential fields to construct approximate distance functions, and use transfinite interpolation to obtain approximations.
   A penalty-free neural network method is developed  to solve  second-order boundary-value problems on complex geometries  in \cite{Sheng2021a}   by  using two neural networks to satisfy essential boundary conditions, and introducing a length factor function to decouple the networks.
The boundary-dependent PINNs   in \cite{Xie2023}   solve PDEs  with complex boundary conditions.  The neural network utilizes the radial basis functions    to  construct trial functions that satisfy boundary conditions automatically,   thus avoiding the need for manual trial function design when dealing with complex boundary conditions.

 In this work, we propose two improved RNN methods for solving elliptic equations, including  the Poisson and biharmonic equations.  Based on the observation that the error  of the  RNN method is concentrated around the boundary, we propose two methods to reduce the boundary error.
 The first method called  RNN-Scaling method  adjusts the optimization problem,  resulting in a modified least squares equation.
 The second improved  RNN method  called   RNN-BP method introduces    interpolation techniques    to enforce  the exact  inhomogeneous Dirichlet or clamped boundary conditions     on   rectangular domains.  
 We provide  extensive numerical experiments to compare the accuracy  of the RNN method with the improved  RNN  methods,  varying  different  numbers of collocation points and    width of the last hidden layer. The numerical results confirm the effectiveness of both improved  methods. The main contributions of this paper can be summarized as follows:	 
 \begin{itemize}
 	\item  The RNN-BP method significantly reduces the error of the   RNN method by enforcing the  inhomogeneous Dirichlet or clamped  boundary conditions.  Specifically,  the RNN-BP method directly deals with the clamped boundary condition without introducing the gradient auxiliary variables.
 	As a result,   the optimization problem does not need to introduce the constraints of gradient relationships,    potentially avoiding  additional errors.
 	
 	\item  The RNN-BP method is proved to be  exact  for the  solutions  of the form $u(x,y) = f_1(x)p_1(y) + f_2(y)p_2(x)$. For the Poisson equation,   $p_1$ and $p_2$ are polynomials of degree no higher than 1,  and  $f_1, f_2$  are  functions in $  C(\Omega)$.  For the biharmonic equation, $p_1$ and $p_2$ are polynomials of degree no higher than 3, and $f_1, f_2$  are  functions in $  C^1(\Omega)$. 
 	
 	\item  The RNN-Scaling method increases the weight of   boundary equations in the optimization problem,  resulting in   a more accurate approximation without taking more collocation points.
 \end{itemize}
 
The remainder of this paper is organized as follows. In Sections 2 and 3, we describe the RNN methods for solving the Poisson and biharmonic equations, respectively. In Section 4, we present numerical examples to illustrate the effectiveness of the RNN-Scaling and RNN-BP methods. In Section 5, we provide a conclusion for this paper.
 
 \section{The improved RNN methods for the Poisson equation}
 In this section, we focus on the Poisson equation with the Dirichlet boundary condition   on the two-dimensional  bounded domain $\Omega \in \mathbb{R}^2$ with   boundary $\partial \Omega$:
 \begin{align}
 	\left\{\begin{aligned}\label{equation-Poisson}
 		-\Delta  u(\bx)& =f(\bx),  &\text{in }& \Omega, \\
 		u(\bx)&=g(\bx),  &\text{on }& \partial \Omega.\\
 	\end{aligned}\right.
 \end{align}
 The source term $f(\bx)$ and the boundary  condition $g(\bx)$  are   given functions.

 \subsection{Randomized neural networks}
 	The randomized neural network employs  a fully connected neural network.  Let $L$ denote the number of hidden layers,
 	let   $M$ denote  the number  of neurons in the last hidden layer, and let  $\phi_j(\bx)$ denote the output of the $j$-th neuron  in  the last hidden layer, where $1\leq j \leq M$.
 	The fully connected neural network is represented by
 	\[
 	\hat u(\bx) = \sum_{j=1}^M \omega_j \phi_j(\bx) = \omega \Phi(\bx),  \quad \bx \in \Omega, 
 	\]
 	where $\Phi(\bx) = (\phi_1(\bx), \cdots, \phi_M(\bx)) = \sigma (W_{L} \cdot \sigma (\cdots \sigma (W_2 \cdot \sigma(W_1 \cdot \bx+b_1 )+b_2) \cdots )+b_{L} )$ 
 	and $\omega = (\omega_1, \cdots, \omega_M)^{\mathrm{T}}$, \( W_k \in \mathbb{R}^{n_k \times n_{k-1}} \) and \(b_k \in \mathbb{R}^{n_k} \) denote the weight matrices and bias vectors,   respectively, and $n_k$ denotes the number  of neurons of the $k$-th hidden layer.    We   employ  two methods  to generate  weights and biases  in this paper:  one of these methods is the default initialization method in PyTorch,  and the other is the uniform random initialization with a  distribution range between $[-R_m, R_m]$.

 	\begin{figure}[!htbp]
 		\centering 
 		\includegraphics[width=5in]{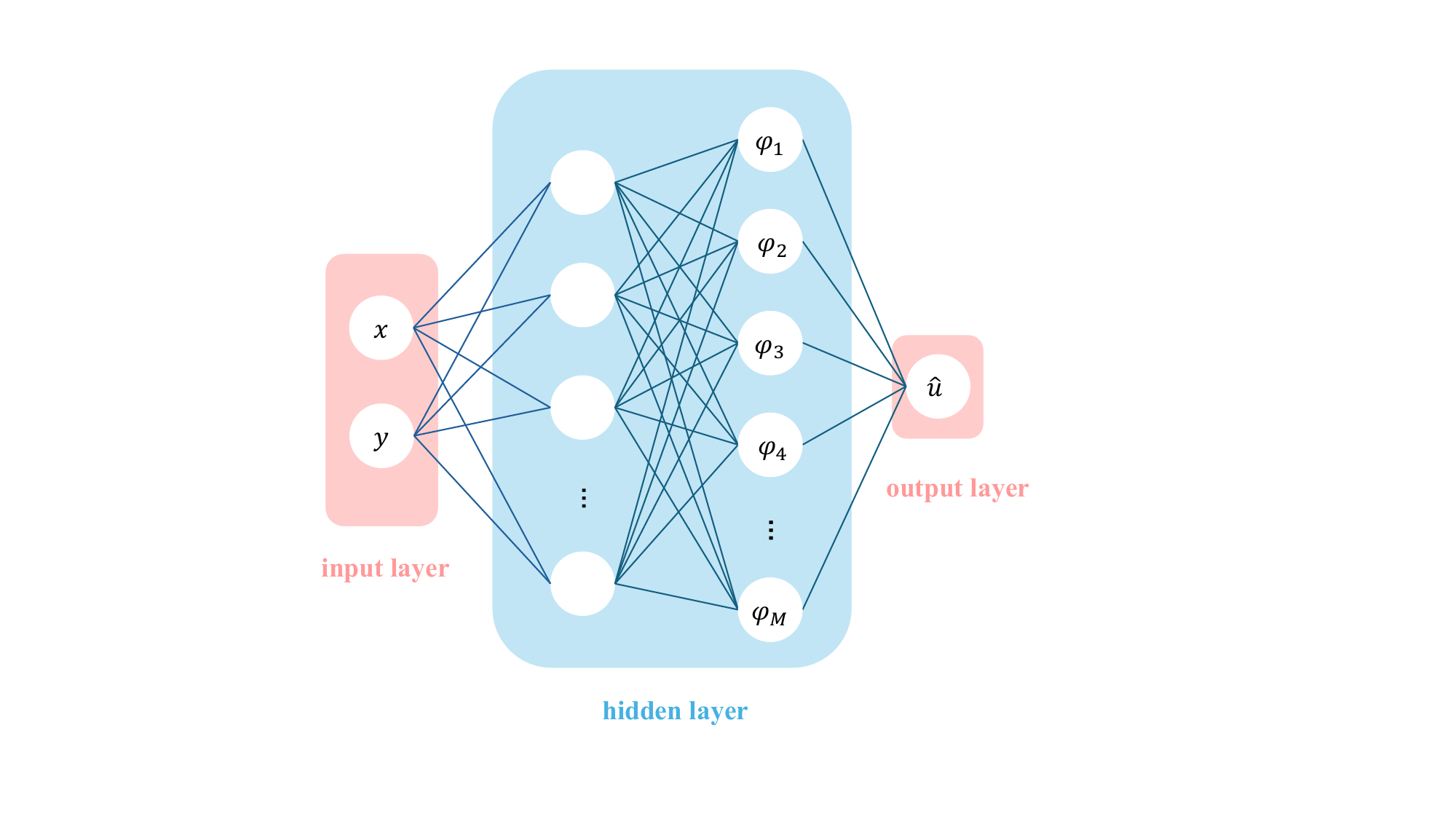}
 		\caption{The architecture of the RNN.}\label{RNN_model}
 	\end{figure}
 	
 	\subsection{RNN method}
 	First,  we select collocation points.  Collocation points are divided into two types:  the interior   points   and the boundary   points.  The  interior collocation points  consist  of $N_f$ points in $\Omega$.  The boundary collocation points consist  of $N_b$ points on  $ \partial\Omega$. 
 	The selection of collocation points is not unique,  they can be random or uniform. In this paper, we  select uniform collocation points.
 	
 	The basic idea of the RNN  is that the weights and biases of the hidden layers are randomly generated and remain fixed. Thus,  we only need to  solve a least squares problem and don't need to train.  
 	The system of linear algebraic equations is
 	\begin{equation}
 		\begin{aligned}\label{LS-1}
 			-\sum_{j=1}^M \omega_j \Delta \varphi_j(\mathbf{x}_f^i) & =f(\mathbf{x}^i_f),   &i=1,\cdots,N_f,
 			\\
 			\sum_{j=1}^M \omega_j  \varphi_j(\mathbf{x}_b^i) & =g(\mathbf{x}_b^i), &i = 1, \cdots, N_b.
 		\end{aligned}
 	\end{equation}
 	Solving this system of equations yields   $\omega$, which consequently yields  the solution  	\( \hat  u(\bx) \).
  
 	We employ the RNN method to solve the Poisson equation \eqref{equation-Poisson} with the exact solution $u=\sin(2\pi x)\sin(2\pi y)$. The absolute error of the RNN method 
 	is plotted in Fig. \ref{error}.
 	\begin{figure}[!htbp]
 			\centering
 			\includegraphics[width=2.5in]{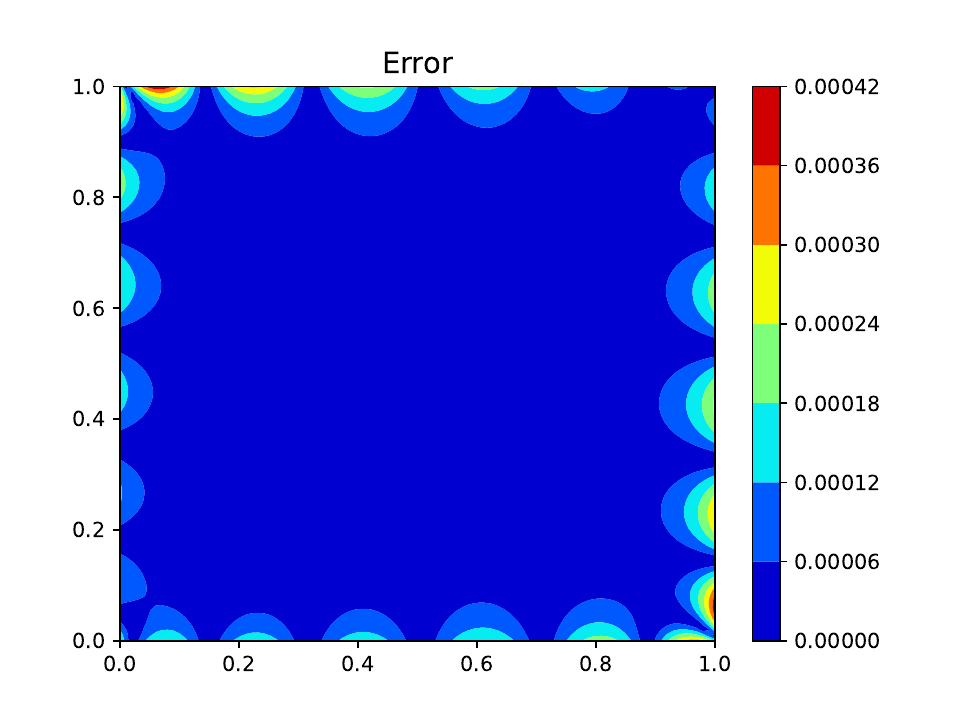}
 			\caption{The absolute  error of the RNN method.} \label{error}
 	\end{figure}
It is observed that the error of the RNN method concentrates around  the boundary $\partial\Omega$. To  improve the accuracy of the RNN method,  it is necessary to approximate more accurately around the boundary.
 	
 	\subsection{RNN-Scaling method}
 	To get a more accurate approximation around the boundary, a straightforward idea is to select a larger $N_b$, would require placing more collocation points on
 	on $\partial\Omega$. However, this may increase the computational cost, and the improvement is not significant. Therefore, we slightly   adjust the  algebraic equations \eqref{LS-1}, to enlarge the weight of  boundary equations.

 	We present the RNN-Scaling method with an example on the square domain. Let the number of collocation points on both the interior and  boundary in both   $x$-direction and $y$-direction be the same, denoted as $N$.
 	 It is obvious that $N_f = N^2$ and $N_f = 4N$.
 The corresponding  equations of the RNN-Scaling method are modified  to
 	\begin{equation}
 		\begin{aligned}\label{LS-sc-Poisson}
 			\frac{1}{N^2}	\sum_{j=1}^M \omega_j L \varphi_j(\mathbf{x}_f^i) & = \frac{1}{N^2} f(\mathbf{x}^i_f),   &i=1,\cdots,N_f,
 			\\
 			\sum_{j=1}^M \omega_j  \varphi_j(\mathbf{x}_b^i) & =g(\mathbf{x}_b^i), &i = 1, \cdots, N_b.
 		\end{aligned}
 	\end{equation}
 Solving the system \eqref{LS-sc-Poisson} of equations   yields the solution  	\( \hat  u(\bx) \).
 	
 	\begin{figure}[!htbp]
 			\centering
 			\includegraphics[width=2.5in]{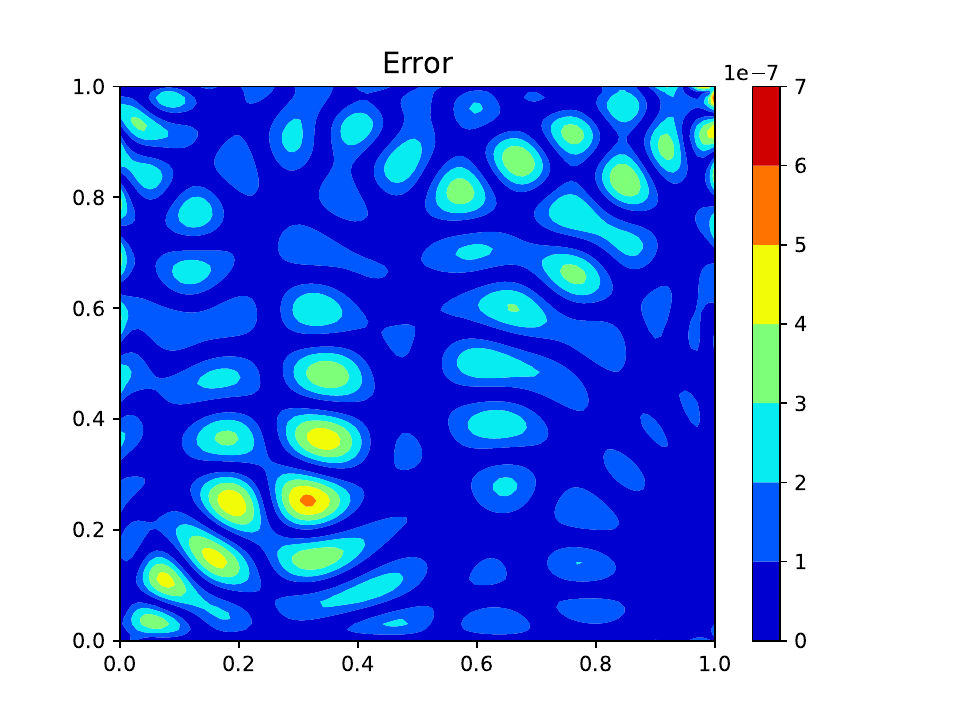}
 			\caption{The absolute  error of the  RNN-Scaling method.} \label{error11}
 	\end{figure}
 	We present the absolute  error of the RNN-Scaling method in Fig. \ref{error11}  and observe that the error of the RNN-Scaling method is much reduced around the boundary $\partial\Omega$.  Consequently, 
 	the total  relative $L^2$ error is also reduced by about  $10^3$ orders of magnitude.
 	\begin{remark}
 		The weight coefficient $\frac{1}{N^2}$ is selected intuitively and numerically. In the classical methods, such as the finite element methods
 		and the finite difference methods, the scale of the equations for interior unknowns is usually larger than that of the boundary unknowns by
 		about $\frac{1}{N^2}$ magnitude.  We  test several numerical examples  and find that $\frac{1}{N^2}$ seems to be a proper 
 		selection for the Poisson equation.
 	\end{remark}
 	
 	\begin{remark}
 		The idea of the  RNN-Scaling could be extended to a general domain. For instance,  on a unit circular domain if we take $h=\frac{1}{N}$ as the ``mesh size’’, i.e., the average distance of collocation points. Then the interior collocation points could be the uniformly distributed points in the unit disk at a distance of $h$, and the boundary collocation points could be $[2\pi N]$ uniformly distributed points on the circle. The distribution of collocation points on the unit circle is illustrated in Fig.   \ref{exp4_Poisson}(a).
 		The weight coefficient is also selected to be $\frac{1}{N^2}$. In this case, the total relative  $L^2$ error is also reduced by  1-2 orders of magnitude.
 	\end{remark}
 	
 	\subsection{RNN-BP method}\label{RNN_BP}
 	As shown in Fig. \ref{error11},  although the error of the  RNN-Scaling method has decreased significantly, it still exists on the boundary $\partial\Omega$.  In this subsection, 
 	we propose a boundary processing technique that enforces  the exact   Dirichlet boundary condition. This boundary processing technique  was also proposed  in \cite{Lagaris1998}.
 The advantage of boundary processing is that the boundary condition is imposed on the numerical solution over the entire boundary, rather than just at collocation points.  We now describe the construction of the RNN-BP method.  For simplicity, we consider the unit square domain $\Omega = (0,1)\times (0,1)$  and denote $\bx$ as $(x, y)$.  It should be noted that  the  technique  can be easily extended to   rectangular domains.
 	
 	\begin{figure}[!htbp]
 		\centering 
 		\includegraphics[width=5in]{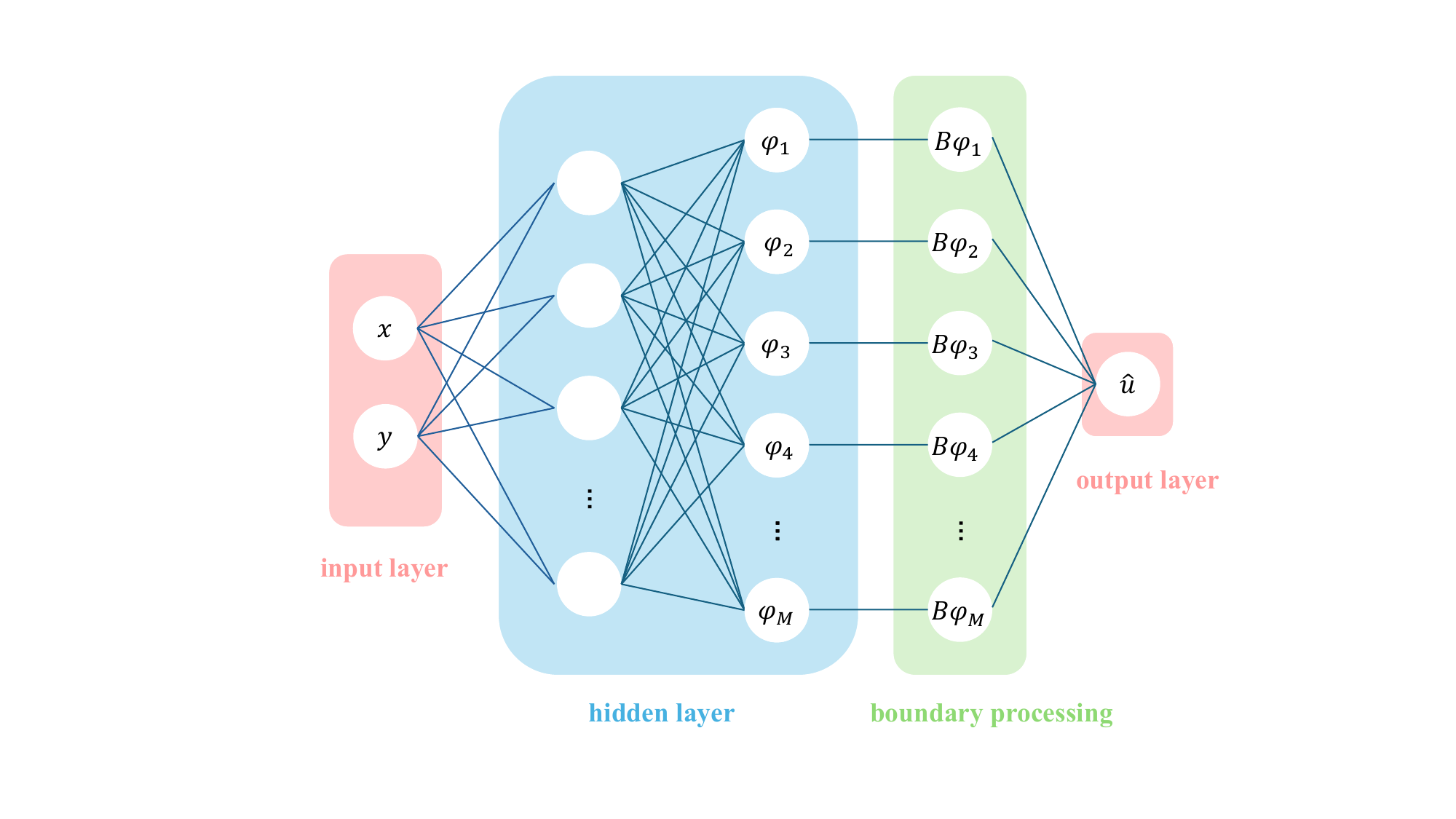}
 		\caption{The architecture of the RNN-BP.}\label{RNN_BP_model}
 	\end{figure}
 	
 	We construct the numerical  solution of the RNN-BP method   as follows:
 	\begin{align}
 		\begin{aligned}\label{u_RNN_BP_Poisson}
 			\hat  u(x,y)=  B(x,y)\sum_{j=1}^{M} \omega_j  \varphi_j(x,y)+  g_D(x, y),
 		\end{aligned}
 	\end{align}
 	where $B$ and $g_D$ should    satisfy
 	\begin{align*}
 		\begin{aligned}
 			B(x,y) &=0, \quad &&\text{on~}  \partial \Omega, \\
 			B(x,y) &\neq 0, \quad &&\text{in~}    \Omega, \\
 			g_D(x,y) &=g(x,y),  \quad&& \text{on~}  \partial \Omega. 
 		\end{aligned}
 	\end{align*} 
 	A straightforward idea is to choose    $B(x,y)= x(1-x)y(1-y)$ for  the Poisson equation.   The  network architecture  is  illustrated in  Fig. \ref{RNN_BP_model}.

 	We next describe the construction of $g_D(x,y)$.  Let  $g_D(x,y) = s(x,y) + P(x,y)$.  We   construct  $P(x,y)$ to  satisfy  the relations  at four  corners:
 	\begin{equation*}
 		\begin{aligned}
 			P(x_i, y_j)&=u(x_i, y_j), \\
 		\end{aligned}
 	\end{equation*}
 	where  $i, j =0, 1$ and $x_0=y_0=0$, $x_1=y_1=1$. 
 	Rewrite the Dirichlet boundary condition  to
 	\begin{align*}
 		\begin{aligned}
 			& u(0, y)=t_1(y),  \quad 0 \leq y \leq 1, \\
 			& u(1, y)=t_2(y),  \quad 0 \leq y \leq 1, \\
 			& u(x, 0)=t_3(x),  \quad 0 \leq x \leq 1, \\
 			& u(x, 1)=t_4(x),  \quad 0 \leq x \leq 1.  
 		\end{aligned}
 	\end{align*}
 By 	using  one-dimensional two-point Lagrange interpolation functions  
 	$l_0(x) =1- x$ and $l_1(x)= x$,    $P(x,y)$  can be expressed as a bilinear Lagrange interpolation function:
 	\begin{equation*}
 		P(x, y)=\left[\begin{array}{c}
 			l_0(x) \\
 			l_1(x) 
 		\end{array}\right]^{\mathrm T}\left[\begin{array}{ll}
 			u(x_0, y_0) & u(x_0, y_1)  \\
 			u(x_1, y_0) & u(x_1, y_1) 
 		\end{array}\right]\left[\begin{array}{c}
 			l_0(y) \\
 			l_1(y) \\
 		\end{array}\right],
 	\end{equation*}
 	or equivalently, 
 	\begin{equation*}
 		P(x, y)=\left[\begin{array}{c}
 			l_0(x) \\
 			l_1(x)
 		\end{array}\right]^{\mathrm T}\left[\begin{array}{llll}
 			t_1(y_0) & t_1(y_1)  \\
 			t_2(y_0)& t_2(y_1)   
 		\end{array}\right]\left[\begin{array}{c}
 			l_0(y) \\
 			l_1(y)
 		\end{array}\right].
 	\end{equation*}
 	Denote $\tilde u(x,y)   \triangleq u(x,y) - P(x,y)$.	Then,  we construct  $s(x,y)$ to satisfy the   Dirichlet condition of $\tilde u$ exactly, i.e., 
 	\begin{align*}
 		s(x,y)= \tilde u(x,y), \quad  \text{on } \partial \Omega. 
 	\end{align*}
 	The boundary conditions of $\tilde u(x,y)$   are denoted  by $\tilde t_i$, and there  hold   
 	\begin{align*}
 		\begin{aligned}
 			& \tilde t_ 1(y) =  t_1(y) - P(0, y), \quad 0 \leq y \leq 1, \\
 			& \tilde t_ 2(y) =  t_2(y) - P(1, y),   \quad 0 \leq y \leq 1, \\
 			&\tilde t_3(x) = t_3(x) - P(x,0),  \quad 0 \leq x \leq 1, \\
 			&  \tilde t_4(x) = t_4(x) - P(x,1),   \quad 0 \leq x \leq 1. 
 		\end{aligned}
 	\end{align*} 
 	We define $s(x,y)$ as 
 	\begin{align*}
 		\begin{aligned}
 			s(x, y) & =l_0(x)\tilde t_1(y)+l_1(x) \tilde t_2(y)+l_0(y)\tilde t_3(x)+l_1(y)\tilde t_4(x).
 		\end{aligned}
 	\end{align*}
  
 After constructing  $s(x,y)$ and $P(x,y)$,  the corresponding equations of the RNN-BP method are derived as follows.
 	\begin{equation}\label{normal-eq-Poisson}
 		\begin{aligned} 
 			-\sum_{j=1}^M \omega_j \Delta  (B \varphi_j)(\mathbf{x}_f^i) & =f(\mathbf{x}^i_f)  +\Delta  s(\bx_f^i) +\Delta P(\bx_f^i),   &i=1,\cdots,N_f.
 		\end{aligned}
 	\end{equation}

 	\begin{theorem}\label{th-exact-bd-Poisson}
 		The numerical solution  \eqref{u_RNN_BP_Poisson}  of the RNN-BP method satisfies the  Dirichlet boundary  condition exactly.
 	\end{theorem}
 	\begin{proof} 	
 		Our aim is to prove that $	\hat  u(x,y) -P(x,y) =\tilde u(x,y)$ on $\partial \Omega$, i.e., $\sum_{i=1}^{M} \omega_i x(1-x)y(1-y) \varphi_i(x,y) + s(x,y) = \tilde u(x,y)$ on $\partial \Omega$.
 		Notice that $\sum_{i=1}^{M} \omega_i x(1-x)y(1-y) \varphi_i(x,y)$ satisfies the homogeneous Dirichlet boundary condition, thus  we just need to prove that $s(0, y)=\tilde t_1(y)$.  
 		It follows   that 
 		$$\begin{aligned}  s(0,y) & =l_0(0)\tilde  t_1(y)+l_1(0) \tilde t_2(y)+l_0(y) \tilde t_3(0)+l_1(y)\tilde  t_4(0) \\ & =\tilde t_1(y),
 		\end{aligned}
 		$$
 		where we use the fact that 	  $\tilde t_i(0)=   \tilde t_i(1)= 0$ for $i=1, \cdots, 4. $
 		
 		The  conditions on the other   boundary   can be proved  similarly. 	It completes the proof.   
 	\end{proof} 
 	
 	\begin{theorem}\label{th1}
 		Assume that the solution to  the least squares problem  \eqref{normal-eq-Poisson}  is uniquely solvable. 
 		Then the  RNN-BP method  is exact for    solutions   of the form 
 		\begin{equation}\label{form-Poisson}
 			f_1(x)p_1(y) + f_2(y)p_2(x),
 		\end{equation}
 		where  $f_1, f_2 \in C(\Omega)$,  $p_1$ and $p_2$ are polynomials of degree no higher than 1.
 	\end{theorem}
 	\begin{proof} 
 		For simplicity, we only provide the proof for the case where    $u(x,y)=	f_1(x)p_1(y)$. The proof for   for the case where  $u(x,y)=	f_1(x)p_1(y)+ f_2(y)p_2(x)$ is similar and  thus   is omitted.
 		
 		The first step is to prove that $u(x,y)=s(x,y)$ under the assumption
 		that $u(x,y) =0$ at the four corner points of $\partial \Omega$.  It is obvious $P(x,y) \equiv 0$ from the above assumption.
 		We  analyze the four terms of  $s(x,y) = l_0(x)  t_1(y)+l_1(x)   t_2(y)+l_0(y)  t_3(x)+l_1(y)  t_4(x)$ sequentially. For the term $t_1(y) = f_1(0)p_1(y) \in \mathbb{P}_1$, we have $t_1(0) = t_1(1) = 0$. Consequently,  $t_1(y) \equiv 0$ in $[0,1]$.  Similarly, we have $t_2(y) \equiv 0$ in $[0,1]$. Hence we obtain
 		$$
 		\begin{aligned}
 			s(x,y)&=  l_0(x)  t_1(y)+l_1(x)   t_2(y)+l_0(y)  t_3(x)+l_1(y)  t_4(x) \\
 			& = l_0(y) f_1(x) p_1(0)+l_1(y) f_1(x) p_1(1) \\
 			& =f_1(x)( l_0(y) p_1(0)+l_1(y) p_1(1) ) \\
 			& =f_1(x) p_1(y).
 		\end{aligned}
 		$$
 		
 		The second step is to prove that  
 		$u(x,y)=s(x,y) + P(x,y)$ when $u(x,y) = f_1(x)p_1(y)$.  Define $Q(x)$
 		as a first-order polynomial  that satisfies  $Q(0) = f_1(0)$ and $Q(1) = f_1(1)$. From the definition of $P(x,y)$ and $Q(x)p_1(y) = P(x,y)$ at four corners of $\partial\Omega$, it follows that  $Q(x)p_1(y) \equiv P(x,y)$. Thus, it yields that
 		\begin{equation}\label{proof-1}
 			\begin{aligned}
 				u(x,y)-P(x,y) &  = (f_1(x)-Q(x)) p_1(y) \\
 				& =s(x, y),    
 			\end{aligned}
 		\end{equation}
 		for any $(x,y)\in \Omega$, where the final equation adopts the conclusion of the first step since $f_1(x)-Q(x) = 0$ at four corners of $\partial\Omega$. Substituting \eqref{proof-1} into \eqref{normal-eq-Poisson} yields that $\sum_{j=1}^M \omega_j \Delta  (B(\bx_f^i) \varphi_j(\bx_f^i)) =  0$  for $i=1,\cdots,N_f$, then it is obvious that $\{\omega_j\}_{j=1}^{M} = 0$ is a solution to the least squares problem. Hence, we obtain
 		$$
 		\hat  u(x,y)=  u(x,y), \quad \text{in~}   \Omega,
 		$$ 
 		which completes the proof.
 	\end{proof}
 	
 	\section{The improved RNN  methods for the biharmonic equation}
 	In this section, we focus on the biharmonic equation with  clamped boundary condition   on the   bounded domain  with  boundary $\partial \Omega$:
 	\begin{align}
 		\left\{\begin{aligned}\label{equation}
 			\Delta^2 u(\bx)& =f(\bx),  &\text{in }& \Omega, \\
 			u(\bx)&=g_1(\bx),  &\text{on }& \partial \Omega,\\
 			\frac{	\partial	u}{\partial \mathbf{n}}(\bx)&=g_2(\bx),  &\text{on }& \partial \Omega.
 		\end{aligned}\right.
 	\end{align}
 	The source term $f(\bx)$ and the  boundary  conditions $g_1(\bx)$ and $g_2(\bx)$ are   given functions,    $\mathbf{n}$ denotes the    unit outward normal vector of    $\partial \Omega$.

 	\subsection{RNN method and RNN-Scaling method}
 For the  biharmonic equation,	the    equations of the   RNN  method are
 	\begin{equation*}
 		\begin{aligned}
 			\sum_{j=1}^M \omega_j 	\Delta^2 \varphi_j(\mathbf{x}_f^i) & =f(\mathbf{x}^i_f),   &i=1,\cdots,N_f,
 			\\
 			\sum_{j=1}^M \omega_j  \varphi_j(\mathbf{x}_b^i) & =g_1(\mathbf{x}_b^i), &i = 1, \cdots, N_b,
 			\\
 			\sum_{j=1}^M \omega_j   \frac{\partial \varphi_j}{\partial \bn} (\mathbf{x}_b^i)& =g_2(\mathbf{x}^i_b),   &i=1,\cdots,N_b,
 		\end{aligned}
 	\end{equation*}
 	where  $N_f$ and  $N_b$ still represent  the  numbers of interior and boundary collocation points, respectively. 
 	Solving this system of equations  yields   $\omega$, which consequently obtains the solution  	\( \hat  u(\bx) \).

 	For the biharmonic equation,  we also observe that the    error of the RNN method is   concentrated around the boundary.  Therefore,  similar to the Poisson equation,  it is necessary to develop an improved method to reduce the error around the boundary.
 	
 	We   describe   the RNN-Scaling method on the square domain. Let the numbers of  collocation points on both interior and boundary in the $x$-direction and $y$-direction be the same, still denoted as $N$, then $N_f = N^2$ and $N_b = 4N$. 
 The least squares problem is then adjusted to
 	\begin{equation}
 		\begin{aligned}\label{LS-2}
 			\frac{1}{N^4}	\sum_{j=1}^M \omega_j 	\Delta^2 \varphi_j(\mathbf{x}_f^i) & = 	\frac{1}{N^4} f(\mathbf{x}^i_f),   &i=1,\cdots,N_f,
 			\\
 			\sum_{j=1}^M \omega_j  \varphi_j(\mathbf{x}_b^i) & =g_1(\mathbf{x}_b^i), &i = 1, \cdots, N_b,
 			\\
 			\frac{1}{N}	\sum_{j=1}^M \omega_j   \frac{\partial \varphi_j}{\partial \bn} (\mathbf{x}_b^i)& = 	\frac{1}{N} g_2(\mathbf{x}^i_f),   &i=1,\cdots,N_b.
 		\end{aligned}
 	\end{equation}
 	Solving the  system \eqref{LS-2}    obtains  the numerical solution  	\( \hat  u(\bx) \).
 	
 	\begin{remark}
 		The idea of the RNN-Scaling could also be extended to general domain for the biharmonic equation. For instance, on a unit circular domain,  if we take $h=\frac{1}{N}$ as the ``mesh size’’,  then the interior collocation points could be the uniformly distributed points in the unit disk at a distance of $h$, and the boundary collocation points could be $[2\pi N]$ uniformly distributed points on the circle. The weight coefficients are also chosen to be $\frac{1}{N^4}$ and $\frac{1}{N}$. The total  relative $L^2$ error is also reduced by  1-2 orders of magnitude.
 	\end{remark}
 	
 	\subsection{RNN-BP method}
 	  We observe that the error of the  RNN-Scaling method is  reduced  significantly, but still exists on the boundary. 
 Therefore,  we enforce the neural network to satisfy the exact clamped  boundary condition   on the square  domain.
 	
The   neural network of RNN-BP method  for the biharmonic equation is 
 	\begin{align}
 		\begin{aligned}\label{u_SNN}
 			\hat  u(x,y)= B(x,y) \sum_{i=1}^{M} \omega_i   \varphi_i(x,y)+  g_D(x,y),
 		\end{aligned}
 	\end{align}
 	where  $B(x,y)$ and $g_D(x,y)$ should satisfy the following conditions:
 	\begin{equation*}
 		\begin{aligned}
 			B(x,y)  &=0, \quad &&\text{on~}  \partial \Omega, \\
 			\frac	{\partial B} {\partial \bn }(x,y) &=0, \quad &&\text{on~}  \partial \Omega, \\
 			B(x,y)  &\neq 0, \quad &&\text{in~}    \Omega, \\
 			g_D(x,y)  &=g_1(x,y),  \quad &&\text{on~}  \partial \Omega, \\
 			\frac	{\partial { g_D} }{\partial \bn }(x,y) &=g_2(x,y), \quad &&\text{on~}  \partial \Omega.
 		\end{aligned}
 	\end{equation*} 
 	For the biharmonic equation,  we  define   $B(x,y)= x^2(1-x)^2y^2(1-y)^2$.

 	Next, we provide the construction  of $g_D(x,y)$.  Let $g_D(x,y) = s(x,y) + P(x,y)$.
 We	construct  $P(x,y)$ to  satisfy    the following relations  at four  corners:
 	\begin{equation*}
 		\begin{aligned}
 			P(x_i, y_j)&=u(x_i, y_j), \\
 			P_x(x_i, y_j)&=u_x(x_i, y_j), \\
 			P_y(x_i, y_j)&=u_y(x_i, y_j), \\
 			P_{x y}(x_i,  y_j) &=u_{x y}(x_i, y_j),
 		\end{aligned}
 	\end{equation*}
 	where $i, j =0, 1$,  $x_0 =y_0=0$, $x_1 = y_1 =1$. 
 	Rewrite the clamped boundary condition   \eqref{equation}  to 
 	\begin{align*}
 		\begin{aligned}
 			& u(0, y)=t_1(y),  \quad \frac{\partial u}{\partial x}(0,y) =h_1(y), \quad 0 \leq y \leq 1, \\
 			& u(1, y)=t_2(y),  \quad \frac{\partial u}{\partial x}(1,y)=h_2(y), \quad 0 \leq y \leq 1, \\
 			& u(x, 0)=t_3(x),  \quad \frac{\partial u}{\partial y}(x,0)=h_3(x), \quad 0 \leq x \leq 1, \\
 			& u(x, 1)=t_4(x),  \quad \frac{\partial u}{\partial y}(x,1)=h_4(x), \quad 0 \leq x \leq 1. 
 		\end{aligned}
 	\end{align*}
 	Recall the two-point Hermite interpolation basis functions 
 	\begin{equation*}
 		\begin{aligned}
 			&H_i(x)=l_i^2(x)\left(1-2 l_i^{\prime}\left(x_i\right)\left(x-x_i\right)\right),  \\
 			&	G_i(x)=l_i^2(x)\left(x-x_i\right),
 		\end{aligned}
 	\end{equation*}
 	for  $i=0,1$,  where   $l_0(x) $ and $l_1(x)$ are the first-order  Lagrange interpolation functions.
We present 	$P(x,y)$  with the following  expression:
 	\begin{equation*}
 		P(x, y)=\left[\begin{array}{c}
 			H_0(x) \\
 			H_1(x) \\
 			G_0(x) \\
 			G_1(x)
 		\end{array}\right]^{\mathrm T}\left[\begin{array}{llll}
 			u(x_0, y_0) & u(x_0, y_1) & u_y(x_0, y_0) & u_y(x_0, y_1) \\
 			u(x_1, y_0) & u(x_1, y_1) & u_y(x_1, y_0) & u_y(x_1, y_1) \\
 			u_x(x_0, y_0) & u_x(x_0, y_1) & u_{x y}(x_0, y_0) & u_{x y}(x_0, y_1) \\
 			u_x(x_1, y_0) & u_x(x_1, y_1) & u_{x y}(x_1, y_0) & u_{x y}(x_1, y_1)
 		\end{array}\right]\left[\begin{array}{c}
 			H_0(y) \\
 			H_1(y) \\
 			G_0(y) \\
 			G_1(y)
 		\end{array}\right].
 	\end{equation*}
Alternatively, the expression for $P(x,y)$ can be written as 
 	\begin{equation*}
 		P(x, y)=\left[\begin{array}{c}
 			H_0(x) \\
 			H_1(x) \\
 			G_0(x) \\
 			G_1(x)
 		\end{array}\right]^{\mathrm T}\left[\begin{array}{llll}
 			t_1(y_0) & t_1(y_1) & h_3(x_0) & h_4(x_0) \\
 			t_2(y_0)& t_2(y_1) & h_3(x_1) & h_4(x_1) \\
 			h_1(y_0) &  h_1(y_1) & h_3'(x_0) & h_4'(x_0) \\
 			h_2(y_0) &  h_2(y_1) &  h_3'(x_1) & h_4'(x_1) 
 		\end{array}\right]\left[\begin{array}{c}
 			H_0(y) \\
 			H_1(y) \\
 			G_0(y) \\
 			G_1(y)
 		\end{array}\right].
 	\end{equation*}
 	
 	Then,  we construct  $s(x,y)$ to satisfy the  clamped condition of $\tilde u(x,y)$, where  $\tilde u(x,y)   \triangleq u(x,y) - P(x,y)$. 
 	\begin{align*}
 		\left\{\begin{aligned}\label{tildeu}
 			s(x,y)&= \tilde u(x,y),  &\text{on }& \partial \Omega,\\
 			\frac{	\partial	s}{\partial \mathbf{n}}(x,y)&= \frac{	\partial	\tilde u}{\partial \mathbf{n}}(x,y),  &\text{on }& \partial \Omega.
 		\end{aligned}\right.
 	\end{align*}
 	The boundary functions of $\tilde u(x,y)$   are denoted  by $\tilde t_i$  and  $\tilde h_i$, and there  hold  
 	\begin{align*}
 		\begin{aligned}
 			& \tilde t_ 1(y) =  t_1(y) - P(0, y),  \quad \tilde h_1(y)  =h_1(y)-\frac{\partial P}{\partial x}(0,y), \quad 0 \leq y \leq 1, \\
 			& \tilde t_ 2(y) =  t_2(y) - P(1, y),  \quad  \tilde h_2(y) =  h_2(y) - \frac{\partial P}{\partial x}(1,y), \quad 0 \leq y \leq 1, \\
 			&\tilde t_3(x) = t_3(x) - P(x,0),  \quad  \tilde h_3(x) = h_3(x) - \frac{\partial P}{\partial y}(x,0), \quad 0 \leq x \leq 1, \\
 			&  \tilde t_4(x) = t_4(x) - P(x,1),  \quad  \tilde h_4(x) = h_4(x) - \frac{\partial P}{\partial y}(x,1), \quad 0 \leq x \leq 1. 
 			&
 		\end{aligned}
 	\end{align*} 
 	Define $s(x, y) $ as
 	\begin{align*}
 		\begin{aligned}
 			s(x, y) & =H_0(x)\tilde t_1(y)+H_1(x) \tilde t_2(y)+H_0(y)\tilde t_3(x)+H_1(y)\tilde t_4(x) \\
 			& +G_0(x)\tilde h_1(y)+G_1(x)\tilde h_2(y)+G_0(y)\tilde h_3(x)+G_1(y) \tilde h_4(x).
 		\end{aligned}
 	\end{align*}

 	Finally, the corresponding   equations for the RNN-BP method  are 
 	\begin{equation}\label{normal-eq-biharmonic}
 		\begin{aligned} 
 			\sum_{j=1}^M \omega_j \Delta^2 (B \varphi_j) (\mathbf{x}_f^i) & =f(\mathbf{x}^i_f)-\Delta^2 s(\bx_f^i) - \Delta^2  P(\bx_f^i),   &i=1,\cdots,N_f.
 		\end{aligned}
 	\end{equation}
 	
 	We emphasize that the clamped boundary condition can be exactly imposed. This boundary processing method   on rectangular domains  is  applicable to any PDEs with the clamped  boundary condition.
 	
 	\begin{theorem}\label{th-exact-bd-biharmonic}
 		The numerical solution  \eqref{u_SNN}  of the RNN-BP method satisfies the clamped boundary condition of \eqref{equation} exactly.
 	\end{theorem}
 	\begin{proof}
 		If we can prove that $\hat  u(x,y)-P(x,y) = \sum_{i=1}^{M} \omega_i B(x,y) \varphi_i(x,y)+  s(x, y)$ satisfies the clamped boundary condition of $\tilde u(x,y )$,  then the theorem follows immediately.  
 		Since $P(x,y)$ is the cubic Hermit interpolation of $u(x,y)$ on $\bar \Omega$, there hold 
 		\begin{align}
 			\begin{aligned}\label{prove1}
 				& \tilde t_i(0)= \tilde h_i(0)=\tilde h_i'(0) = 0, \\
 				&\tilde  t_i(1)= \tilde  h_i(1)=\tilde h_i'(1) = 0, 
 			\end{aligned}
 		\end{align}
 		for $i=1, \cdots, 4$.

 		It is easy to check  that $\sum_{i=1}^{M} \omega_i x^2(1-x)^2 y^2(1-y)^2 \varphi_i(x,y)$ satisfies the homogeneous clamped  boundary condition,  so  we need to prove that $s(0, y)=\tilde t_1(y),    \dfrac{\partial s}{\partial x}(0, y)  =\tilde h_1(y)$.   It follows  from \eqref{prove1} that 
 		$$\begin{aligned}  s(0,y) & =H_0(0)\tilde t_1(y)+H_1(0)\tilde  t_2(y)+H_0(y) \tilde t_3(0)+H_1(y)\tilde  t_4(0) \\ & +G_0(0)\tilde  h_1(y)+G_1(0) \tilde h_2(y)+G_0(y)\tilde  h_3(0)+G_1(y)\tilde  h_4(0) \\ & =\tilde t_1(y),
 			\\
 			\frac{\partial s}{\partial x}(0, y)  & =H_0^{\prime}(0) \tilde t_1(y)+H_1^{\prime}(0)\tilde  t_2(y)+H_0(y) \tilde t_3^{\prime}(0)+H_1(y)\tilde  t_4^{\prime}(0) \\ & +G_0^{\prime}(0) \tilde h_1(y)+G_1^{\prime}(0) \tilde h_2(y)+G_0(y) \tilde h_3^{\prime}(0)+G_1(y)\tilde  h_4^{\prime}(0) \\ & =\tilde h_1(y).\end{aligned}
 		$$
 		The clamped conditions on the other  boundary  can be proved similarly.	It completes the proof.
 	\end{proof}

 	\begin{theorem}
 		Assume that the solution to the least squares problem  \eqref{normal-eq-biharmonic}  is uniquely solvable.
 		Then the  RNN-BP method  is exact for   solutions   of the form 
 		\begin{equation}\label{form-bihar}
 			f_1(x)p_1(y) + f_2(y)p_2(x),
 		\end{equation}
 		where $f_1, f_2 \in C^1(\Omega)$,   $p_1$ and $p_2$ are polynomials of degree no higher than 3.
 	\end{theorem}
 	\begin{proof} 
 		The proof of this theorem   is similar to that of Theorem \ref{th1}, but for completeness, we  still provide the proof.  We still only provide the proof for the case where  $u(x,y)=	f_1(x)p_1(y)$.
 		
 		We also divide our proof into two steps. The first step is to prove that   $u(x,y)=s(x,y)$ under  the assumption 
 		that $u(x,y) =0$ and $\dfrac{\partial u} {\partial \bn}(x,y) = 0$ at the four corner points of $\partial \Omega$. It is obvious $P(x,y) \equiv 0$ from the above assumption.
 		We  analyze the eight terms of  $s(x,y)$ sequentially.  Similar to the proof of  Theorem \ref{th1},   we have $t_1(y) = t_2(y) = h_1(y) = h_2(y) =  0$ in $[0,1]$.  Then it follows that
 		$$
 		\begin{aligned}
 			s(x,y)&=  H_0(y)  t_3(x)+H_1(y)  t_4(x) + G_0(y)  h_3(x) +  G_1(y)  h_4(x)  \\
 			& =f_1(x)( H_0(y)p_1(0) + H_1(y)p_1(1) - G_0(y)p'_1(0) +G_1(y)p'_1(1) ) \\
 			& =f_1(x) p_1(y).
 		\end{aligned}
 		$$
 		
 		The second step is to prove that  $u(x,y)=s(x,y) + P(x,y)$ when  $u(x,y)=	f_1(x)p_1(y)$. Define $Q(x)$
 		as a first-order polynomial such that $Q(0) = f_1(0)$, $Q(1) = f_1(1)$, $Q(0) = -f'_1(0)$ and $Q(1) = f'_1(1)$. From the definition of $P(x,y)$, it can be inferred that $Q(x)p_1(y) \equiv P(x,y)$ on $\Omega$. It yields that
 		\begin{equation}\label{proof-2}
 			\begin{aligned}
 				u(x,y)-P(x,y) &  = (f_1(x)-Q(x)) p_1(y) \\
 				& =s(x, y),
 			\end{aligned}
 		\end{equation}
 		for any $(x,y)\in \Omega$,  where the final equation of  \eqref{proof-2}  adopts the conclusion of the first step. 
 		Substituting \eqref{proof-2} into \eqref{normal-eq-biharmonic} yields  that $\sum_{j=1}^M \omega_j \Delta  (B(\bx_f^i) \varphi_j(\bx_f^i)) = 0$ for $i=1,\cdots,N_f$, which implies  that $\{\omega_j\}_{j=1}^{M} = 0$ is a   solution to the least squares problem.
 		Hence, we obtain
 		\begin{align*}
 			\begin{aligned}
 				\hat  u(x,y)=  u(x,y), \quad \text{in~}   \Omega.
 			\end{aligned}
 		\end{align*}
 		It completes the proof.
 	\end{proof}
 	
\section{Numerical experiments}
In the numerical experiments,  we  compare  the   performance of  the  three methods by using      relative $L^2$ error 
\begin{equation*}
	\|e\|_{L^2}=\frac{\sqrt{\sum_{i=1}^{N_{test}}\left|\hat u(\bx_i)-u(\bx_i)\right|^2}}{\sqrt{\sum_{i=1}^{N_{test}}\left|u(\bx_i)\right|^2}},
\end{equation*}
where $N_{test}$ denotes  the number  of   uniform collocation points in $\bar \Omega$ and we set $N_{test} = 10^4$ for all examples. Throughout this paper, we    employ  Sin   function  as the activation function.  In all tables,  we   abbreviate RNN-Scaling  as RNN-S for simplicity.

\subsection{Poisson equation}
First, we  test our methods  for the Poisson equation.

\subsubsection{Example 1.1}
In the first example, we consider the  exact solution on $\Omega = (0,1)\times (0,1)$:
\begin{equation*}
	u(x,y) = \sin2\pi x \sin2\pi y.
\end{equation*}

The numbers of points in $x$ and $y$ dimensions are set   to be the same,  which are  denoted as  $N$.    We employ a neural network    containing 2 hidden layers with 100 and $M$ neurons, respectively. 
We compare the relative $L^2$ errors of the  RNN, RNN-Scaling, and RNN-BP methods across varying  different numbers  of   collocation points and    $M$.  Tables \ref{exp1-Poisson-km}  and \ref{exp1-Poisson-Rm1}   present the comparison results between the default   and uniform random  initialization $R_m =1$,  respectively.
 
 {\tiny
\begin{table}[!htbp]
	\centering
	\scriptsize
	\caption{Comparison of three  methods with the default  initialization  for Example 1.1.}
	\begin{tabular}{ccccccccccccc}
		\toprule
		\multicolumn{1}{l}{method}    & $N$ & $M$           &  {50}   &  {100}   &  {150}   &  {200}   &  {250}   &  {300}   &  {400}   &  {500}   \\ \midrule
		\multicolumn{1}{l}{RNN}      &  8  & $\|e\|_{L^2}$ & 9.44e-1 & 2.76e-3  & 1.01e-3  & 4.82e-4  & 7.79e-4  & 5.49e-4  & 4.49e-4  & 3.66e-4  \\
		& 12  & $\|e\|_{L^2}$ & 9.15e-1 & 3.13e-4  & 1.11e-6  & 3.99e-7  & 1.83e-7  & 8.14e-7  & 2.53e-7  & 5.97e-7  \\
		& 16  & $\|e\|_{L^2}$ & 9.45e-1 & 6.63e-4  & 2.43e-6  & 2.66e-7  & 1.35e-7  & 4.53e-8  & 2.67e-7  & 1.16e-7  \\
		& 24  & $\|e\|_{L^2}$ & 1.04e+0 & 9.26e-4  & 3.61e-6  & 1.33e-6  & 2.06e-7  & 1.76e-7  & 1.46e-7  & 1.41e-7  \\
		& 32  & $\|e\|_{L^2}$ & 1.14e+0 & 1.04e-3  & 3.99e-6  & 2.51e-6  & 2.84e-7  & 1.87e-7  & 1.27e-6  & 2.83e-7  \\
		& 48  & $\|e\|_{L^2}$ & 1.34e+0 & 1.20e-3  & 4.40e-6  & 2.35e-6  & 1.13e-6  & 1.94e-6  & 1.65e-6  & 1.08e-6  \\ 
		\midrule
		\multicolumn{1}{l}{RNN-S } &  8  & $\|e\|_{L^2}$ & 2.09e-2 & 2.76e-3  & 1.01e-3  & 4.82e-4  & 7.79e-4  & 5.49e-4  & 3.06e-4  & 3.66e-4  \\
		& 12  & $\|e\|_{L^2}$ & 1.75e-2 & 2.04e-5  & 4.35e-7  & 2.18e-7  & 3.04e-7  & 3.38e-7  & 6.22e-7  & 5.11e-7  \\
		& 16  & $\|e\|_{L^2}$ & 1.57e-2 & 1.42e-5  & 9.53e-8  & 8.78e-8  & 1.08e-7  & 2.19e-8  & 1.65e-7  & 2.60e-8  \\
		& 24  & $\|e\|_{L^2}$ & 1.43e-2 & 1.31e-5  & 1.75e-7  & 3.29e-7  & 5.43e-8  & 8.21e-8  & 2.56e-7  & 3.32e-8  \\
		& 32  & $\|e\|_{L^2}$ & 1.41e-2 & 1.29e-5  & 3.67e-7  & 3.42e-7  & 1.63e-7  & 8.12e-8  & 2.43e-7  & 5.47e-8  \\
		& 48  & $\|e\|_{L^2}$ & 1.48e-2 & 1.27e-5  & 3.31e-6  & 1.26e-6  & 1.37e-6  & 3.31e-6  & 2.18e-6  & 2.27e-7  \\ 
		\midrule
		\multicolumn{1}{l}{RNN-BP}    &  8  & $\|e\|_{L^2}$ & 3.55e-4 & 3.42e-4  & 1.89e-4  & 2.43e-4  & 1.01e-4  & 1.02e-4  & 1.64e-4  & 8.70e-5  \\
		& 12  & $\|e\|_{L^2}$ & 1.09e-4 & 2.37e-7  & 7.55e-8  & 5.59e-8  & 4.99e-8  & 7.74e-8  & 4.37e-8  & 7.58e-8  \\
		& 16  & $\|e\|_{L^2}$ & 1.05e-4 & 3.72e-8  & 7.54e-10 & 4.63e-10 & 6.68e-10 & 8.87e-11 & 6.72e-10 & 3.61e-10 \\
		& 24  & $\|e\|_{L^2}$ & 1.02e-4 & 2.30e-8  & 1.03e-10 & 1.20e-10 & 3.86e-10 & 1.13e-10 & 1.57e-10 & 3.86e-11 \\
		& 32  & $\|e\|_{L^2}$ & 9.89e-5 & 2.24e-8  & 1.63e-10 & 1.29e-10 & 2.36e-10 & 2.66e-10 & 2.74e-10 & 1.17e-10 \\
		& 48  & $\|e\|_{L^2}$ & 9.40e-5 & 2.13e-8  & 3.53e-10 & 5.00e-10 & 1.81e-10 & 1.37e-10 & 4.14e-10 & 4.34e-11 \\ 
		\bottomrule
		&     &               &         &
	\end{tabular}%
	\label{exp1-Poisson-km}%
\end{table}%
 }

\begin{table}[!htbp]
	\centering
		\scriptsize
	\caption{Comparison of  three  methods  with uniform random  initialization  ($R_m =1$)  for Example 1.1.}
	\begin{tabular}{cccccccccccccccc}
		\toprule
		\multicolumn{1}{l}{method} & $N$   & $M$           & {50} & {100} &{150} & {200} & {250} & {300} & {400} & {500} \\
		\midrule
		\multicolumn{1}{l}{RNN} & 8     & $\|e\|_{L^2}$ & 2.19e+0 & 7.02e-3 & 4.37e-3 & 3.77e-3 & 2.89e-3 & 2.84e-3 & 4.05e-3 & 2.71e-3 \\
		& 12    & $\|e\|_{L^2}$ & 2.28e+0 & 5.55e-2 & 4.95e-4 & 1.32e-4 & 3.07e-5 & 2.36e-5 & 2.79e-5 & 2.35e-5 \\
		& 16    & $\|e\|_{L^2}$ & 2.47e+0 & 6.08e-2 & 1.88e-3 & 3.34e-5 & 7.44e-7 & 9.90e-7 & 1.22e-7 & 2.34e-7 \\
		& 24    & $\|e\|_{L^2}$ & 2.73e+0 & 6.30e-2 & 2.46e-3 & 8.14e-5 & 3.02e-6 & 1.15e-7 & 2.86e-10 & 4.17e-11 \\
		& 32    & $\|e\|_{L^2}$ & 2.94e+0 & 6.48e-2 & 2.73e-3 & 9.64e-5 & 3.61e-6 & 1.78e-7 & 9.82e-10 & 3.54e-11 \\
		& 48    & $\|e\|_{L^2}$ & 3.25e+0 & 6.83e-2 & 3.05e-3 & 1.11e-4 & 4.24e-6 & 2.06e-7 & 1.62e-9 & 1.24e-10 \\
		\midrule
		\multicolumn{1}{l}{RNN-S} & 8     & $\|e\|_{L^2}$ & 9.25e-2 & 7.02e-3 & 4.37e-3 & 3.77e-3 & 2.89e-3 & 2.84e-3 & 4.05e-3 & 2.71e-3 \\
		& 12    & $\|e\|_{L^2}$ & 8.07e-2 & 6.80e-4 & 4.51e-5 & 1.32e-4 & 3.07e-5 & 2.36e-5 & 2.79e-5 & 2.35e-5 \\
		& 16    & $\|e\|_{L^2}$ & 8.58e-2 & 5.73e-4 & 1.61e-5 & 8.15e-7 & 5.36e-7 & 3.49e-7 & 1.51e-7 & 2.34e-7 \\
		& 24    & $\|e\|_{L^2}$ & 9.02e-2 & 5.07e-4 & 1.22e-5 & 3.36e-7 & 1.87e-8 & 2.06e-9 & 1.42e-10 & 3.48e-11 \\
		& 32    & $\|e\|_{L^2}$ & 9.27e-2 & 5.00e-4 & 1.13e-5 & 3.01e-7 & 1.26e-8 & 5.67e-10 & 1.20e-11 & 2.69e-12 \\
		& 48    & $\|e\|_{L^2}$ & 9.85e-2 & 5.12e-4 & 1.08e-5 & 2.87e-7 & 1.24e-8 & 4.77e-10 & 5.73e-12 & 8.27e-13 \\
		\midrule
		\multicolumn{1}{l}{RNN-BP} & 8     & $\|e\|_{L^2}$ & 1.15e-2 & 1.01e-2 & 4.40e-3 & 3.90e-3 & 5.34e-3 & 6.46e-3 & 4.21e-3 & 4.61e-3 \\
		& 12    & $\|e\|_{L^2}$ & 7.21e-3 & 1.77e-4 & 4.44e-5 & 1.05e-4 & 1.20e-4 & 5.07e-5 & 1.26e-4 & 5.76e-5 \\
		& 16    & $\|e\|_{L^2}$ & 6.87e-3 & 5.64e-5 & 2.17e-6 & 6.42e-7 & 1.30e-6 & 3.11e-7 & 1.66e-7 & 5.13e-7 \\
		& 24    & $\|e\|_{L^2}$ & 6.43e-3 & 4.06e-5 & 4.01e-7 & 1.38e-8 & 2.18e-9 & 2.51e-10 & 2.45e-11 & 2.12e-11 \\
		& 32    & $\|e\|_{L^2}$ & 6.18e-3 & 3.85e-5 & 3.80e-7 & 8.44e-9 & 4.35e-10 & 4.56e-11 & 9.71e-13 & 3.72e-13 \\
		& 48    & $\|e\|_{L^2}$ & 5.91e-3 & 3.57e-5 & 3.82e-7 & 8.42e-9 & 3.63e-10 & 1.69e-11 & 8.15e-14 & 2.55e-14 \\
		\bottomrule&               
	\end{tabular}%
	\label{exp1-Poisson-Rm1}%
\end{table}%
From Tables \ref{exp1-Poisson-km}  and \ref{exp1-Poisson-Rm1}, we conclude the following observations.  The errors of  all three  methods  show  a trend of first decreasing and then stabilizing under both two initialization methods.
The errors of the  RNN-scaling method  are   smaller than the RNN method.  Especially under uniform random initialization, the performance is more stable,  and the errors gradually decrease as  $N$ and $M$ increase.
The RNN-BP method, under both initialization methods,  achieves smallest among the three methods  and these errors   decrease  as $N$ and $M$ increase.
\begin{figure}[!htbp]
	\centering 
	\begin{subfigure}[t]{0.32\linewidth}
		\centering
		\includegraphics[width=1.9in]{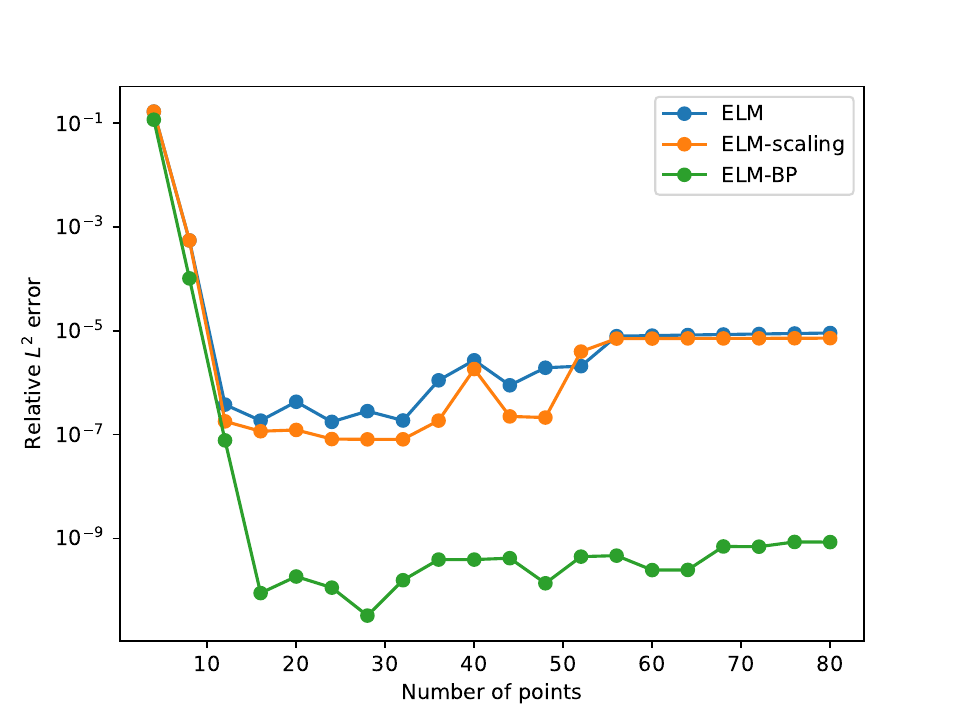}
		\caption{}
	\end{subfigure}
	\begin{subfigure}[t]{0.32\linewidth}
		\centering
		\includegraphics[width=1.9in]{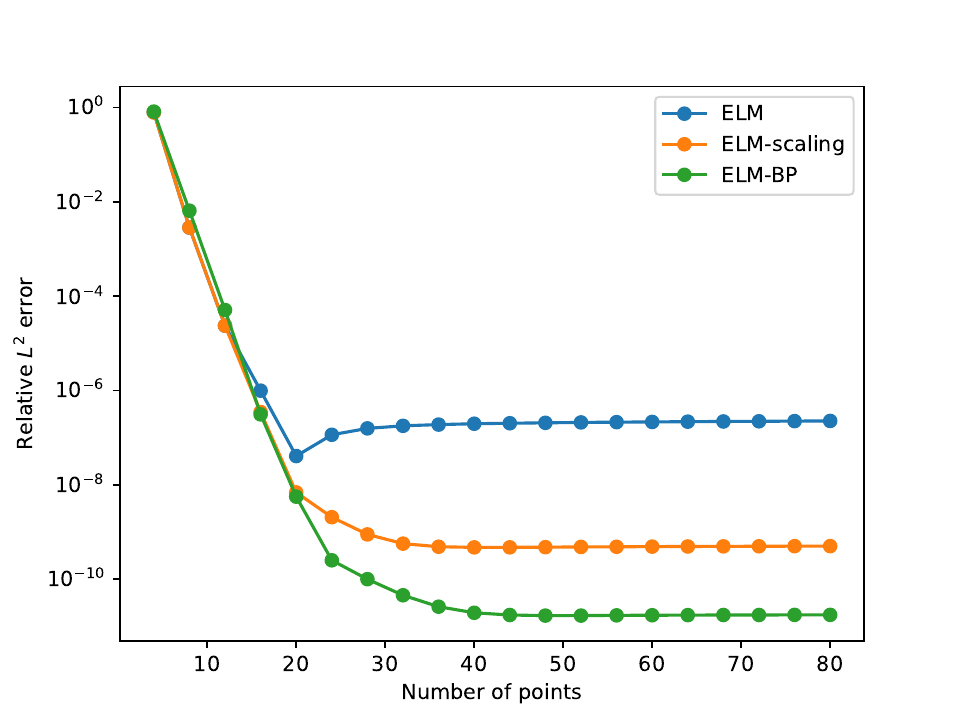}
		\caption{}
	\end{subfigure}
	\begin{subfigure}[t]{0.32\linewidth}
			\centering 
		\includegraphics[width=1.9in]{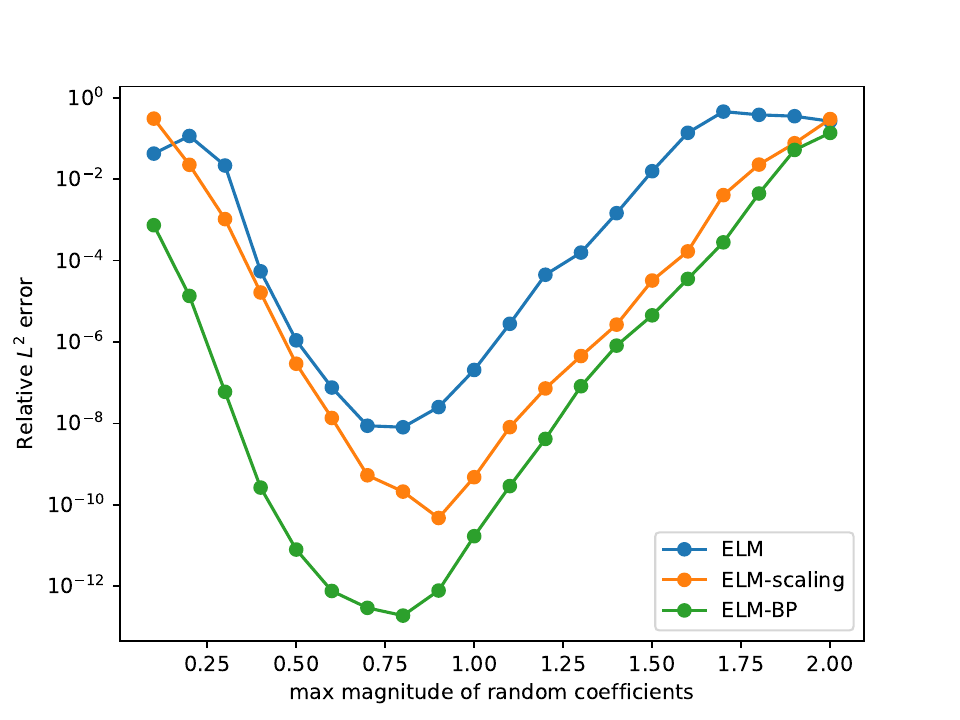}
		\caption{}  \label{exp1-Poisson-compare-Rm}
			\end{subfigure}
	\caption{Comparison of relative $L^2$ errors of three methods for Example 1.1: (a) Varying numbers of points with the default initialization. (b) Varying numbers of points with uniform random initialization ($R_m = 1$). (c) Varying max magnitude of random coefficients.}  \label{exp1-Poisson-compare-N}
\end{figure}

  Figs. \ref{exp1-Poisson-compare-N}(a)-(b) illustrate that the RNN-Scaling and RNN-BP methods tend to produce smaller $L^2$ errors than the RNN method, especially under uniform random initialization with $R_m=1$. This indicates that these two methods are more robust to initialization conditions and may perform better in practice. Under uniform random initialization, the performance of the  three methods is generally better than that under  the default initialization condition.

We compare the  three methods  across varying values of   $R_m$, where we  fix  $M=300$ and $N=48$, and vary  $R_m$ from 0.1 to 2. From  Fig. \ref{exp1-Poisson-compare-N}(c), it can be seen that as $R_m$ varies, the three methods exhibit a similar trend. Specifically, when $R_m$ is around 0.3 to 1.5, the errors of the three methods are relatively small; when $R_m>1.5$,  the errors increase rapidly. The RNN-BP method produces the smallest error among the three methods.

\subsubsection{Example 1.2}	
The second numerical example considers the exact solution on  $\Omega  = (0,2)^2$:
\footnotesize
\begin{equation*}
	u(x, y)=-\left[2 \cos \left(\frac{3}{2} \pi x+\frac{2 \pi}{5}\right)+\frac{3}{2} \cos \left(3 \pi x-\frac{\pi}{5}\right)\right]\left[2 \cos \left(\frac{3}{2} \pi y+\frac{2 \pi}{5}\right)+\frac{3}{2} \cos \left(3 \pi y-\frac{\pi}{5}\right)\right].
\end{equation*}
\normalsize

We compare  the three methods varying max  magnitude of random coefficient   with the fixed network  [2, 100, 500, 1]  and  $N=96$.  The result  is  presented  in Fig. \ref{exp2-Poisson}(a).  It is  observed that the optimal value of $R_m \approx 1.2$. The RNN-BP method achieves  the smallest error among the three methods, with the smallest error reaching 4.75$\times 10^{-10}$. We  compare the methods with varying values of  $M$, with fixed $R_m=1$  and $N=64$.  Fig. \ref{exp2-Poisson}(b) illustrates  that the errors of all three methods decrease  as $M$ increases, and the RNN-BP method remains the one with the smallest error. 

\begin{figure}[!htbp]
	\centering 
	\begin{subfigure}[t]{0.43\linewidth}
		\centering
		\includegraphics[width=2.5in]{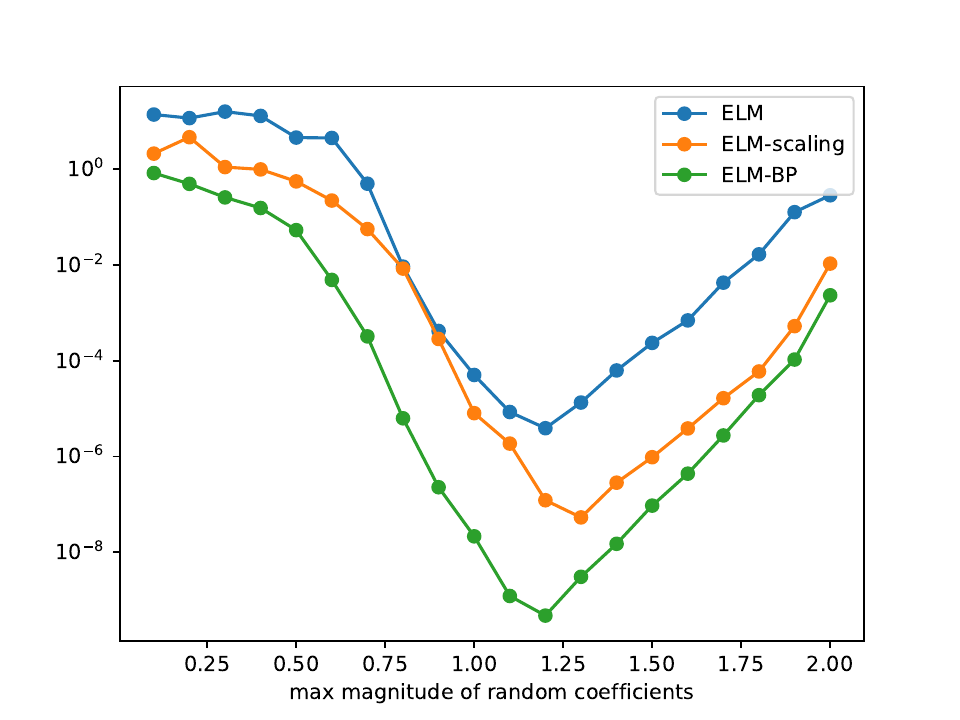}
		\caption{Varying max magnitude of random coefficient  with uniform random initialization.}
	\end{subfigure}
	\hspace{1.5cm} 
	\begin{subfigure}[t]{0.43\linewidth}
		\centering
		\includegraphics[width=2.5in]{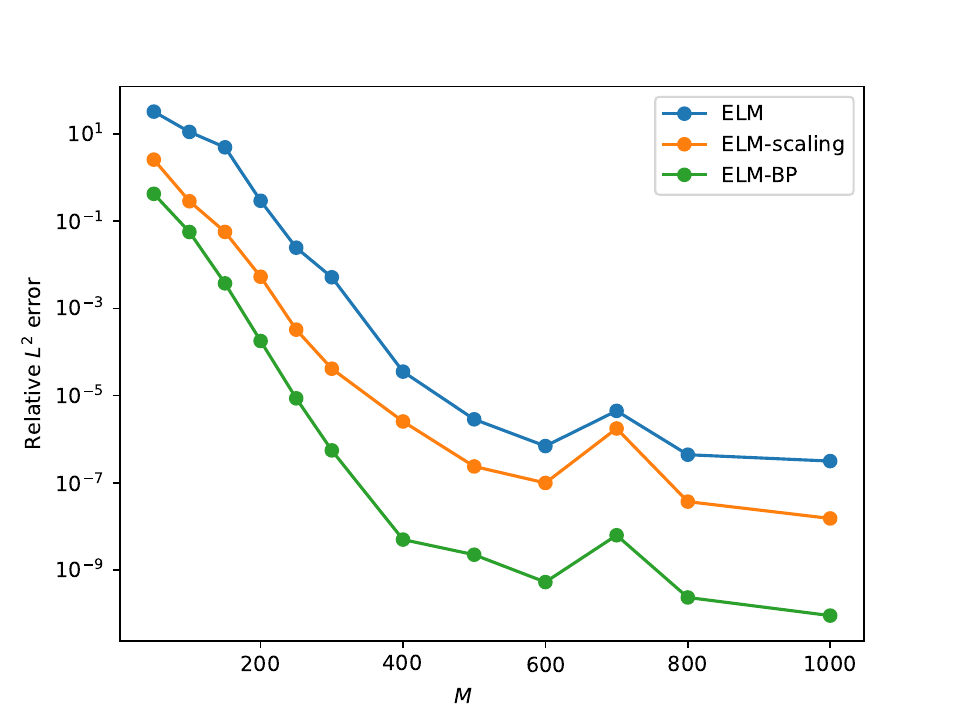}
		\caption{Varying    $M$  with uniform random initialization ($R_m=1$).}
	\end{subfigure}
	\caption{Comparison of relative $L^2$ errors of three methods for Example 1.2.}  \label{exp2-Poisson}
\end{figure}

\subsubsection{Example 1.3}
This example aims to  demonstrate  that the RNN-BP method is exact for the  solution  of the form \eqref{form-Poisson}. We consider the exact solution on $\Omega = (0,1)\times (0,1)$:
\begin{equation*}
	u(x,y) = x^{10} + y^{10} + x^k  \sin y + y^k\cos x,
\end{equation*}
where $k$ is an integer.

We compare the three methods with different  numbers of collocation points for $k=1$ and $k=2$, using a  fixed network [2,  100,  300,  1] and random initialization ($R_m=1$).   The results are presented   in Table \ref{exp3_Poisson}.  For the RNN-BP method, it can be observed  that the  relative $L^2$ errors achieve  machine precision when $k=1$.  This verifies that the RNN-BP method is exact for the  solution  of the form \eqref{form-Poisson}.  The errors decrease rapidly as the numbers
of collocation points increase  when $k = 2$,  with the smallest error about $10^{-14}$.  The relative $L^2$ errors of the RNN-Scaling method are at most 3 orders of magnitude smaller than those of the RNN method.

\begin{table}[htbp]
	\centering
	\scriptsize
	\caption{Comparison of three  methods  with different $N$ for Example 1.3.}
	\begin{tabular}{clccccccc}
		\toprule
		$k$   & $N$   & 4     & 12    & 20    & 28    & 36    & 44    & 52 \\
		\midrule
		1     & RNN   & 8.17e-1 & 1.41e-4 & 2.44e-7 & 6.32e-7 & 8.23e-7 & 9.00e-7 & 9.43e-7 \\
		& RNN-S & 8.17e-1 & 1.41e-4 & 3.55e-8 & 5.47e-9 & 3.67e-9 & 3.56e-9 & 3.60e-9 \\
		& RNN-BP & 2.35e-16 & 2.60e-16 & 2.57e-16 & 2.19e-16 & 2.19e-16 & 2.19e-16 & 2.19e-16 \\
		\midrule
		2     & RNN   & 1.05e+0 & 1.70e-4 & 2.92e-7 & 7.58e-7 & 9.87e-7 & 1.08e-6 & 1.13e-6 \\
		& RNN-S & 1.05e+0 & 1.70e-4 & 4.25e-8 & 6.56e-9 & 4.39e-9 & 4.26e-9 & 4.30e-9 \\
		& RNN-BP & 2.64e-2 & 1.15e-6 & 1.39e-11 & 8.14e-13 & 1.62e-13 & 9.64e-14 & 9.06e-14 \\
		\bottomrule
	\end{tabular}%
	\label{exp3_Poisson} 
\end{table}%

\subsubsection{Example 1.4}
This example   considers the exact solution  
\begin{equation*}
	u(x,y) = x^4 + y^4
\end{equation*}
on the unit circle centered at   $(0,0)$. 

We fix $N=64$  
and vary   $M$ to compare the   RNN method  and the RNN-Scaling method, where  the network architecture is [2, 100, 100, $M$, 1] and  the default initialization method is employed. Fig. \ref{exp4_Poisson}(a) shows an example of the  uniform  collocation points with
$N=32$.  Fig. \ref{exp4_Poisson}(b)  presents the comparison results, 
it can be seen that the RNN-Scaling method produces the smallest error (about $10^{-9}$), and typically results in a reduction of 2 orders of magnitude in error compared to the RNN method.
Figs. \ref{exp4_Poisson}(c)-(d) illustrate  the absolute errors of  both methods when 
$N=64$ and $M=200$, and it can be observed that the error of the RNN method is concentrated on the boundary, while the error of the RNN-Scaling method is concentrated in the interior.

\begin{figure}[!htbp]
	\centering 
	\begin{subfigure}[t]{0.2\linewidth}
		\centering
		\includegraphics[width=\linewidth]{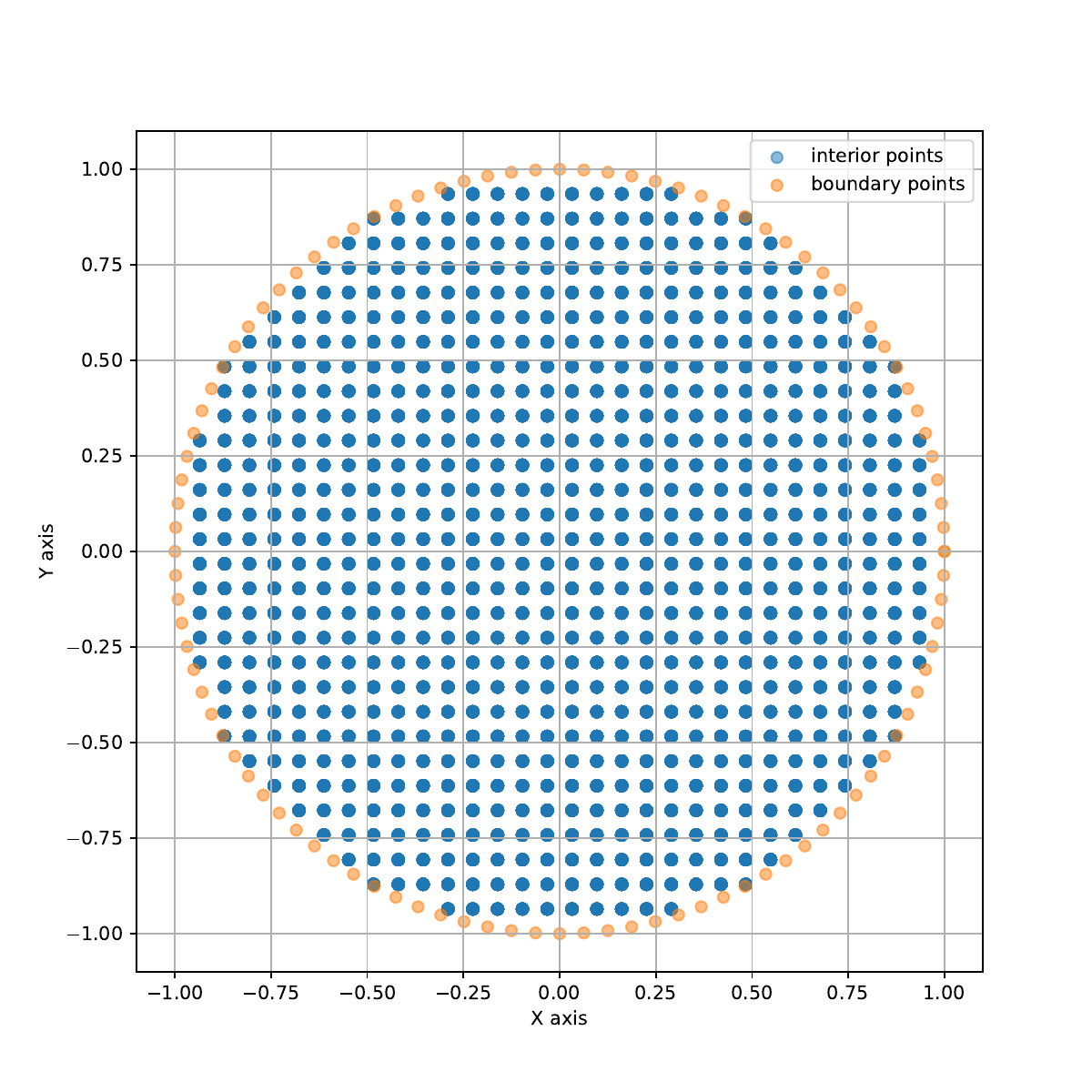}
		\caption{The uniform collocation points on the unit circle.}  
	\end{subfigure}
	\hfill
	\begin{subfigure}[t]{0.245\linewidth}
		\centering
		\includegraphics[width=\linewidth]{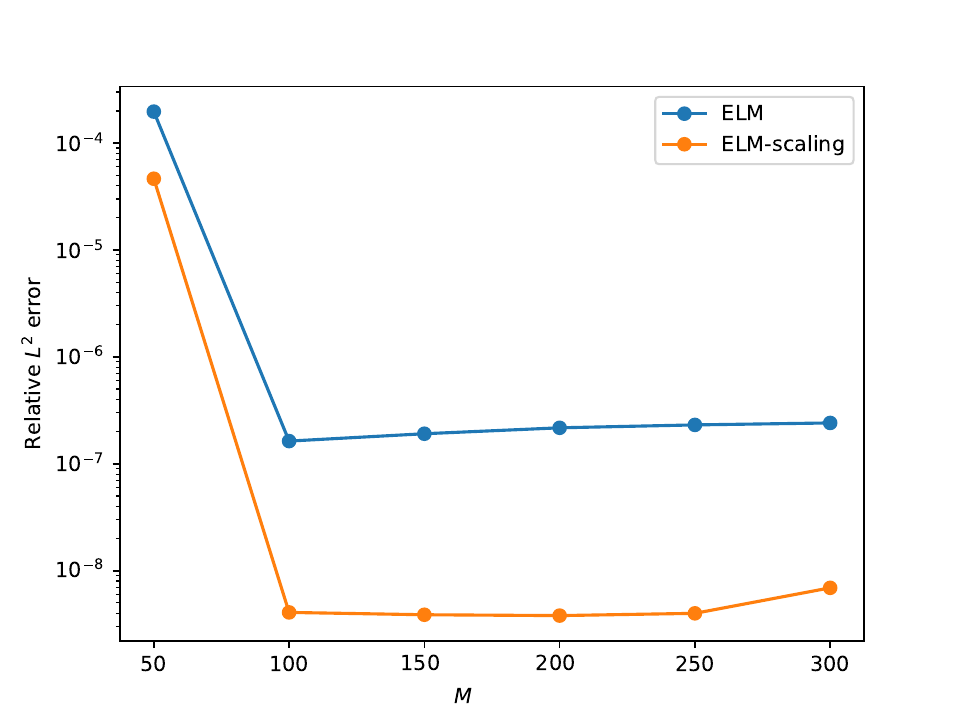}
		\caption{ Varying  $M$ with the default initialization.}  
	\end{subfigure}
	\hfill
	\begin{subfigure}[t]{0.25\linewidth}
		\centering
		\includegraphics[width=\linewidth]{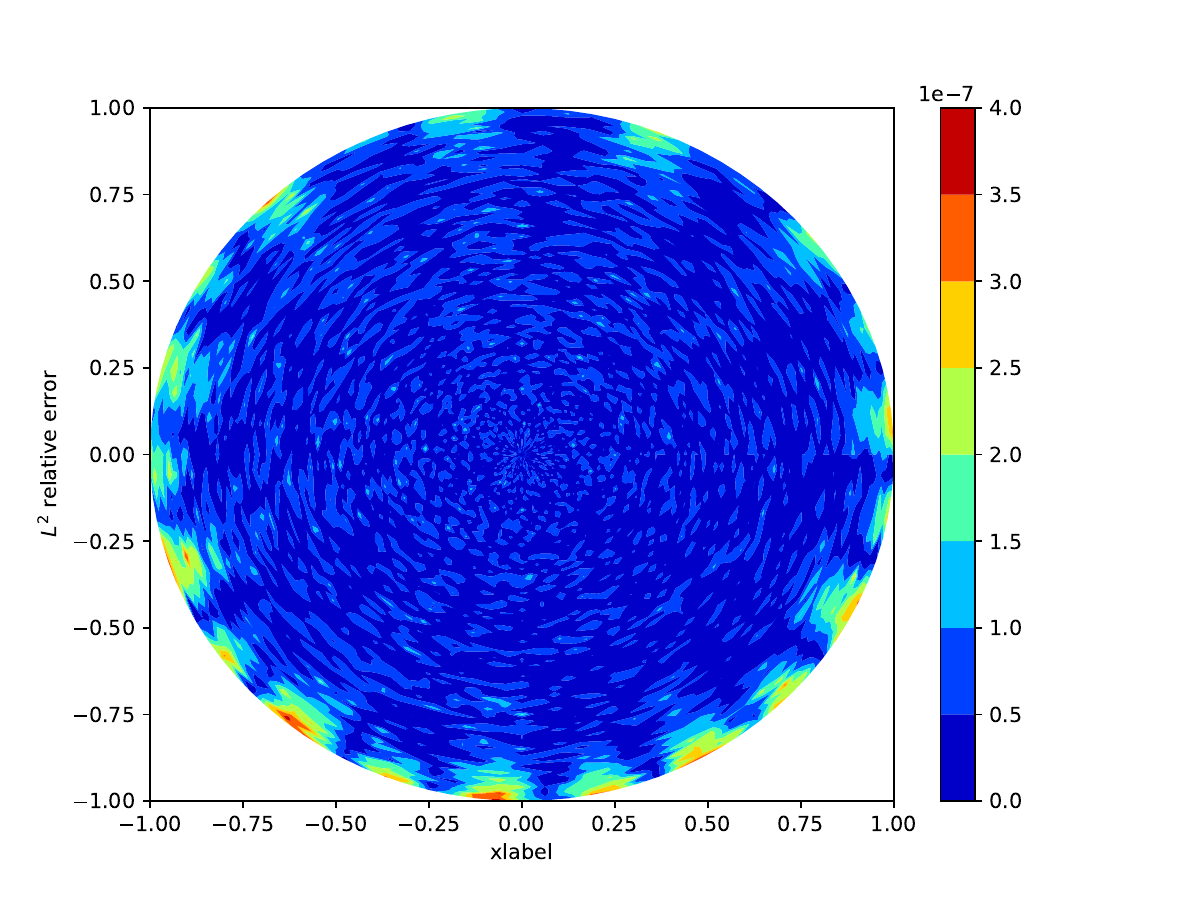}
		\caption{The absolute  error of the  RNN method.}  
	\end{subfigure}
	\hfill
	\begin{subfigure}[t]{0.25\linewidth}
		\centering
		\includegraphics[width=\linewidth]{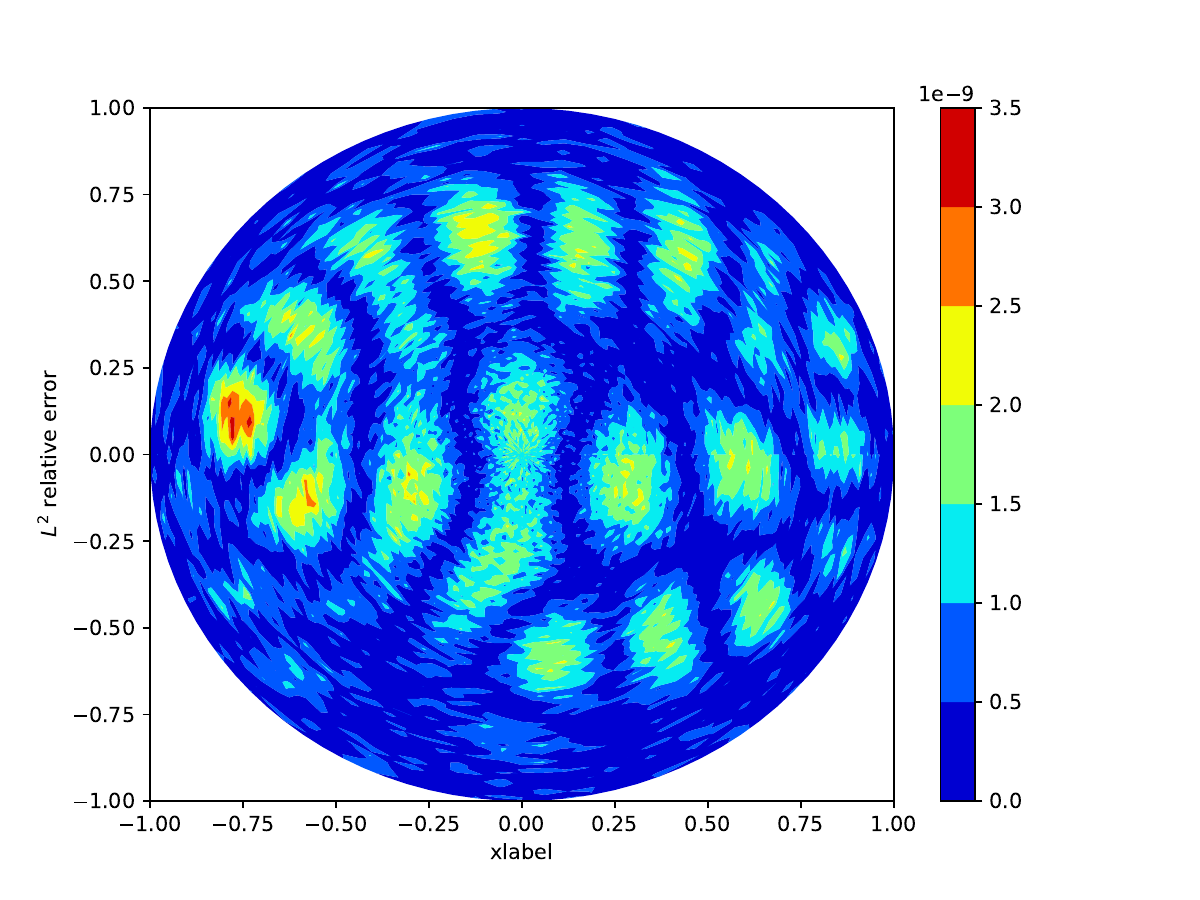}
		\caption{The absolute  error of the  RNN-Scaling method.}  
	\end{subfigure}
	\caption{The numerical results for Example 1.4.}\label{exp4_Poisson}
\end{figure}

From Examples 1.1-1.4, we offer the following summary for solving the Poisson equation.  The RNN-BP method consistently achieves the smallest errors among the three methods. 
 More specifically, 
 when $M$ and $N$ are relatively large, the errors of the  RNN-BP method   are, on average,  4 orders of magnitude smaller than those of the RNN method in Example 1.1,  and 3 orders of magnitude smaller  than those of the RNN method in Example 1.2.
Moreover, under both initialization methods, the errors of  the RNN-BP method are significantly smaller than those of  the other two methods, indicating that the RNN-BP method is more robust in practice.

On both the unit circular and square domains, the RNN-Scaling method is more accurate than the  RNN method.
 Specifically, on the unit circular domain,  the relative $L^2$ errors of the RNN-Scaling method are,  on average, 2 orders of magnitude smaller than those of the RNN method.  
On the square domain, under the default initialization method,  the errors of the RNN-Scaling method  are,  on average 2 orders of magnitude smaller than the RNN method when $M$ is small;  and the  difference between the two methods  is not significant when $M$ is large. 
Under the uniform random initialization, the errors are,  on average, 2 orders of magnitude smaller than those of the RNN method.  

\subsection{Biharmonic equation}
Next, we test our methods for the biharmonic equation.
\subsubsection{Example 2.1}
In the first example, we consider the  exact solution on $\Omega = (0,1)\times (0,1)$:
\begin{equation*}
	u(x,y) = \sin\pi x \sin\pi y.
\end{equation*}

  We employ a neural network    containing 2 hidden layers with 100 and $M$ neurons, respectively. 
We compare the relative $L^2$ errors of the  three  methods with different numbers  of   collocation points and  $M$.  Tables \ref{exp1-km}  and \ref{exp1-Rm1}  present the comparison results with  the  default   and uniform random initialization $R_m $=1, respectively.  
Tables  \ref{exp1-km}  and    \ref{exp1-Rm1}  demonstrate  that the errors of  all   three methods significantly decrease as 
$M$ and $N$ increase.   
  The RNN-BP method exhibits the best performance, followed by the RNN-Scaling method, and then the   RNN method.  Specifically, under  the default  initialization,  the RNN-Scaling method results in a reduction of 1 to 4 orders of magnitude in error,  and the RNN-BP  method results in a reduction of 6 to 9 orders of magnitude in error. The RNN-BP method achieves machine precision when $M\geq 150$ and $N\geq 16$. 
Under random initialization, the RNN-Scaling method achieves a reduction of up to 5 orders of magnitude in error, while the RNN-BP method results in a reduction of up to 7 orders of magnitude.
These  indicate that the RNN-BP and RNN-Scaling  methods have  a clear advantage over the  RNN method.
 Figs.  \ref{exp1-error}(a)-(c) display the  error distributions for the three methods with $M=300$ and $N=32$.  It can be seen that the error of the  RNN   is predominantly  concentrated around the boundary, while the RNN-Scaling method has effectively reduced the errors around  the boundary. Furthermore, the RNN-BP method is exact on the boundary, achieving  the most accurate results.

\begin{table}[htbp]
	\centering
	\scriptsize
	\caption{Comparison of three  methods    with the default  initialization    for Example 2.1.}
	\begin{tabular}{ccccccccc}
		\toprule
		\multicolumn{1}{l}{method} & $N$   & $M$   & 50    & 100   & 150   & 200   & 250   & 300 \\
		\midrule
		\multicolumn{1}{l}{RNN} & 8     & $\|e\|_{L^2}$ & 4.73e-2 & 1.65e-5 & 1.03e-8 & 4.13e-9 & 3.65e-9 & 4.40e-9 \\
		& 12    & $\|e\|_{L^2}$ & 5.48e-2 & 1.50e-4 & 3.62e-8 & 1.69e-10 & 5.20e-11 & 1.54e-10 \\
		& 16    & $\|e\|_{L^2}$ & 5.87e-2 & 1.78e-4 & 6.17e-8 & 1.20e-9 & 5.79e-10 & 1.20e-10 \\
		& 24    & $\|e\|_{L^2}$ & 6.48e-2 & 1.93e-4 & 7.35e-8 & 1.75e-9 & 1.92e-9 & 3.65e-10 \\
		& 32    & $\|e\|_{L^2}$ & 7.03e-2 & 2.01e-4 & 8.83e-8 & 1.98e-9 & 8.14e-10 & 1.96e-9 \\
		\midrule
		\multicolumn{1}{l}{RNN-S } & 8     & $\|e\|_{L^2}$ & 1.78e-4 & 1.26e-7 & 8.94e-9 & 4.00e-9 & 3.45e-9 & 4.29e-9 \\
		& 12    & $\|e\|_{L^2}$ & 7.22e-4 & 5.68e-8 & 3.19e-11 & 1.64e-11 & 7.06e-12 & 7.13e-12 \\
		& 16    & $\|e\|_{L^2}$ & 9.56e-4 & 3.78e-8 & 2.97e-11 & 5.76e-11 & 1.57e-11 & 6.29e-12 \\
		& 24    & $\|e\|_{L^2}$ & 7.57e-4 & 5.05e-8 & 5.16e-10 & 2.80e-10 & 1.29e-10 & 8.97e-11 \\
		& 32    & $\|e\|_{L^2}$ & 8.52e-4 & 6.57e-8 & 1.25e-9 & 1.39e-9 & 6.19e-10 & 8.57e-10 \\
		\midrule
		\multicolumn{1}{l}{RNN-BP} & 8     & $\|e\|_{L^2}$ & 3.37e-9 & 5.79e-10 & 2.80e-10 & 1.54e-10 & 1.05e-10 & 1.85e-10 \\
		& 12    & $\|e\|_{L^2}$ & 1.49e-8 & 8.37e-13 & 2.49e-14 & 8.27e-14 & 5.12e-14 & 3.00e-14 \\
		& 16    & $\|e\|_{L^2}$ & 1.20e-8 & 4.65e-13 & 9.27e-16 & 5.32e-16 & 2.67e-16 & 3.35e-16 \\
		& 24    & $\|e\|_{L^2}$ & 9.33e-9 & 5.44e-13 & 7.17e-16 & 1.15e-15 & 1.28e-16 & 2.37e-16 \\
		& 32    & $\|e\|_{L^2}$ & 8.18e-9 & 5.09e-13 & 6.73e-16 & 4.13e-16 & 1.87e-16 & 4.73e-16 \\
		\bottomrule
	\end{tabular}%
	\label{exp1-km} 
\end{table}%

\begin{table}[htbp]
	\centering
	\scriptsize
	\caption{Comparison of three   methods    with  uniform random  initialization ($R_m =1$)  for Example 2.1.}
	\begin{tabular}{ccccccccc}
		\toprule
		\multicolumn{1}{l}{method} & $N$   & $M$   & 50    & 100   & 150   & 200   & 250   & 300 \\
		\midrule
		\multicolumn{1}{l}{RNN} & 8     & $\|e\|_{L^2}$ & 4.96e+0 & 4.05e-3 & 6.38e-4 & 6.74e-4 & 3.30e-4 & 1.05e-3 \\
		& 12    & $\|e\|_{L^2}$ & 5.09e+0 & 4.27e-1 & 4.24e-2 & 4.19e-5 & 3.04e-7 & 4.32e-7 \\
		& 16    & $\|e\|_{L^2}$ & 5.05e+0 & 5.62e-1 & 4.28e-2 & 5.66e-3 & 1.28e-4 & 2.03e-6 \\
		& 24    & $\|e\|_{L^2}$ & 5.00e+0 & 6.33e-1 & 4.00e-2 & 7.98e-3 & 3.18e-4 & 3.30e-5 \\
		& 32    & $\|e\|_{L^2}$ & 4.98e+0 & 6.58e-1 & 4.08e-2 & 9.14e-3 & 3.94e-4 & 4.60e-5 \\
		\midrule
		\multicolumn{1}{l}{RNN-S } & 8     & $\|e\|_{L^2}$ & 8.15e-2 & 9.32e-4 & 6.38e-4 & 6.74e-4 & 3.30e-4 & 1.05e-3 \\
		& 12    & $\|e\|_{L^2}$ & 3.96e-2 & 3.00e-3 & 4.53e-5 & 2.79e-6 & 3.04e-7 & 4.32e-7 \\
		& 16    & $\|e\|_{L^2}$ & 4.36e-2 & 3.72e-3 & 3.21e-5 & 1.20e-6 & 2.97e-8 & 3.45e-9 \\
		& 24    & $\|e\|_{L^2}$ & 2.73e-2 & 3.50e-3 & 3.23e-5 & 7.27e-7 & 1.21e-8 & 7.52e-10 \\
		& 32    & $\|e\|_{L^2}$ & 1.01e-1 & 2.97e-3 & 5.86e-5 & 9.14e-7 & 2.13e-8 & 7.95e-10 \\
		\midrule
		\multicolumn{1}{l}{RNN-BP} & 8     & $\|e\|_{L^2}$ & 2.01e-4 & 2.35e-4 & 6.07e-5 & 3.25e-4 & 4.15e-4 & 3.09e-4 \\
		& 12    & $\|e\|_{L^2}$ & 6.49e-4 & 1.93e-6 & 1.03e-6 & 9.59e-7 & 2.86e-7 & 1.25e-6 \\
		& 16    & $\|e\|_{L^2}$ & 6.55e-4 & 1.17e-6 & 4.42e-8 & 8.24e-9 & 3.85e-9 & 2.34e-9 \\
		& 24    & $\|e\|_{L^2}$ & 6.32e-4 & 5.04e-7 & 6.13e-8 & 1.26e-9 & 1.82e-11 & 1.02e-11 \\
		& 32    & $\|e\|_{L^2}$ & 6.12e-4 & 3.85e-7 & 9.31e-8 & 1.54e-9 & 6.55e-11 & 1.11e-11 \\
		\bottomrule
	\end{tabular}%
	\label{exp1-Rm1}%
\end{table}%

\begin{figure}[!htbp]
	\centering 
	\begin{subfigure}[t]{0.31\linewidth}
		\centering
		\includegraphics[width=1.8in]{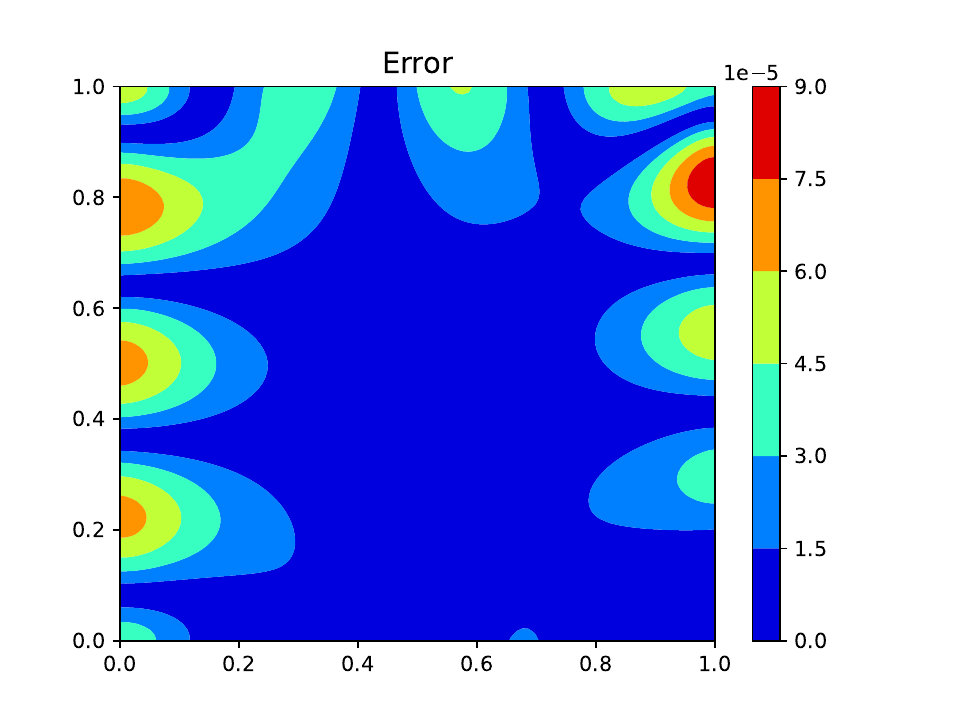}
		\caption{The absolute  error of the  RNN method.} \label{exp1-error-a}
	\end{subfigure}
	\hfill
	\begin{subfigure}[t]{0.31\linewidth}
		\centering
		\includegraphics[width=1.8in]{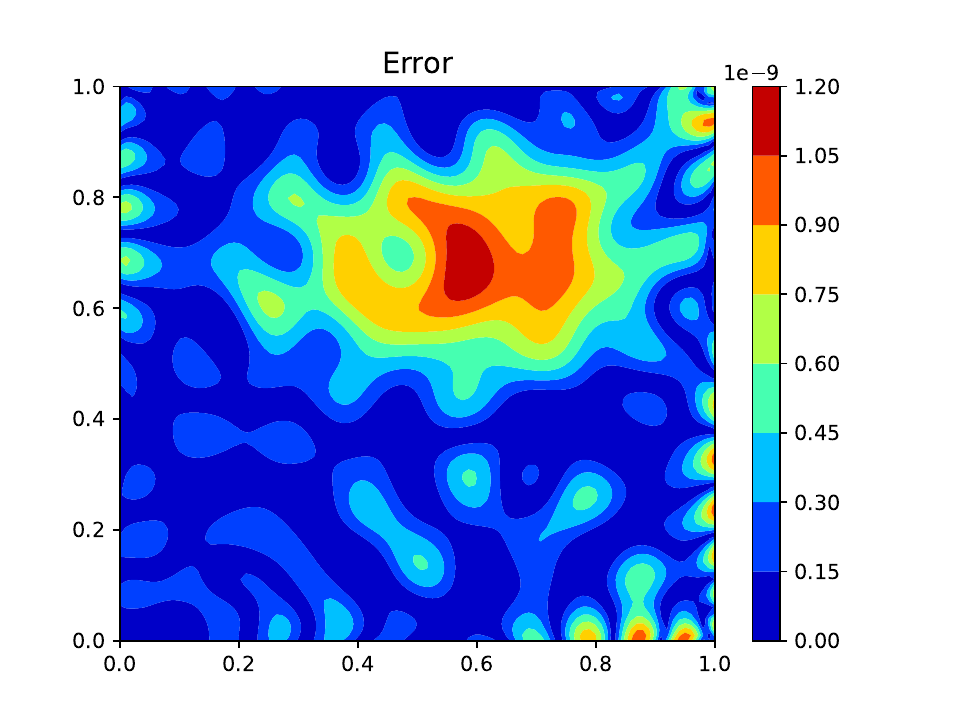}
		\caption{The absolute  error of the  RNN-Scaling method.} 
	\end{subfigure}
	\hfill
	\begin{subfigure}[t]{0.31\linewidth}
		\centering
		\includegraphics[width=1.8in]{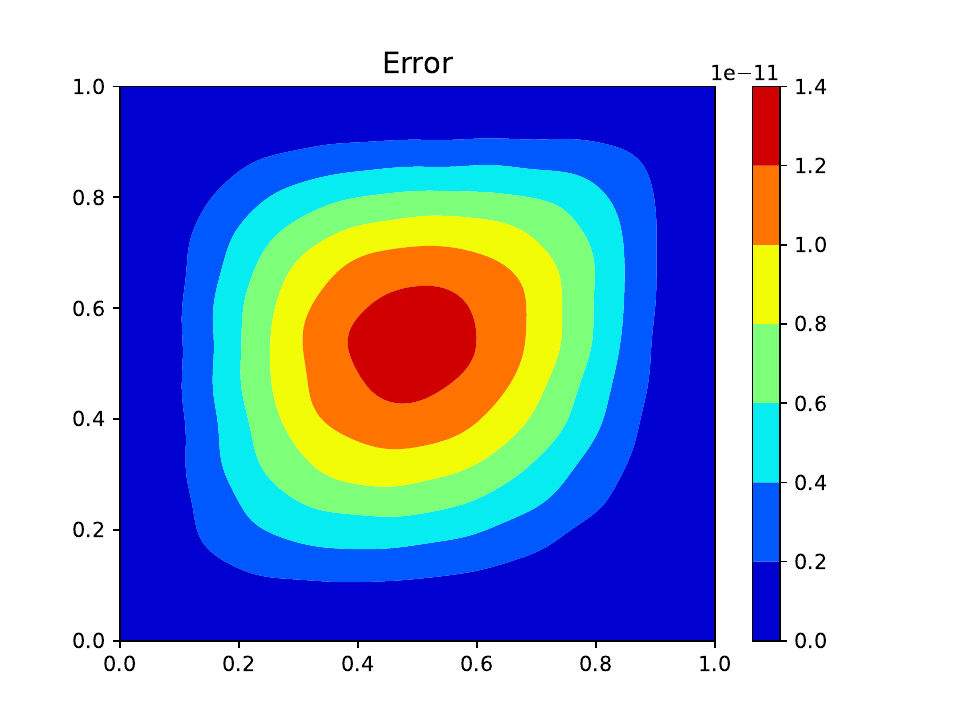}
		\caption{The absolute  error of the  RNN-BP method.} \label{exp1-error-c}
	\end{subfigure}
	\caption{The absolute  errors of three  methods  with uniform random  initialization    for Example 2.1. (The neural network  architecture is [2, 100, 300, 1] and $N=32$.)} \label{exp1-error}
\end{figure}

We  set the network architecture as [2, 100, 300, 1] and compare the performance of the three methods with different numbers of collocation points. Figs. \ref{exp1-compare-N-M}(a)-(b)  provide  the comparison results with the default   and random  initialization ($R_m=1$), respectively. 
It can be seen that regardless of the initialization method, the RNN-BP method consistently achieves  smaller errors than the other two methods. Under  the default initialization, the RNN-Scaling method performs slightly better than the RNN method.  Under  uniform random initialization, the RNN-Scaling method significantly outperforms the RNN method.

We  fix the number  of collocation points and compare the  three methods with different   $M$. Figs. \ref{exp1-compare-N-M}(c)-(d)  illustrate the comparison results with the default initialization and uniform random initialization, respectively.  The RNN-BP method outperforms the other two methods in terms of accuracy under both initialization methods.  Under  the default initialization, the RNN-Scaling method achieves   smaller errors  than the RNN method when $M\leq 150$;  and the errors of both methods become similar when 
$M>150$.
Under  uniform random initialization, the RNN-Scaling method consistently produces  smaller errors  than the RNN method, with a reduction of up to 5 orders of magnitude.

We compare the  three methods  with different  $R_m$, where we  fix  $M=300$ and $N=32$, and vary  $R_m$ from 0.1 to 2.   Fig. \ref{exp1-compare-Rm} shows  that as $R_m$ varies, the three methods exhibit a similar trend. Specifically, when $R_m$ is around 0.2 to 1.1, the errors of the three methods are relatively small. When $R_m>1.1$,  the errors increase rapidly. The RNN-BP method still  produces the smallest error.

\begin{figure}[!htbp]
	\centering 
		\hspace{-12mm}
	\begin{subfigure}[t]{0.40\linewidth}
		\centering
		\includegraphics[width=2.8in]{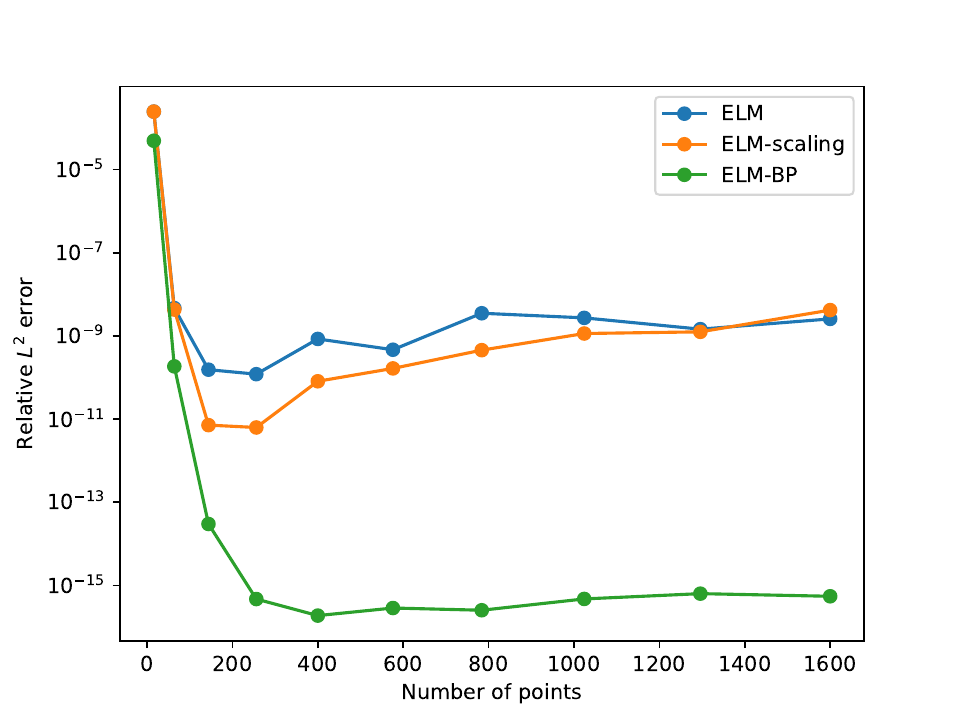}
		\caption{}
	\end{subfigure}
	\hspace{8mm}
	\begin{subfigure}[t]{0.40\linewidth}
		\centering
		\includegraphics[width=2.8in]{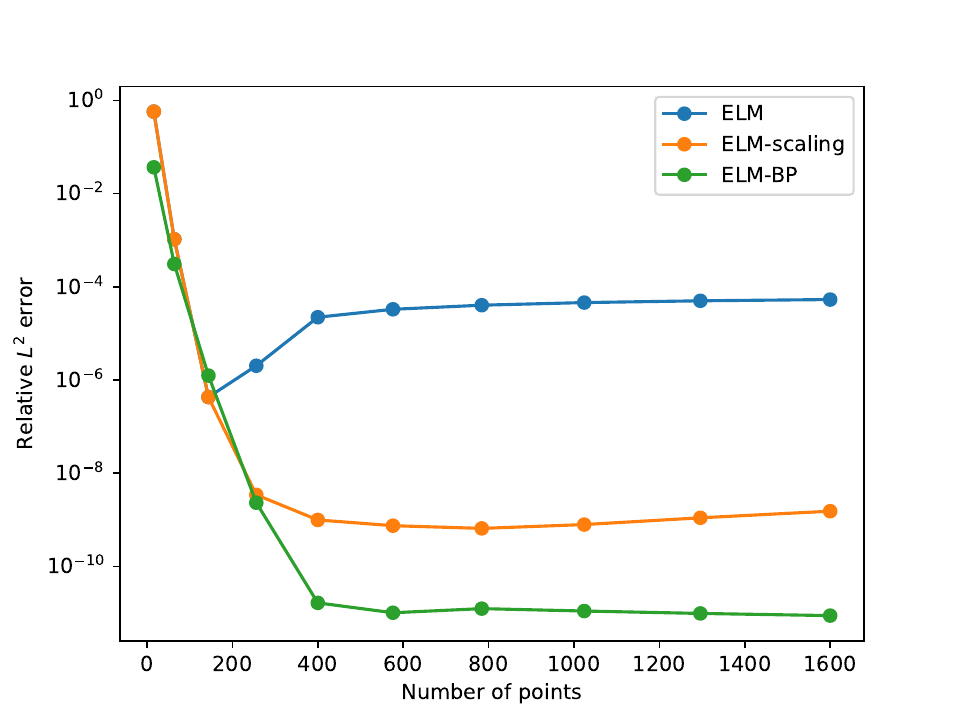}
		\caption{}
	\end{subfigure}
	\\
		\hspace{-12mm}
	\begin{subfigure}[t]{0.40\linewidth}
		\centering
		\includegraphics[width=2.8in]{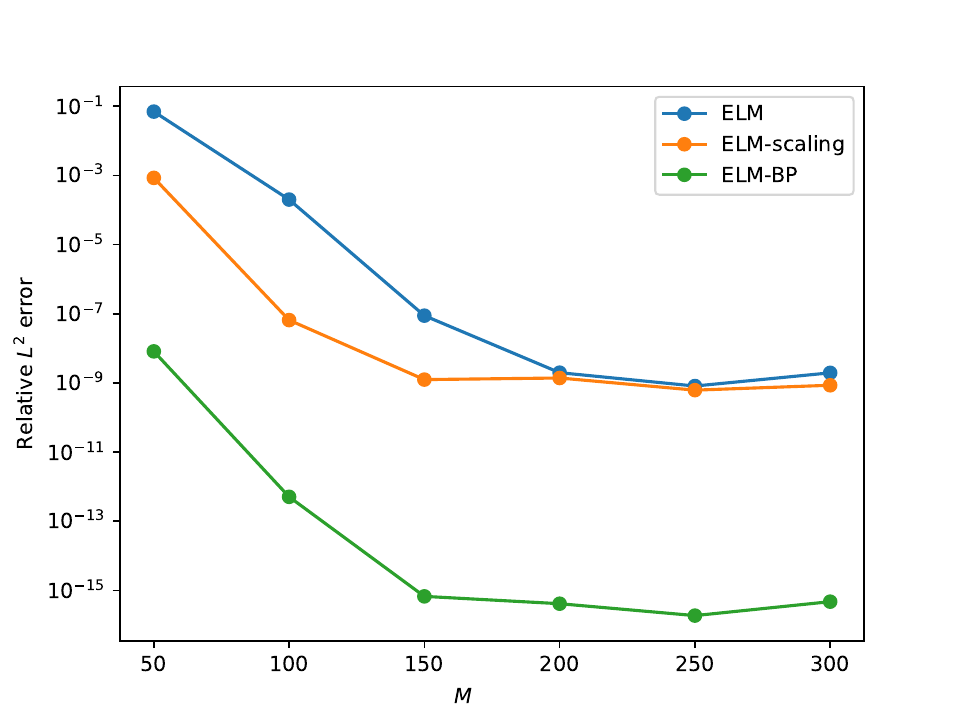}
		\caption{}
	\end{subfigure}
	\hspace{8mm}
	\begin{subfigure}[t]{0.4\linewidth}
		\centering
		\includegraphics[width=2.8in]{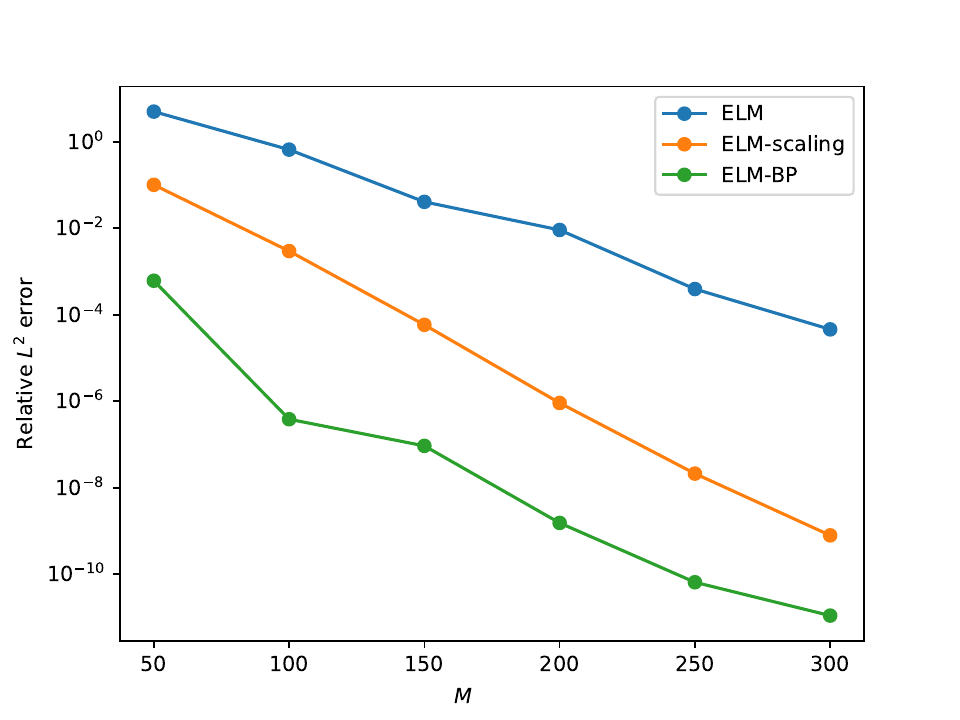}
		\caption{}
	\end{subfigure}
	\caption{Comparison of three methods for Example 2.1: (a) Varying numbers of points with the default initialization. (b) Varying numbers of points with uniform random initialization ($R_m = 1$).  (c) Varying  $M$ with the  default initialization.
		(d) Varying   $M$ with uniform random initialization ($R_m = 1$). }  \label{exp1-compare-N-M}
\end{figure}

\begin{figure}[!htbp]
	\centering 
	\includegraphics[width=3in]{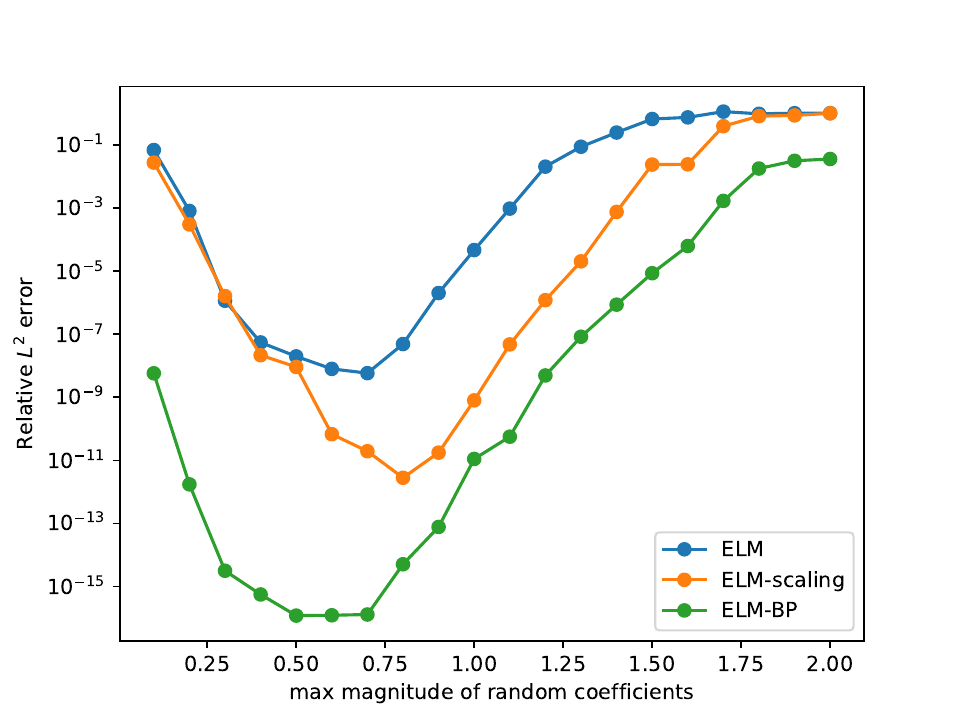}
	\caption{Comparison of relative $L^2$ errors of three methods varying max magnitude of random coefficients for Example 2.1.}\label{exp1-compare-Rm}
\end{figure}

\subsubsection{Example 2.2}
In this example, we consider the exact solution on $\Omega = (0,1)\times (0,1)$:
\begin{equation*}
	u(x,y) = \sin(2\pi x)\sin (2\pi y).
\end{equation*}

The numerical performance of  this example is generally similar to that of Example 2.1, so we will provide a brief analysis  of   the numerical results. Tables \ref{exp2-km} and    \ref{exp2-Rm1} compare the three  methods with different numbers of collocation points and  $M$, and initialization methods. It can be seen that the error with the default initialization is smaller, with up to 5 orders of magnitude reduction. Under both initialization methods, the RNN-BP method produces  the smallest  errors.  Under  the uniform random initialization with ($R_m=1$), the RNN-Scaling method is   more accurate than the RNN method,  with up to a 5 orders of magnitude reduction in error. 
Figs. \ref{exp2-error}(a)-(c) plot  absolute   errors  of the three methods, which shows that the RNN-Scaling method reduces the error around significantly and the  RNN-BP method is exact on the boundary.

\begin{table}[htbp]
	\centering 
	\scriptsize
	\caption{Comparison of three   methods   with the default  initialization    for Example 2.2.}
	\begin{tabular}{ccccccccc}
		\toprule
		\multicolumn{1}{l}{method} & $N$   & $M$   & 50    & 100   & 150   & 200   & 250   & 300 \\
		\midrule
		\multicolumn{1}{l}{RNN} & 8     & $\|e\|_{L^2}$ & 1.29e+2 & 4.61e-2 & 4.77e-5 & 1.16e-5 & 1.22e-5 & 1.80e-5 \\
		& 12    & $\|e\|_{L^2}$ & 1.15e+2 & 1.12e-1 & 6.48e-4 & 7.91e-7 & 5.14e-7 & 4.54e-7 \\
		& 16    & $\|e\|_{L^2}$ & 1.12e+2 & 2.16e-1 & 9.74e-4 & 8.81e-6 & 9.12e-7 & 6.00e-6 \\
		& 24    & $\|e\|_{L^2}$ & 1.13e+2 & 3.42e-1 & 1.17e-3 & 1.44e-5 & 5.13e-6 & 5.87e-6 \\
		& 32    & $\|e\|_{L^2}$ & 1.15e+2 & 4.18e-1 & 1.25e-3 & 1.65e-5 & 6.47e-6 & 2.41e-5 \\
		\midrule
		\multicolumn{1}{l}{RNN-S } & 8     & $\|e\|_{L^2}$ & 2.06e-1 & 1.21e-4 & 1.41e-5 & 8.28e-6 & 9.92e-6 & 6.01e-6 \\
		& 12    & $\|e\|_{L^2}$ & 1.76e-1 & 6.80e-5 & 3.22e-7 & 2.20e-7 & 4.97e-8 & 5.18e-8 \\
		& 16    & $\|e\|_{L^2}$ & 1.67e-1 & 5.36e-5 & 3.32e-7 & 1.29e-7 & 1.53e-7 & 7.00e-8 \\
		& 24    & $\|e\|_{L^2}$ & 1.82e-1 & 1.14e-4 & 8.63e-7 & 8.79e-7 & 6.97e-7 & 7.03e-8 \\
		& 32    & $\|e\|_{L^2}$ & 2.19e-1 & 1.50e-4 & 2.55e-6 & 3.17e-6 & 4.91e-6 & 7.05e-7 \\
		\midrule
		\multicolumn{1}{l}{RNN-BP} & 8     & $\|e\|_{L^2}$ & 8.12e-6 & 2.19e-6 & 7.74e-7 & 3.76e-7 & 3.25e-7 & 3.75e-7 \\
		& 12    & $\|e\|_{L^2}$ & 1.47e-5 & 5.28e-9 & 3.22e-10 & 2.11e-10 & 1.64e-10 & 9.05e-11 \\
		& 16    & $\|e\|_{L^2}$ & 1.37e-5 & 1.12e-9 & 8.14e-12 & 1.67e-12 & 5.67e-13 & 1.18e-12 \\
		& 24    & $\|e\|_{L^2}$ & 1.32e-5 & 2.05e-9 & 6.61e-12 & 1.58e-12 & 9.17e-13 & 1.37e-12 \\
		& 32    & $\|e\|_{L^2}$ & 1.30e-5 & 2.58e-9 & 1.04e-11 & 1.04e-12 & 7.39e-13 & 1.94e-12 \\
		\bottomrule
	\end{tabular}%
	\label{exp2-km}%
\end{table}%

\begin{table}[htbp]
	\centering
	\scriptsize
	\caption{Comparison of three  methods   with uniform random  initialization  ($R_m =1$)  for Example 2.2.}
	\begin{tabular}{ccccccccc}
		\toprule
		\multicolumn{1}{l}{method} & $N$   & $M$   & 50    & 100   & 150   & 200   & 250   & 300 \\
		\midrule
		\multicolumn{1}{l}{RNN} & 8     & $\|e\|_{L^2}$ & 2.09e+1 & 1.87e-2 & 2.48e-3 & 2.26e-3 & 1.31e-3 & 3.17e-3 \\
		& 12    & $\|e\|_{L^2}$ & 2.09e+1 & 1.87e-2 & 2.48e-3 & 2.26e-3 & 1.31e-3 & 3.17e-3 \\
		& 16    & $\|e\|_{L^2}$ & 2.09e+1 & 1.87e-2 & 2.48e-3 & 2.26e-3 & 1.31e-3 & 3.17e-3 \\
		& 24    & $\|e\|_{L^2}$ & 1.69e+1 & 1.31e+1 & 3.35e-1 & 4.36e-2 & 2.59e-3 & 1.55e-4 \\
		& 32    & $\|e\|_{L^2}$ & 1.66e+1 & 1.41e+1 & 3.64e-1 & 5.45e-2 & 3.10e-3 & 2.61e-4 \\
		\midrule
		\multicolumn{1}{l}{RNN-S } & 8     & $\|e\|_{L^2}$ & 3.03e-1 & 7.73e-3 & 2.48e-3 & 2.26e-3 & 1.31e-3 & 3.17e-3 \\
		& 12    & $\|e\|_{L^2}$ & 2.02e-1 & 7.74e-3 & 2.24e-4 & 5.19e-6 & 2.82e-6 & 1.27e-6 \\
		& 16    & $\|e\|_{L^2}$ & 1.69e-1 & 1.40e-2 & 1.65e-4 & 4.60e-6 & 1.63e-7 & 1.94e-8 \\
		& 24    & $\|e\|_{L^2}$ & 2.65e-1 & 1.54e-2 & 2.48e-4 & 6.65e-6 & 1.20e-7 & 4.32e-9 \\
		& 32    & $\|e\|_{L^2}$ & 5.11e-1 & 1.66e-2 & 3.15e-4 & 8.86e-6 & 2.30e-7 & 6.33e-9 \\
		\midrule
		\multicolumn{1}{l}{RNN-BP} & 8     & $\|e\|_{L^2}$ & 2.03e-2 & 6.73e-3 & 4.88e-3 & 5.21e-3 & 4.64e-3 & 1.17e-2 \\
		& 12    & $\|e\|_{L^2}$ & 5.73e-2 & 2.70e-4 & 2.38e-5 & 1.57e-5 & 1.42e-5 & 8.63e-6 \\
		& 16    & $\|e\|_{L^2}$ & 5.25e-2 & 3.07e-4 & 3.23e-6 & 4.44e-7 & 4.19e-7 & 9.10e-8 \\
		& 24    & $\|e\|_{L^2}$ & 4.77e-2 & 3.14e-4 & 7.98e-6 & 9.57e-8 & 7.94e-9 & 3.47e-10 \\
		& 32    & $\|e\|_{L^2}$ & 4.53e-2 & 3.26e-4 & 9.20e-6 & 8.52e-8 & 6.61e-9 & 4.68e-10 \\
		\bottomrule
	\end{tabular}%
	\label{exp2-Rm1}%
\end{table}%

\begin{figure}[!htbp]
	\centering 
	\begin{subfigure}[t]{0.3\linewidth}
		\centering
		\includegraphics[width=1.9in]{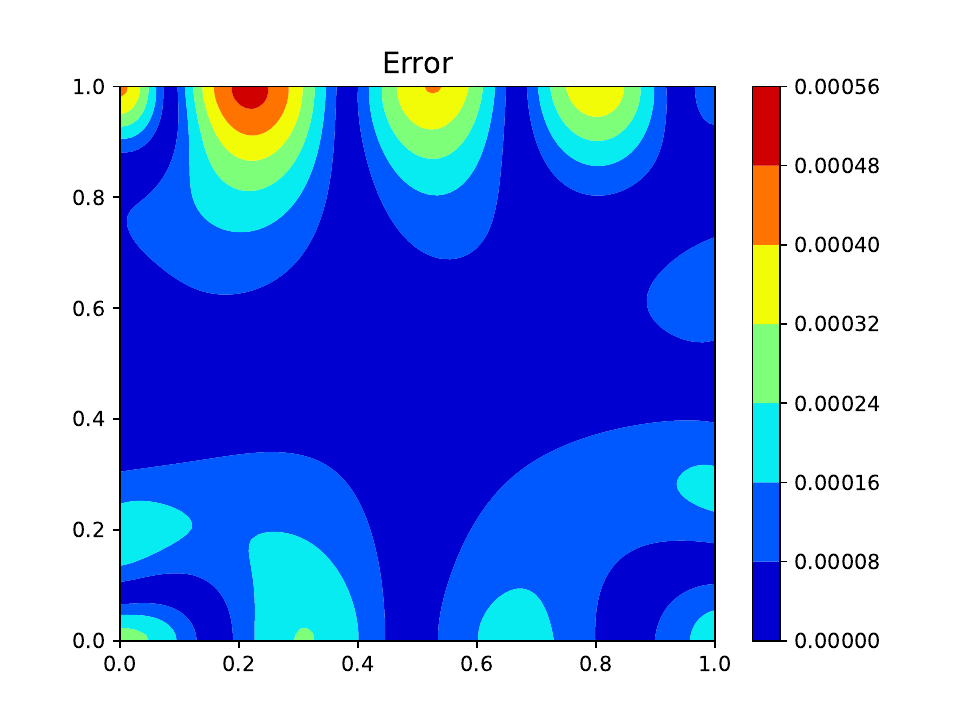}
		\caption{The absolute  error of RNN.}  
	\end{subfigure}
	\hfill
	\begin{subfigure}[t]{0.3\linewidth}
		\centering
		\includegraphics[width=1.9in]{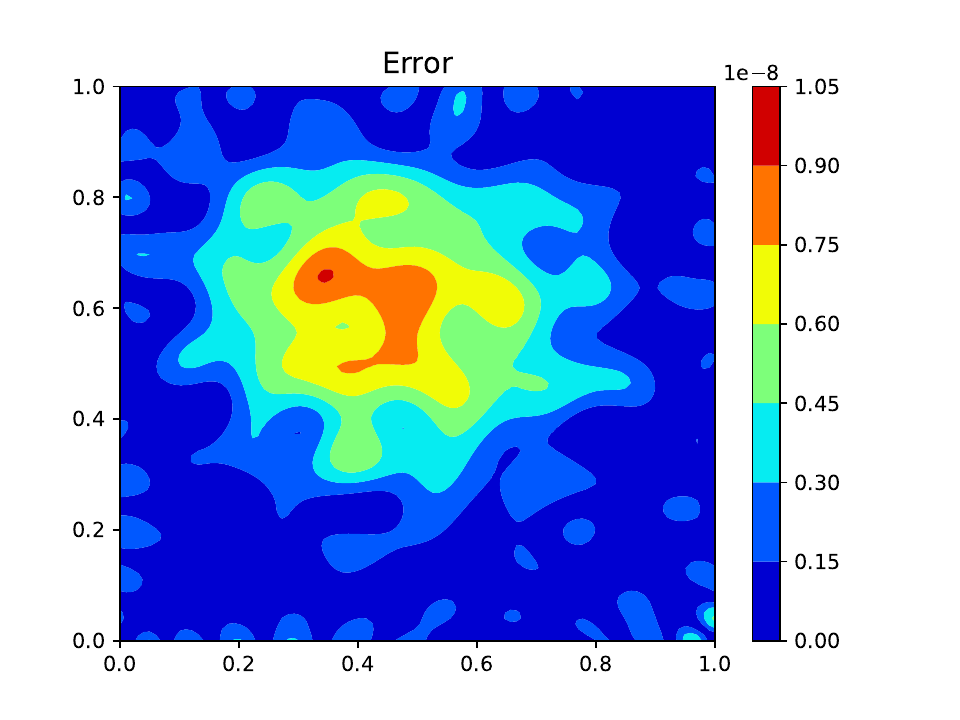}
		\caption{The absolute  error of RNN-Scaling.}  
	\end{subfigure}
	\hfill
	\begin{subfigure}[t]{0.3\linewidth}
		\centering
		\includegraphics[width=1.9in]{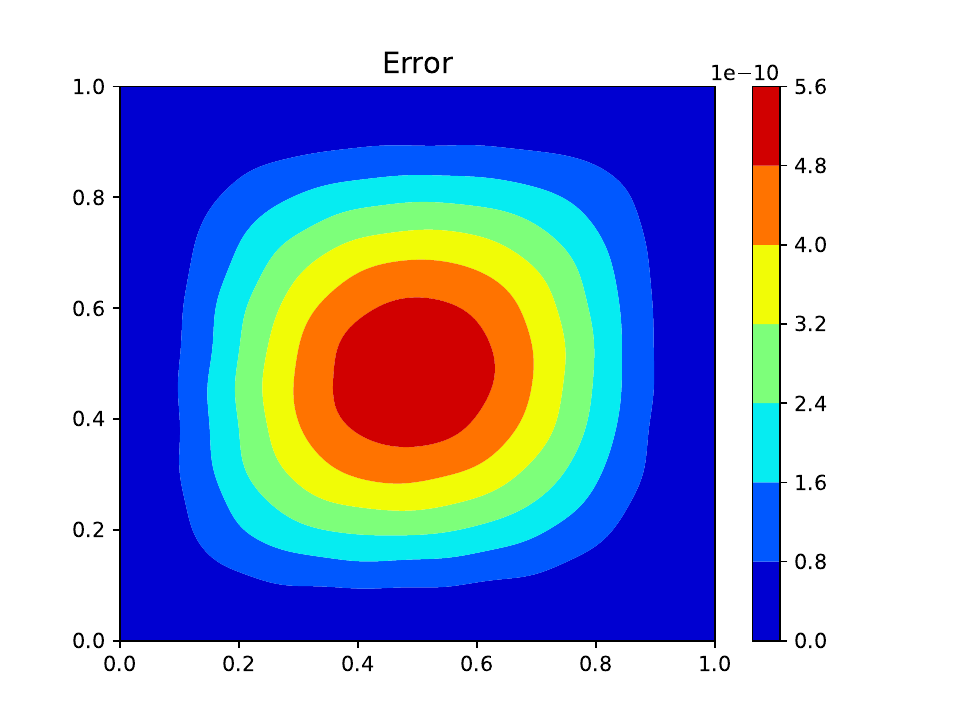}
		\caption{The absolute  error of RNN-BP.}  
	\end{subfigure}
	\caption{The absolute  errors of three  methods  with uniform random  initialization    for example 2.2. (The neural network  architecture is [2, 100, 300, 1] and $N=32$.)}  \label{exp2-error}
\end{figure}

Figs. \ref{exp2-compare-N-M}(a)-(b) compare  the three methods with different numbers of  collocation points, with Fig. \ref{exp2-compare-N-M}(a) employing the  default initialization and Fig. \ref{exp2-compare-N-M}(b)  employing  uniform random initialization and with the network fixed as [2, 100, 300, 1].  As the numbers of collocation points increase, the errors of the three methods decrease rapidly. Figs. \ref{exp2-compare-N-M}(c)-(d)  compare the three methods for different $M$, with the network fixed as [2, 100, $M$, 1] and $N$ fixed as 32.  As $M$  increases, the errors of the three methods also decrease rapidly.
 From Figs. \ref{exp2-compare-N-M}(a)-(d), it can be observed  that the RNN-BP method consistently produces the smallest error,  and  the RNN-Scaling method is more accurate than the RNN method  under  uniform random initialization ($R_m=1$).

Fig. \ref{exp2-compare-Rm} compares the  three methods for different values of $R_m$, with the  fixed network   [2, 100, 300, 1] and $N=32$. The three methods achieve smaller errors when $R_m$ is around 0.2 to 1.1.
The smallest errors of the RNN-BP, RNN-Scaling, and RNN methods are magnitudes of  $10^{-12}$, $10^{-10}$,  and $10^{-6}$, respectively, indicating the effectiveness of the  RNN-BP and RNN-Scaling methods.

\begin{figure}[!htbp]
	\centering 
	\hspace{-12mm}
	 \begin{subfigure}[t]{0.40\linewidth}
		\centering
		\includegraphics[width=2.8in]{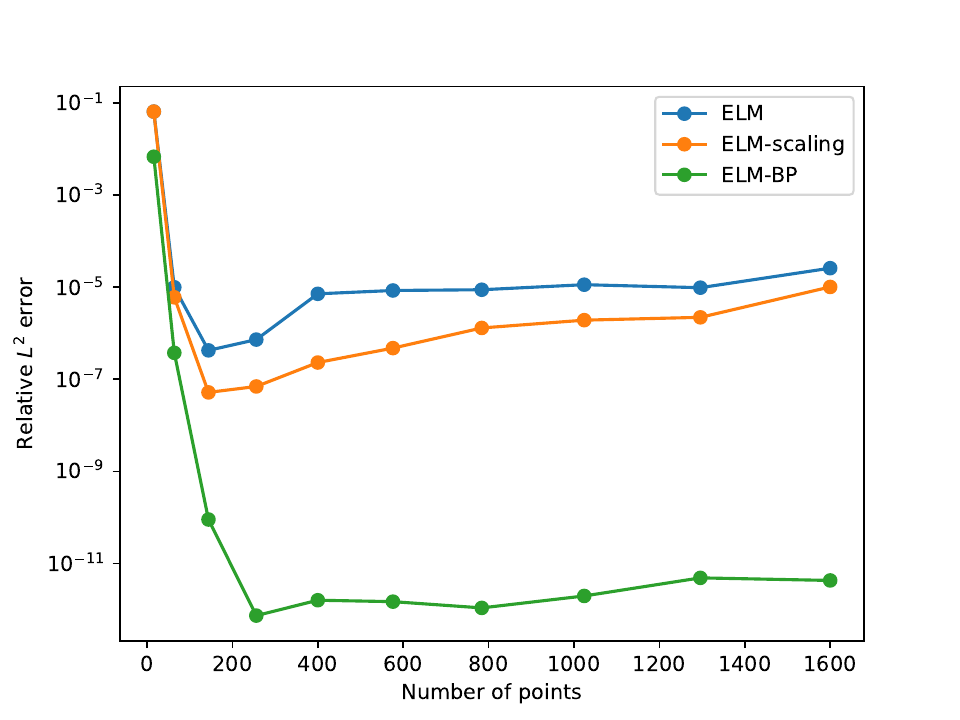}
		\caption{}
	\end{subfigure}
	\hspace{8mm}
	\begin{subfigure}[t]{0.40\linewidth}
		\centering
		\includegraphics[width=2.8in]{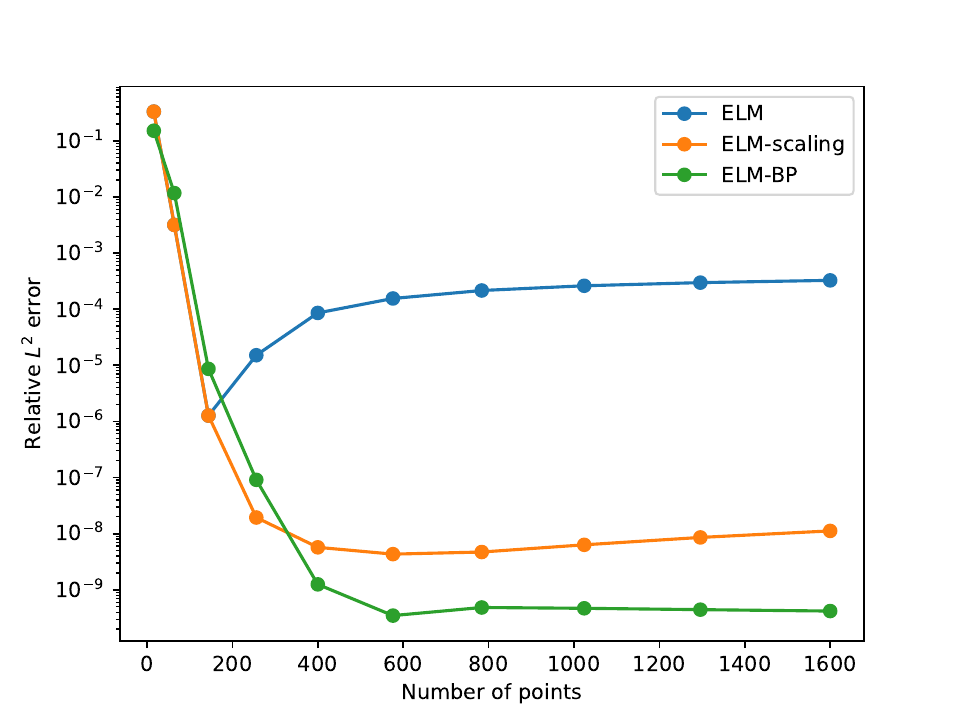}
		\caption{}
	\end{subfigure} 
	\\
	\hspace{-12mm}
	\begin{subfigure}[t]{0.40\linewidth}
		\centering
		\includegraphics[width=2.8in]{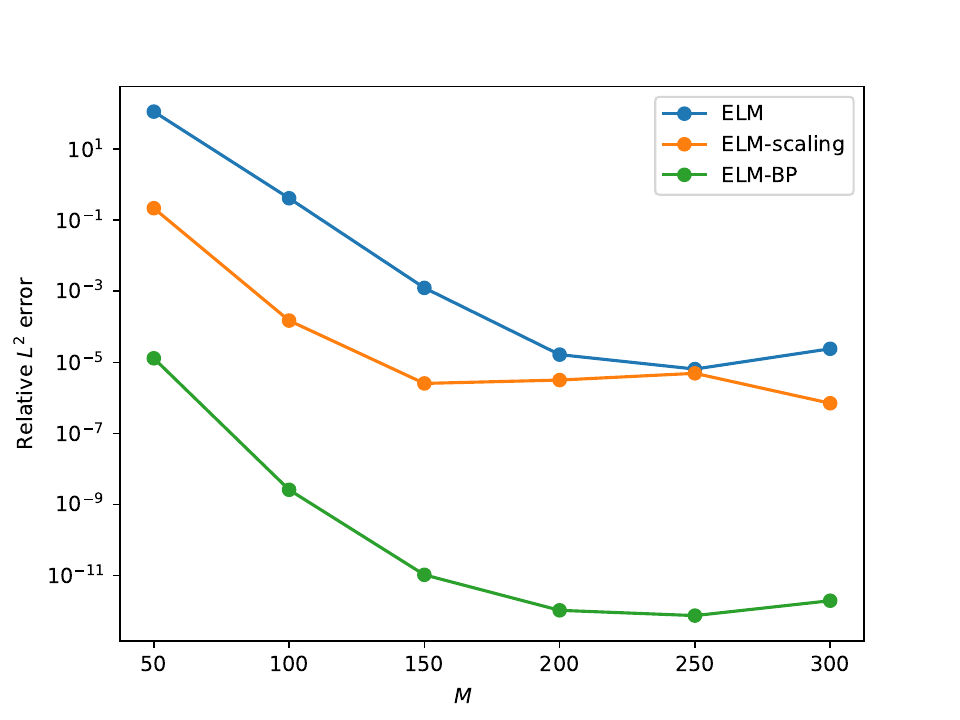}
		\caption{}
	\end{subfigure}
	\hspace{8mm}
	\begin{subfigure}[t]{0.4\linewidth}
		\centering
		\includegraphics[width=2.8in]{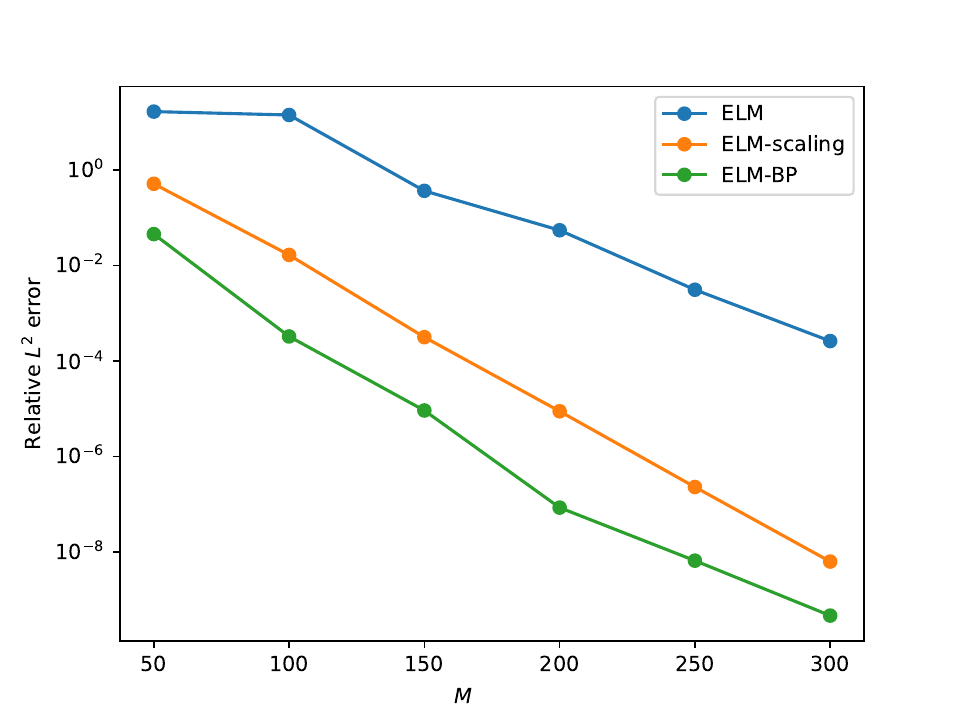}
		\caption{}
	\end{subfigure}
	\caption{Comparison of three methods for Example 2.2: (a) Varying numbers of points with the default initialization. (b) Varying numbers of points with uniform random initialization ($R_m = 1$).  (c) Varying   $M$ with the  default initialization.
		(d) Varying   $M$ with uniform random initialization ($R_m = 1$). }  \label{exp2-compare-N-M}
\end{figure}

\begin{figure}[!htbp]
	\centering 
	\includegraphics[width=2.6in]{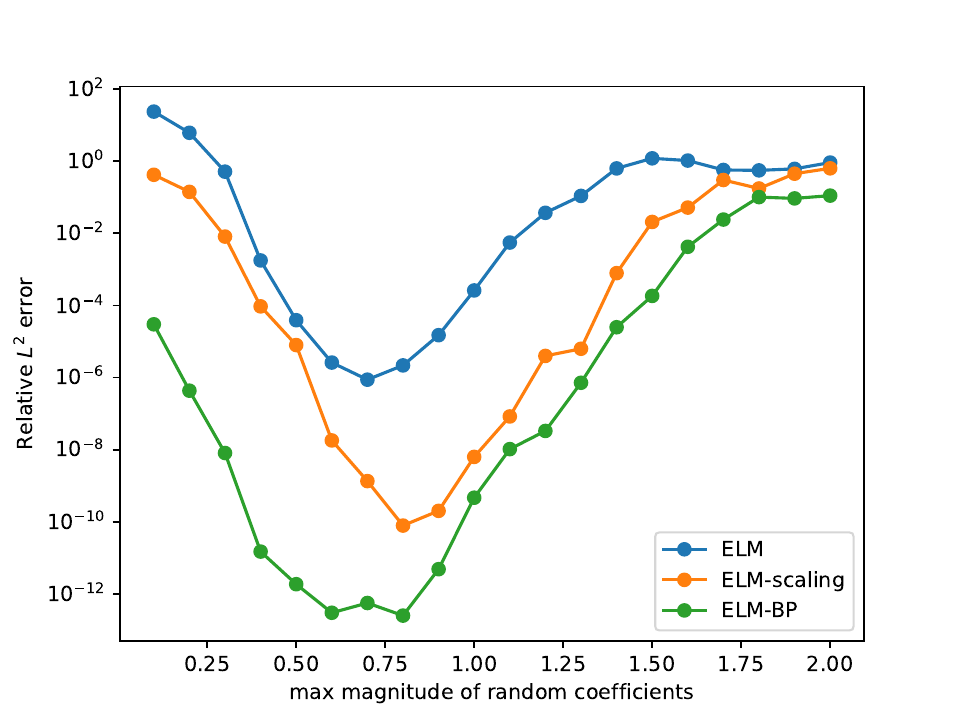}
	\caption{Comparison of relative $L^2$ errors of three methods varying max magnitude of random coefficients for Example 2.2.}  \label{exp2-compare-Rm}
\end{figure}

\subsubsection{Example 2.3}
This example considers the exact solution  with the exponential form
\begin{equation*}
	u(x,y) = \mathrm{e}^{x^2+y^2+xy}
\end{equation*}
on $\Omega = (0,1)\times (0,1)$, with the source term $f(x,y)$ and boundary condition corresponding to this solution.

Similar to Example 2.1, we first present Figs. \ref{exp3-compare-N-M}(a)-(b), which compare   the three methods across varying numbers of collocation points, with the fixed network [2, 100, 300, 1].  
Subsequently, Figs. \ref{exp3-compare-N-M}(c)-(d)   compare  the three methods with different    $M$, where network architecture is [2, 100, $M$, 1], with fixed $N=32$.
Figs. \ref{exp3-compare-N-M}(a) and (c) employ the default initialization, and Figs. \ref{exp3-compare-N-M}(b) and (d) utilize uniform random initialization ($R_m=1$), respectively.
As in Example 2.1,  the RNN-BP method consistently produces the smallest error, with the error level reaching up to $10^{-14}$.  Figs. \ref{exp3-error}(a)-(c) present absolute   errors  for the three methods with network architecture [2,  100,  300,  1] and $N=32$.
The RNN-Scaling method produces   smaller errors than the  RNN method,   with a reduction of up to 5 orders of magnitude when employing uniform random initialization.    

\begin{figure}[!htbp]
	\centering 
	\begin{subfigure}[t]{0.32\linewidth}
		\centering
		\includegraphics[width=1.9in]{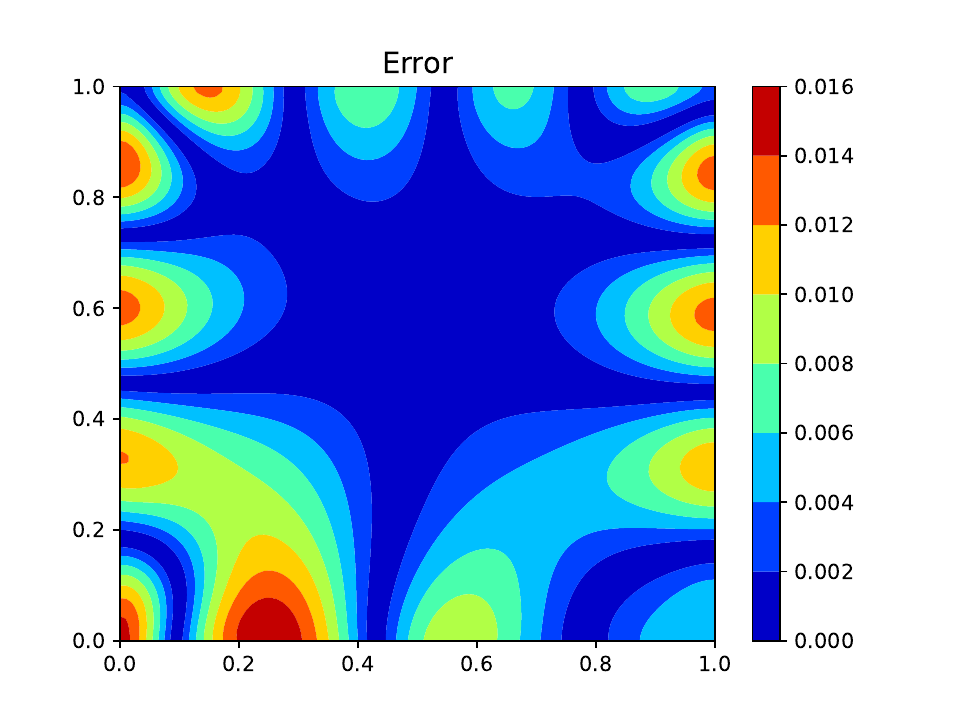}
		\caption{The absolute  error of the  RNN method.}  
	\end{subfigure}
	\hfill
	\begin{subfigure}[t]{0.32\linewidth}
		\centering
		\includegraphics[width=1.9in]{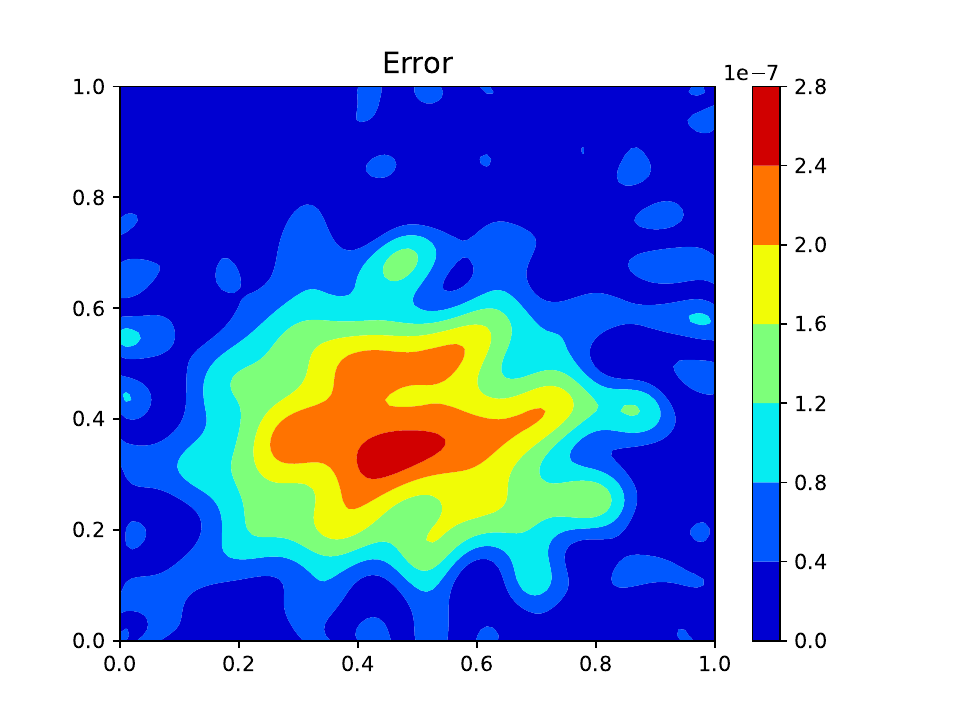}
		\caption{The absolute  error of the RNN-Scaling method.}  
	\end{subfigure}
	\hfill
	\begin{subfigure}[t]{0.32\linewidth}
		\centering
		\includegraphics[width=1.9in]{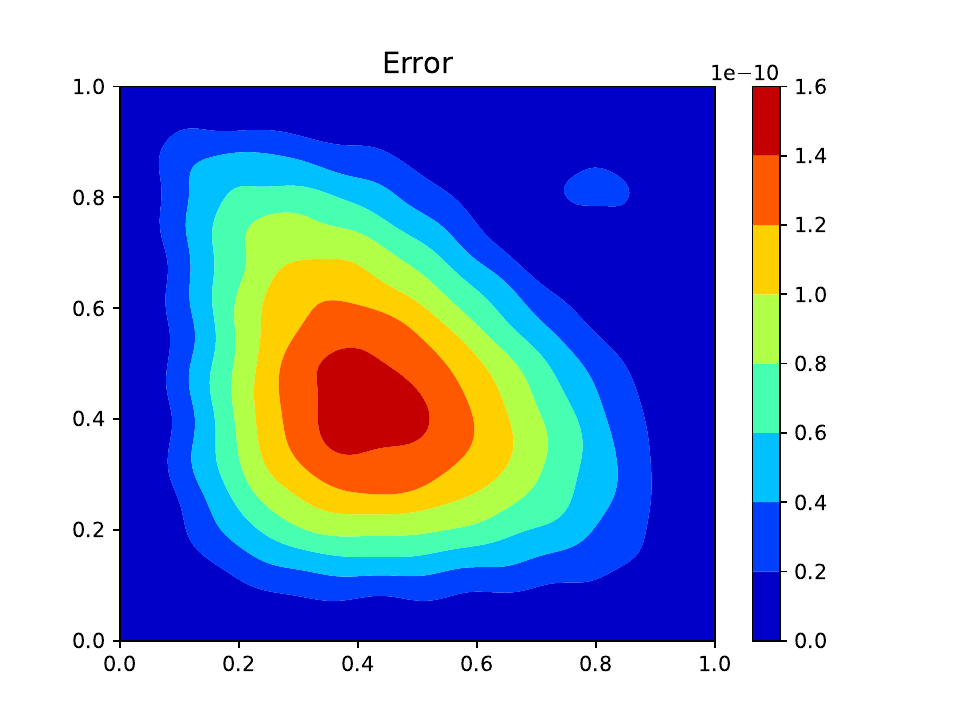}
		\caption{The absolute  error of the RNN-BP method.}  
	\end{subfigure}
	\caption{The absolute  errors of three  methods  with uniform random  initialization    for Example 2.3. (The neural network  architecture is [2, 100, 300, 1] and $N=32$.)}  \label{exp3-error}
\end{figure}

Fig. \ref{exp3-compare-Rm} shows a comparison of   the three methods for different values of $R_m$, with the  fixed network   [2, 100, 300, 1] and $N=32$. The three methods achieve smaller errors when $R_m$ is around 0.2 to 1.3.
The smallest errors of the RNN-BP, RNN-Scaling, and RNN methods are magnitude of  $10^{-13}$, $10^{-10}$, and $10^{-6}$, respectively, indicating the effectiveness of the RNN-BP and RNN-Scaling methods.

\begin{figure}[!htbp]
	\centering 
	\hspace{-10mm}
	\begin{subfigure}[t]{0.40\linewidth}
		\centering
		\includegraphics[width=2.8in]{Pics/exp1/exp1_N_km.pdf}
		\caption{}
	\end{subfigure}
	\hspace{8mm}
	\begin{subfigure}[t]{0.40\linewidth}
		\centering
		\includegraphics[width=2.8in]{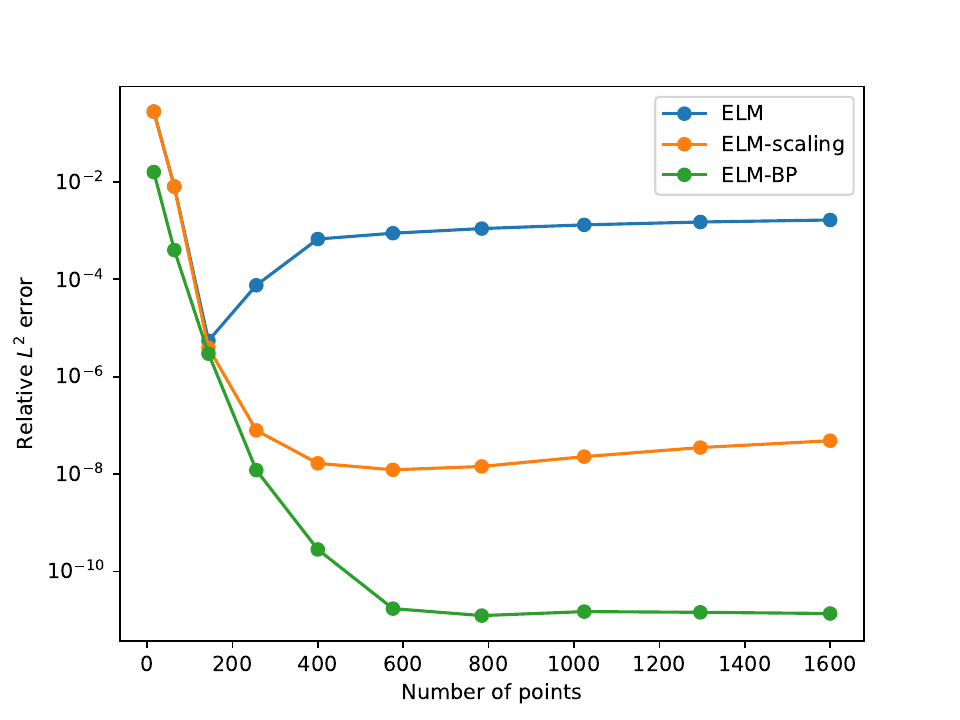}
		\caption{}
	\end{subfigure}
	\\
		\hspace{-10mm}
	\begin{subfigure}[t]{0.40\linewidth}
		\centering
		\includegraphics[width=2.8in]{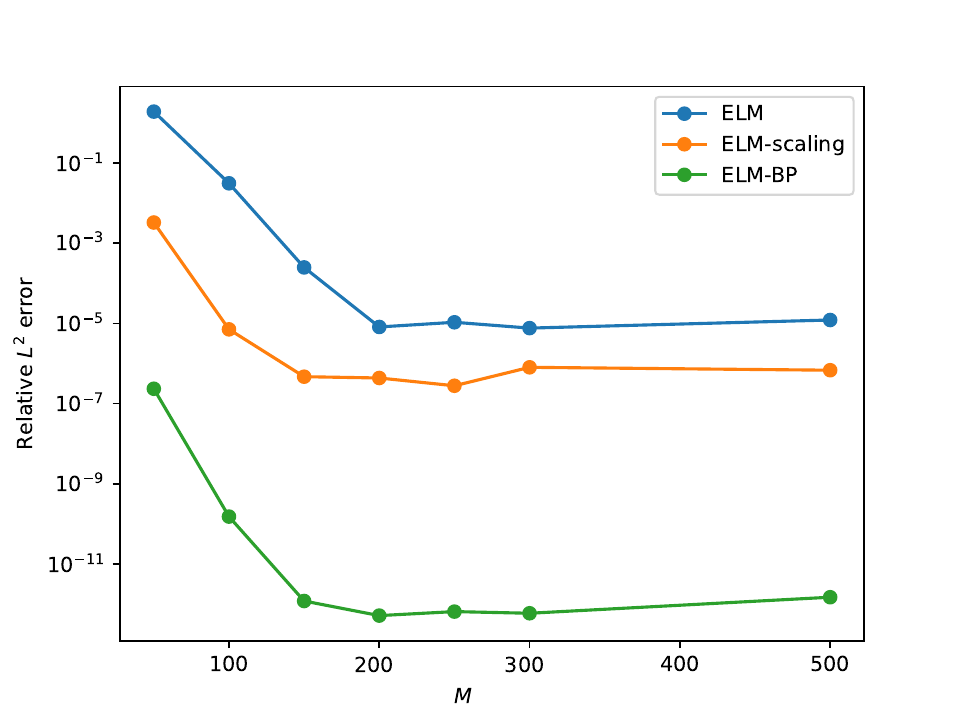}
		\caption{}
	\end{subfigure}
	\hspace{8mm}
	\begin{subfigure}[t]{0.4\linewidth}
		\centering
		\includegraphics[width=2.8in]{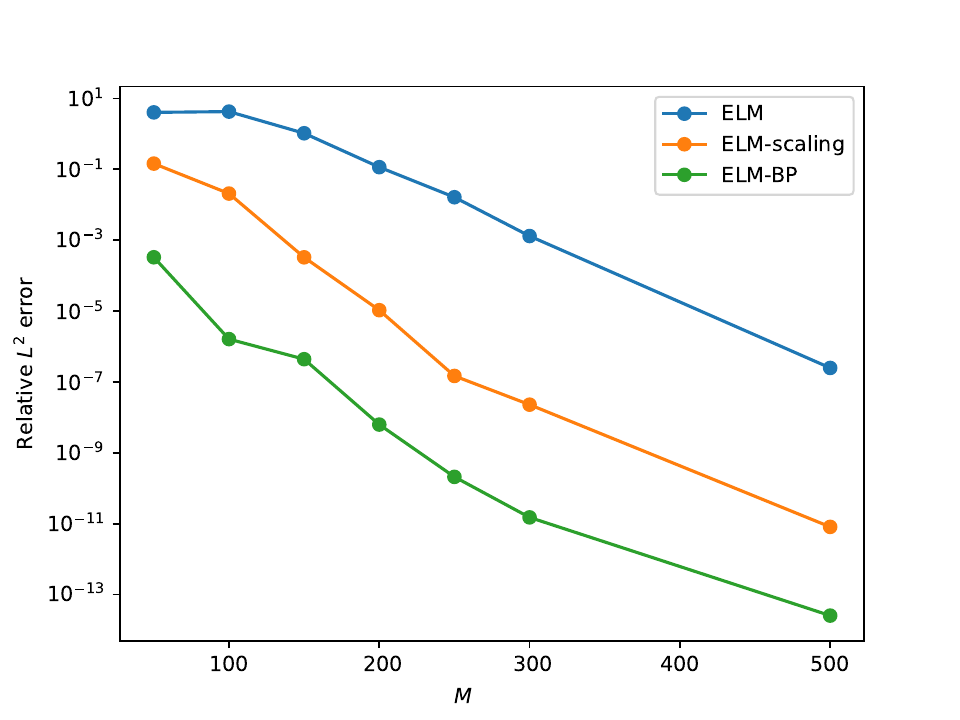}
		\caption{}
	\end{subfigure}
	\caption{Comparison   of three methods for Example 2.3: (a) Varying numbers of points with the default initialization. (b) Varying numbers of points with uniform random initialization ($R_m = 1$).  (c) Varying   $M$ with the  default initialization.
		(d) Varying  $M$ with the uniform random initialization ($R_m = 1$). }  \label{exp3-compare-N-M}
\end{figure}

\begin{figure}[!htbp]
	\centering 
	\includegraphics[width=2.6in]{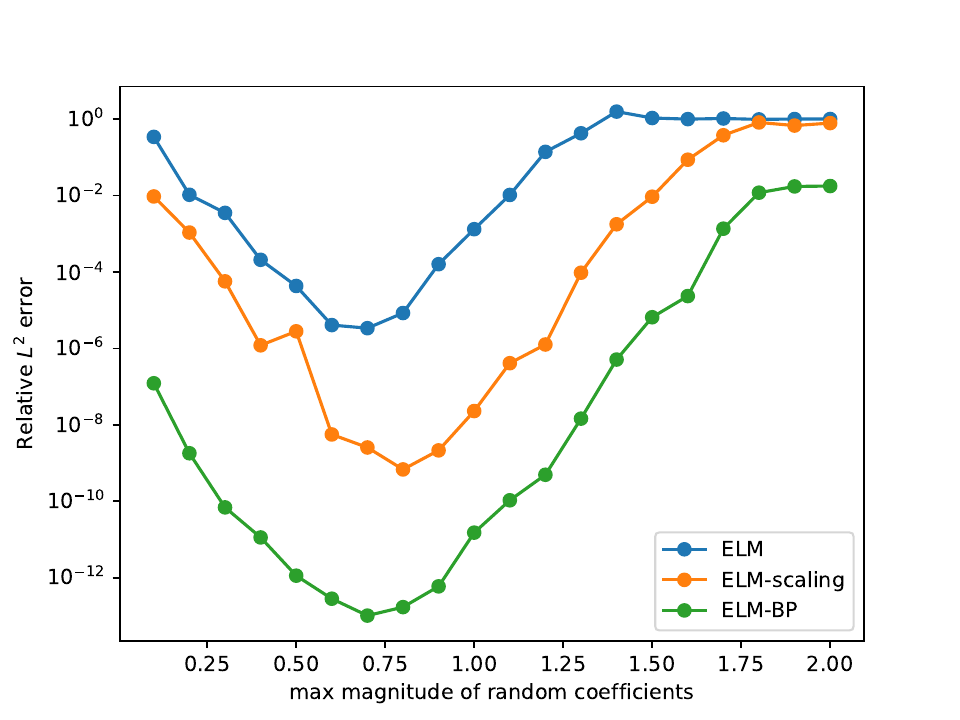}
	\caption{Comparison of relative $L^2$ errors of three methods varying max magnitude of random coefficients for Example 2.3.}\label{exp3-compare-Rm}
\end{figure}

\subsubsection{Example 2.4}
This example validates that the RNN-BP method is exact for solutions of the form \eqref{form-bihar}. We consider the exact solution on $\Omega = (0,1)\times (0,1)$:
\begin{equation*}
	u(x,y) = x^{10} + y^{10} + x^k \sin y + y^k\cos x.
\end{equation*}

 For $k=3$ and $k=4$,  the three methods are compared  for different numbers of collocation points with the fixed network [2,  100,  300,  1]  and uniform random initialization ($R_m=1$), where the results are presented in Table \ref{exp4}.  For the RNN-BP method, it can be seen that its relative $L^2$ errors achieve  machine precision when $k=3$, and  the errors  decrease  rapidly as the numbers of collocation points increase when $k=4$, with the error level reaching up to $10^{-14}$. This verifies that the RNN-BP method is exact for solutions of the form \eqref{form-bihar}. We can also observe that as the numbers of collocation points increase, the errors of the  RNN   initially decrease but then   increase. In contrast, the errors of the  RNN-Scaling      decrease  as the numbers of collocation points increase, and the errors are much smaller than those of the RNN method.

\begin{table}[htbp]
	\centering
	\scriptsize
	\caption{Comparison of three   methods  with different $N$ for Example 2.4.} \label{exp4}
	\begin{tabular}{clccccccc}
		\toprule
		$k$   & $N$   & 4     & 8     & 12    & 16    & 20    & 24    & 28    \\
		\midrule
		3     & RNN   & 7.30e-1 & 5.18e-2 & 5.55e-5 & 2.34e-4 & 3.75e-3 & 5.91e-3 & 7.47e-3 \\
		& RNN-S & 7.30e-1 & 5.18e-2 & 5.16e-5 & 1.07e-6 & 1.44e-7 & 1.35e-7 & 1.49e-7 \\
		& RNN-BP & 4.02e-16 & 4.03e-16 & 4.01e-16 & 9.82e-16 & 4.03e-16 & 4.00e-16 & 4.03e-16 \\
		\midrule
		4     & RNN   & 8.07e-1 & 5.65e-2 & 5.62e-5 & 2.53e-4 & 4.08e-3 & 6.42e-3 & 8.11e-3 \\
		& RNN-S & 8.07e-1 & 5.65e-2 & 5.62e-5 & 1.16e-6 & 1.56e-7 & 1.46e-7 & 1.61e-7 \\
		& RNN-BP & 1.20e-4 & 1.22e-6 & 3.96e-9 & 1.56e-11 & 1.21e-13 & 5.44e-14 & 5.79e-14 \\
		\bottomrule
	\end{tabular}%
\end{table}%

\subsubsection{Example 2.5}
The fifth numerical example considers the following exact solution on $(0,2)\times (0,2)$:
\footnotesize
\begin{equation*}
	u(x, y)=-\left[2 \cos \left(\frac{3}{2} \pi x+\frac{2 \pi}{5}\right)+\frac{3}{2} \cos \left(3 \pi x-\frac{\pi}{5}\right)\right]\left[2 \cos \left(\frac{3}{2} \pi y+\frac{2 \pi}{5}\right)+\frac{3}{2} \cos \left(3 \pi y-\frac{\pi}{5}\right)\right].
\end{equation*}
\normalsize 
The figure of  exact solution is displayed in \ref{exp5-error}(a).

Fig.  \ref{exp5-compare}(a) illustrates the  comparison of the three methods for different  numbers of collocation points, with the  fixed network   [2, 100, 500, 1] and uniform random initialization ($R_m=1$).  Fig. \ref{exp5-compare}(b) illustrates the comparison of the methods for different values of $M$, with  the fixed  network [2, 100, $M$, 1] and uniform random initialization ($R_m=1$).  Finally, Fig. \ref{exp5-compare}(c) shows a comparison of the methods for different  values of $R_m$, with the  fixed network [2, 100, 500, 1] and $N=32$. The relative $L^2$ errors   for the three methods reach the   levels of $10^{-4}$, $10^{-6}$, and $10^{-10}$, respectively. The RNN-BP method consistently achieves  the smallest error, followed by the  RNN-Scaling, RNN method, which   reflects the effectiveness of the RNN-BP and RNN-Scaling methods. The absolute errors are plotted in Figs.  \ref{exp5-error}(c)-(d), the 
error of the RNN method  is  concentrated around the boundary, while the errors of the other  two methods do not.

\begin{figure}[!htbp]
	\centering 
	\begin{subfigure}[t]{0.3\linewidth}
		\centering
		\includegraphics[width=1.9in]{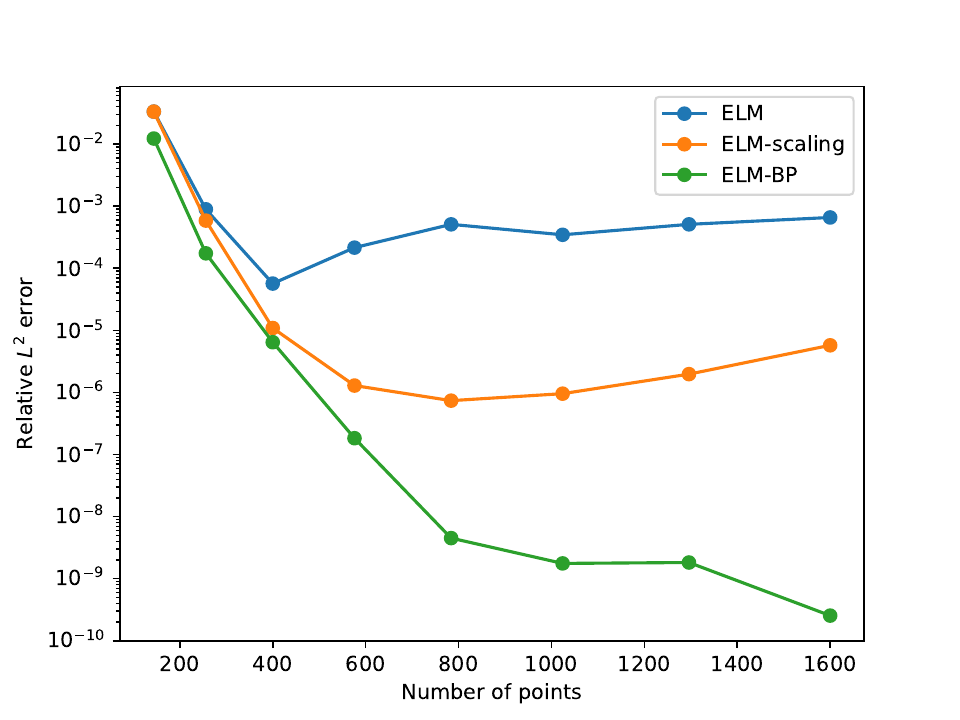}
		\caption{}
	\end{subfigure}
	\hfill
	\begin{subfigure}[t]{0.3\linewidth}
		\centering
		\includegraphics[width=1.9in]{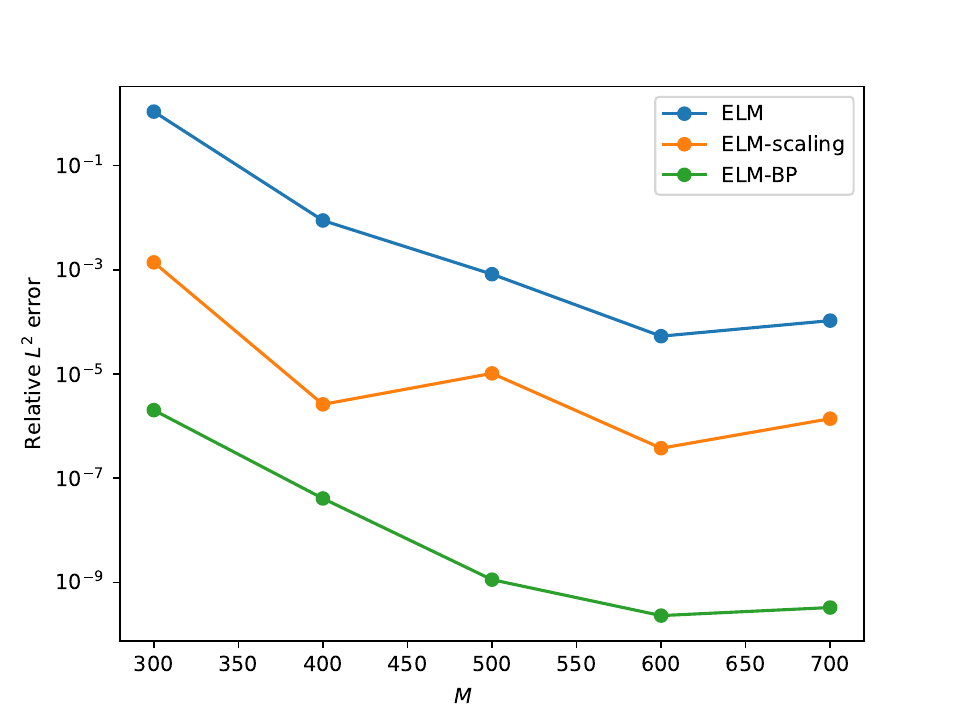}
		\caption{}
	\end{subfigure}
	\hfill 
	\begin{subfigure}[t]{0.3\linewidth}
		\centering
		\includegraphics[width=1.9in]{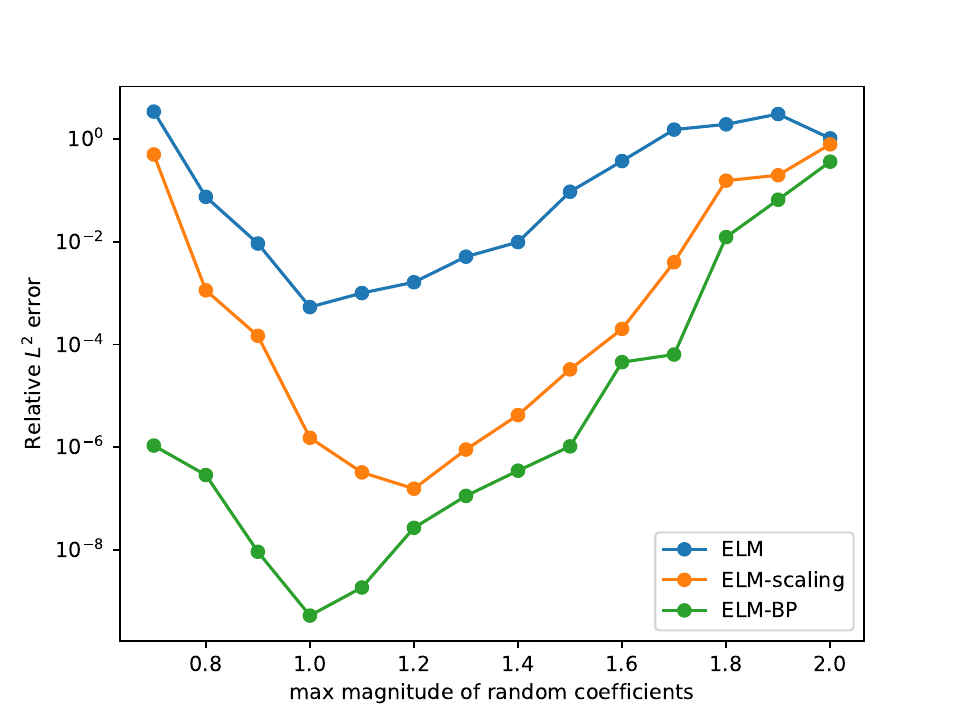}
		\caption{}
	\end{subfigure}
	\caption{Comparison of relative $L^2$ errors of three methods for Example 2.5:   (a) Varying numbers of points with uniform random initialization ($R_m = 1$) and the  fixed network architecture [2, 100, 500, 1].  (b) Varying     $M$ with uniform random initialization ($R_m = 1$),  network architecture  [2, 100, $M$, 1]  and $N=32$.
		(c) Varying $R_m$ with the  fixed network architecture  [2, 100, 500, 1] and $N = 32$. }  \label{exp5-compare}
\end{figure}

\begin{figure}[!htbp]
	\centering 
	\begin{subfigure}[t]{0.23\linewidth}
		\centering
		\includegraphics[width=1.5in]{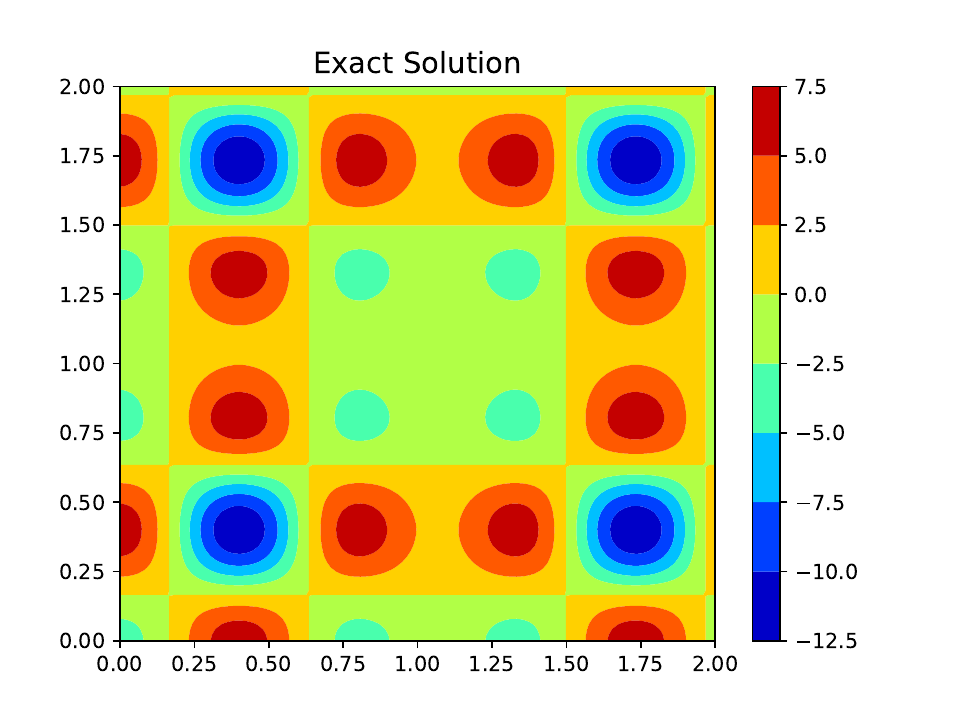}
		\caption{The exact solution.}  
	\end{subfigure}
	\hfill
	\begin{subfigure}[t]{0.23\linewidth}
		\centering
		\includegraphics[width=1.5in]{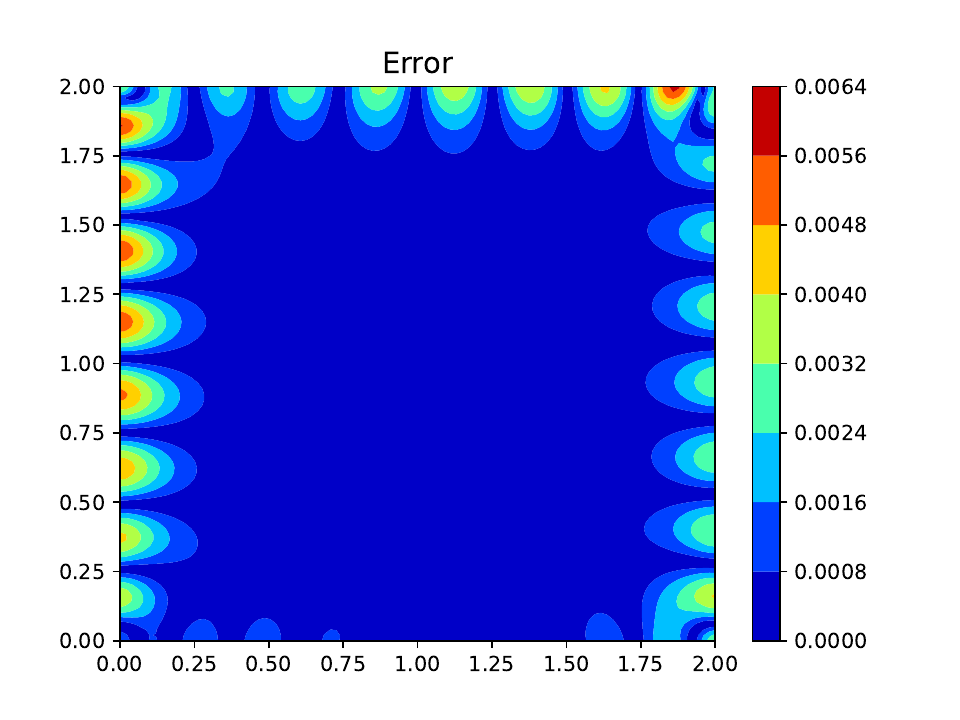}
		\caption{The absolute  error of the RNN method.}  
	\end{subfigure}
	\hfill
	\begin{subfigure}[t]{0.23\linewidth}
		\centering
		\includegraphics[width=1.5in]{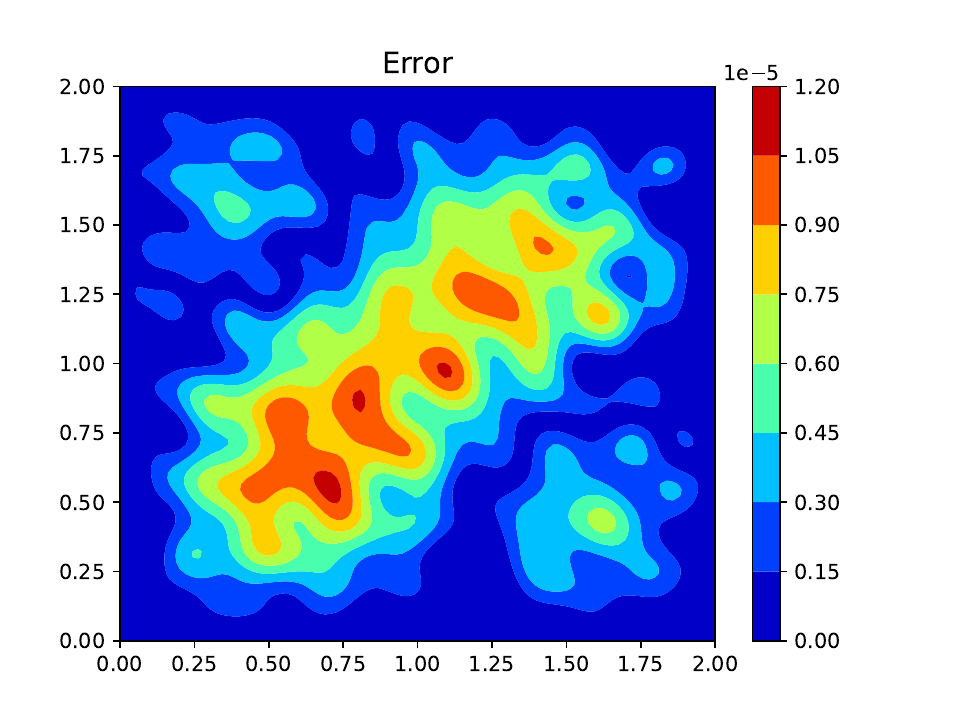}
		\caption{The absolute  error of the  RNN-Scaling method.}  
	\end{subfigure}
	\hfill
	\begin{subfigure}[t]{0.23\linewidth}
		\centering
		\includegraphics[width=1.5in]{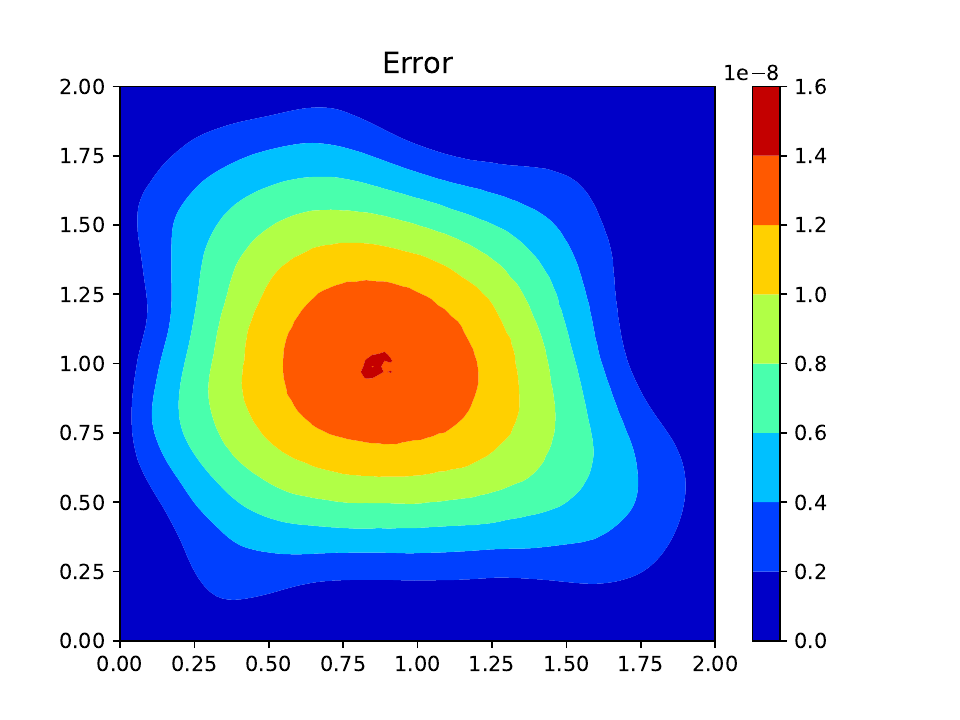}
		\caption{The absolute  error of the RNN-BP method.}  
	\end{subfigure}
	\caption{The exact solution, absolute  errors of three  methods  with uniform  random initialization ($R_m=1$) for Example 2.5. (The   network  architecture is [2, 100, 500, 1] and $N=32$.)}  \label{exp5-error}
\end{figure}

\subsubsection{Example 2.6}
The final example   considers the exact solution  
\begin{equation*}
	u(x,y) = x^4 + y^4
\end{equation*}
on the unit circle centered at   $(0,0)$.

\begin{figure}[!htbp]
	\centering 
	\begin{subfigure}[t]{0.3\linewidth}
		\centering
		\includegraphics[width=1.9in]{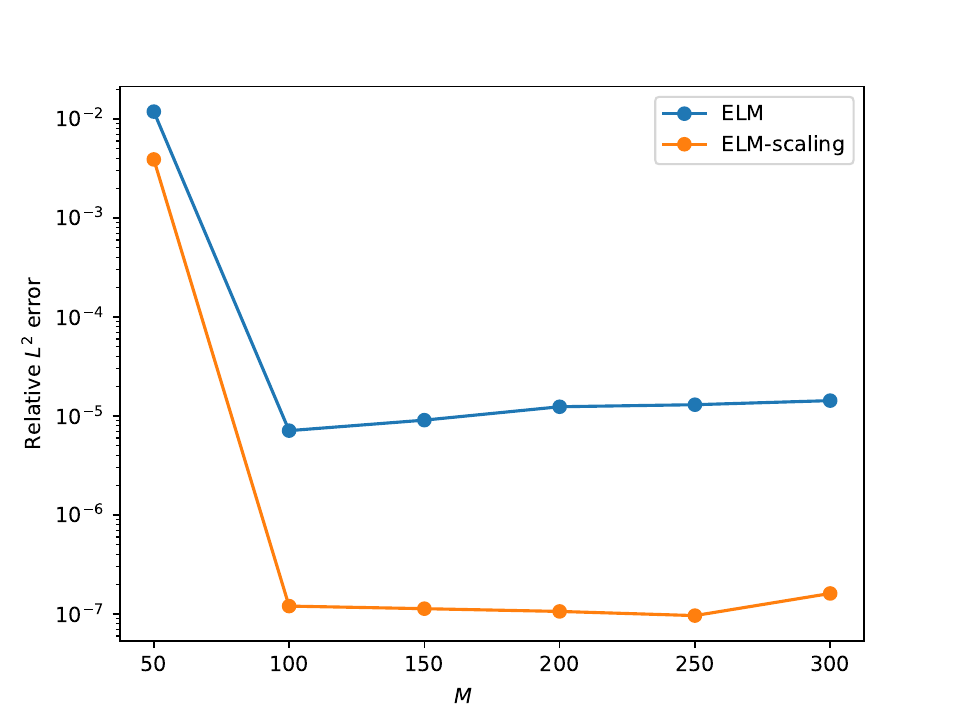}
		\caption{}
	\end{subfigure}
	\hfill
	\begin{subfigure}[t]{0.3\linewidth}
		\centering
		\includegraphics[width=1.9in]{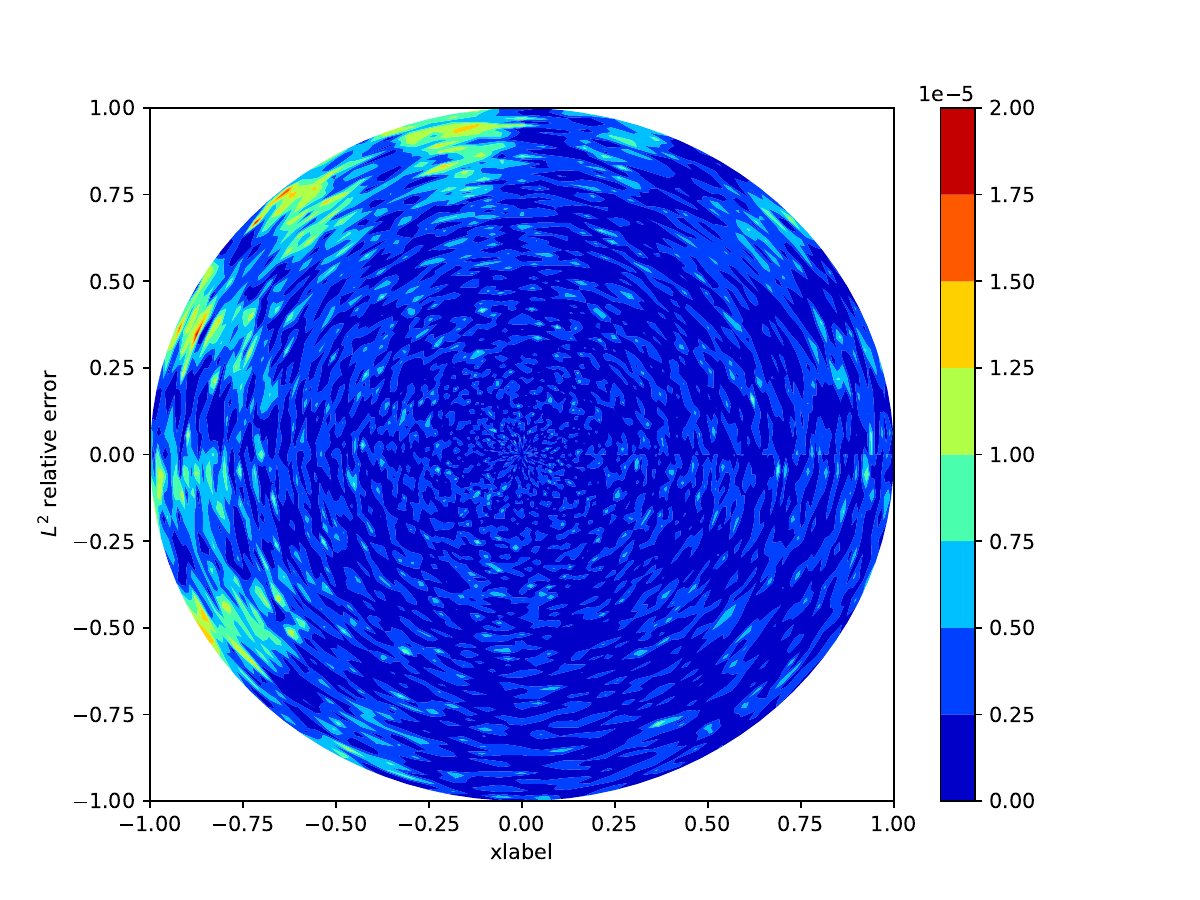}
		\caption{}
	\end{subfigure}
	\hfill 
	\begin{subfigure}[t]{0.3\linewidth}
		\centering
		\includegraphics[width=1.9in]{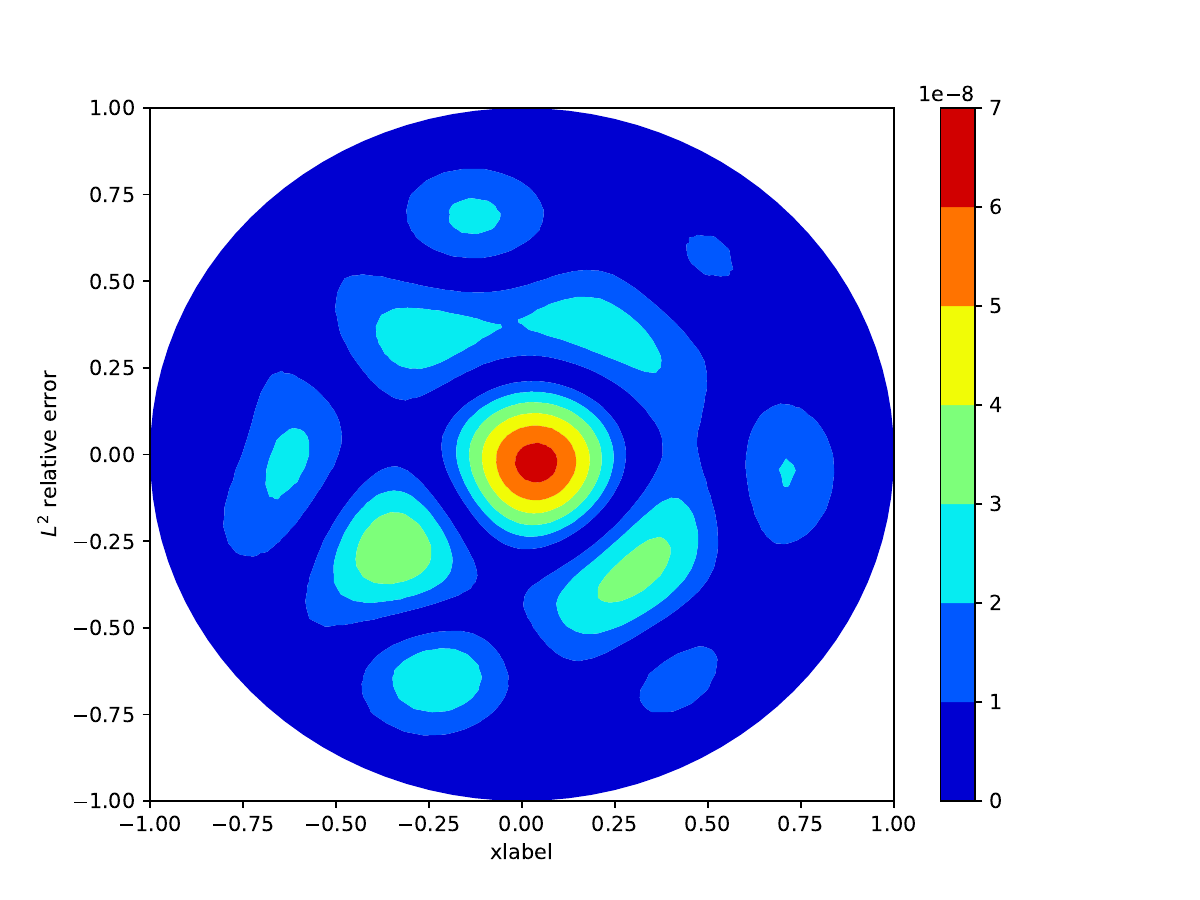}
		\caption{}
	\end{subfigure}
	\caption{Comparison of the RNN and RNN-Scaling methods for Example 2.6:   (a) Varying  $M$ with the default initialization,  network architecture  [2, 100, 100, $M$, 1]  and $N=64$.
		(b) The absolute error of the RNN method.  (c) The absolute error of the RNN-Scaling method. }  \label{exp6-compare}
\end{figure}

We fix $N=64$  
and vary   $M$ to compare the   RNN method  and the RNN-Scaling method, where  the network architecture is [2, 100, 100, $M$, 1] and  the default initialization   is employed.
Fig. \ref{exp6-compare}(a)  presents the comparison results, 
it can be seen that the RNN-Scaling method produces the smallest error, and typically results in a reduction of 2 orders of magnitude in error compared to the RNN method.
Figs. \ref{exp6-compare}(b)-(c) illustrate  the absolute errors of both methods when 
$N=64$ and $M=250$.  We observe  that the error of the RNN method is concentrated around the boundary, while the error of the RNN-Scaling method is concentrated in the interior.

From Examples 2.1 to 2.6,  we summarize as follows.    The  RNN-BP method  consistently achieves the smallest errors  compared to the other two methods.  
In   Examples 2.1 to 2.3, when $M$ and the numbers of the collocation points are relatively large, the errors of the  RNN-BP method are  reduced by 6-7 orders of magnitude compared to the RNN method,  with a  maximum reduction of  up to 9 orders of magnitude.
  
For the  RNN-Scaling method, when $M$ and the numbers of the collocation points are relatively large,  the errors of Examples 2.1 to 2.4 are reduced by    4-5 orders of magnitude;  the errors of Example 2.5 are    reduced by  2-3 orders of magnitude, demonstrating  the efficiency of the RNN-Scaling method.

\section{Conclusion}
This work   proposes two improved RNN methods   for  solving    elliptic  equations.  The first   is the  RNN-BP method that enforces exact boundary conditions    on rectangular domains, including both  Dirichlet and clamped boundary conditions.  The enforcement approach for clamped boundary condition  is  direct and does not need to introduce the auxiliary gradient variables, which reduces computation and avoids potential additional errors. We demonstrate theoretically and numerically that the RNN-BP method is exact for some solutions with specific forms.

Secondly, the RNN-Scaling method is introduced, which modifies the linear algebraic equations by increasing the weight of  boundary equations.    The method is  extended to the  circular domain and  the  errors  are 1-2 orders of magnitude smaller than those of the   RNN, with the potential for further generalization to general domains.

Finally, we present   several  numerical examples  to compare the performance of  the three  methods. Both  of the improved randomized neural network methods achieve  higher accuracy than the   RNN method.  The error is reduced by 6 orders of magnitude for some tests.

\section*{CRediT authorship contribution statement}
$\mathbf{Huifang, Zhou:}$  Writing-original draft, Visualization, Validation,  Software,  Methodology, Investigation, Formal analysis, Conceptualization, Funding acquisition.
$\mathbf{Zhiqiang, Sheng:}$ Conceptualization,  Writing-review \& editing, Funding acquisition, Supervision,   Project administration,  Software.

\section*{Declaration of competing interest}
The authors declare that they have no known competing financial interests or personal relationships that could have appeared to influence the work reported in this paper.

\section*{Data availability}
Data will be made available on request.

\section*{Acknowledgement}
 This work is partially supported by
	the National Natural Science Foundation of China (12201246,12071045),  Fund of National Key Laboratory of Computational Physics, LCP Fund for Young Scholar
	(6142A05QN22010), 
	  National Key R\&D Program of China (2020YFA0713601),
	and by the Key Laboratory of Symbolic Computation and Knowledge Engineering of Ministry of Education, Jilin University, Changchun, 130012, P.R. China.
 
\bibliographystyle{siam}
\bibliography{library}

\begin{thebibliography}{10}

\bibitem{Barrett1999a}
{\sc J.~W. Barrett, J.~F. Blowey, and H.~Garcke}, {\em Finite element
  approximation of the {{Cahn--Hilliard}} equation with degenerate mobility},
  SIAM Journal on Numerical Analysis, 37 (1999), pp.~286--318.

\bibitem{Ben-Artzi2009}
{\sc M.~{Ben-Artzi}, I.~Chorev, J.-P. Croisille, and D.~Fishelov}, {\em A
  compact difference scheme for the biharmonic equation in planar irregular
  domains}, SIAM Journal on Numerical Analysis, 47 (2009), pp.~3087--3108.

\bibitem{Bialecki2012}
{\sc B.~Bialecki}, {\em A fourth order finite difference method for the
  {{Dirichlet}} biharmonic problem}, Numerical Algorithms, 61 (2012),
  pp.~351--375.

\bibitem{Bialecki2010}
{\sc B.~Bialecki and A.~Karageorghis}, {\em Spectral {{Chebyshev}} collocation
  for the {{Poisson}} and biharmonic equations}, SIAM Journal on Scientific
  Computing, 32 (2010), pp.~2995--3019.

\bibitem{Chen2020}
{\sc J.~Chen, R.~Du, and K.~Wu}, {\em A comparison study of deep {{Galerkin}}
  method and deep {{Ritz}} method for elliptic problems with different boundary
  conditions}, Communications in Mathematical Research, 36 (2020),
  pp.~354--376.

\bibitem{Cui2020}
{\sc M.~Cui and S.~Zhang}, {\em On the uniform convergence of the weak
  {{Galerkin}} finite element method for a singularly-perturbed biharmonic
  equation}, Journal of Scientific Computing, 82 (2020), pp.~1--15.

\bibitem{Doha2008}
{\sc E.~Doha and A.~Bhrawy}, {\em Efficient spectral-{{Galerkin}} algorithms
  for direct solution of fourth-order differential equations using {{Jacobi}}
  polynomials}, Applied Numerical Mathematics, 58 (2008), pp.~1224--1244.

\bibitem{Dong2021a}
{\sc S.~Dong and N.~Ni}, {\em A method for representing periodic functions and
  enforcing exactly periodic boundary conditions with deep neural networks},
  Journal of Computational Physics, 435 (2021), p.~110242.

\bibitem{E2018}
{\sc W.~E and B.~Yu}, {\em The deep {{Ritz}} method: {{A}} deep learning-based
  numerical algorithm for solving variational problems}, Communications in
  Mathematics and Statistics, 6 (2018), pp.~1--12.

\bibitem{Henn2005}
{\sc S.~Henn}, {\em A multigrid method for a fourth-order diffusion equation
  with application to image processing}, SIAM Journal on Scientific Computing,
  27 (2005), pp.~831--849.

\bibitem{Jagtap2020}
{\sc A.~D. Jagtap, E.~Kharazmi, and G.~E. Karniadakis}, {\em Conservative
  physics-informed neural networks on discrete domains for conservation laws:
  {{Applications}} to forward and inverse problems}, Computer Methods in
  Applied Mechanics and Engineering, 365 (2020), p.~113028.

\bibitem{Lagaris1998}
{\sc I.~Lagaris, A.~Likas, and D.~Fotiadis}, {\em Artificial neural networks
  for solving ordinary and partial differential equations}, IEEE Transactions
  on Neural Networks, 9 (1998), pp.~987--1000.

\bibitem{Lagaris2000}
{\sc I.~Lagaris, A.~Likas, and D.~Papageorgiou}, {\em Neural-network methods
  for boundary value problems with irregular boundaries}, IEEE Transactions on
  Neural Networks, 11 (2000), pp.~1041--1049.

\bibitem{LePotier2005a}
{\sc C.~Le~Potier}, {\em Finite volume monotone scheme for highly anisotropic
  diffusion operators on unstructured triangular meshes}, Comptes Rendus
  Mathematique, 341 (2005), pp.~787--792.

\bibitem{liu2023a}
{\sc Y.~Liu and W.~Ma}, {\em Gradient auxiliary physics-informed neural network
  for nonlinear biharmonic equation}, Engineering Analysis with Boundary
  Elements, 157 (2023), pp.~272--282.

\bibitem{Lyu2021}
{\sc L.~Lyu, K.~Wu, R.~Du, and J.~Chen}, {\em Enforcing exact boundary and
  initial conditions in the deep mixed residual method}, CSIAM Transactions on
  Applied Mathematics, 2 (2021), pp.~748--775.

\bibitem{Lyu2022}
{\sc L.~Lyu, Z.~Zhang, M.~Chen, and J.~Chen}, {\em {{MIM}}: {{A}} deep mixed
  residual method for solving high-order partial differential equations},
  Journal of Computational Physics, 452 (2022), p.~110930.

\bibitem{Mai-Duy2007}
{\sc N.~{Mai-Duy} and R.~Tanner}, {\em A spectral collocation method based on
  integrated {{Chebyshev}} polynomials for two-dimensional biharmonic
  boundary-value problems}, Journal of Computational and Applied Mathematics,
  201 (2007), pp.~30--47.

\bibitem{Ming2006}
{\sc W.~Ming and J.~Xu}, {\em The {{Morley}} element for fourth order elliptic
  equations in any dimensions}, Numerische Mathematik, 103 (2006),
  pp.~155--169.

\bibitem{Ming2007}
\leavevmode\vrule height 2pt depth -1.6pt width 23pt, {\em Nonconforming
  tetrahedral finite elements for fourth order elliptic equations}, Mathematics
  of Computation, 76 (2007), pp.~1--19.

\bibitem{Mohanty2000}
{\sc R.~Mohanty}, {\em A fourth-order finite difference method for the general
  one-dimensional nonlinear biharmonic problems of first kind}, Journal of
  Computational and Applied Mathematics, 114 (2000), pp.~275--290.

\bibitem{Mozolevski2003}
{\sc I.~Mozolevski and E.~S{\"u}li}, {\em A priori error analysis for the
  hp-version of the discontinuous {{Galerkin}} finite element method for the
  biharmonic equation}, Computational Methods in Applied Mathematics, 3 (2003),
  pp.~596--607.

\bibitem{Park2013}
{\sc C.~Park and D.~Sheen}, {\em A quadrilateral {{Morley}} element for
  biharmonic equations}, Numerische Mathematik, 124 (2013), pp.~395--413.

\bibitem{Raissi2019}
{\sc M.~Raissi, P.~Perdikaris, and G.~Karniadakis}, {\em Physics-informed
  neural networks: {{A}} deep learning framework for solving forward and
  inverse problems involving nonlinear partial differential equations}, Journal
  of Computational Physics, 378 (2019), pp.~686--707.

\bibitem{Sheng2021a}
{\sc H.~Sheng and C.~Yang}, {\em {{PFNN}}: {{A}} penalty-free neural network
  method for solving a class of second-order boundary-value problems on complex
  geometries}, Journal of Computational Physics, 428 (2021), p.~110085.

\bibitem{Sheng2008}
{\sc Z.~Sheng and G.~Yuan}, {\em A nine point scheme for the approximation of
  diffusion operators on distorted quadrilateral meshes}, SIAM Journal on
  Scientific Computing, 30 (2008), pp.~1341--1361.

\bibitem{Sheng2016}
\leavevmode\vrule height 2pt depth -1.6pt width 23pt, {\em A new nonlinear
  finite volume scheme preserving positivity for diffusion equations}, Journal
  of Computational Physics, 315 (2016), pp.~182--193.

\bibitem{Sirignano2018}
{\sc J.~Sirignano and K.~Spiliopoulos}, {\em {{DGM}}: {{A}} deep learning
  algorithm for solving partial differential equations}, Journal of
  Computational Physics, 375 (2018), pp.~1339--1364.

\bibitem{sukumar2022}
{\sc N.~Sukumar and A.~Srivastava}, {\em Exact imposition of boundary
  conditions with distance functions in physics-informed deep neural networks},
  Computer Methods in Applied Mechanics and Engineering, 389 (2022), p.~114333.

\bibitem{Ventsel2001}
{\sc E.~Ventsel and T.~Krauthammer}, {\em Thin {{Plates}} and {{Shells}}:
  {{Theory}}, {{Analysis}}, and {{Applications}}}, Marcel Deekker Inc, New
  York, 2nd edition~ed., 2001.

\bibitem{Xie2023}
{\sc Y.~Xie, Y.~Ma, and Y.~Wang}, {\em Automatic boundary fitting framework of
  boundary dependent physics-informed neural network solving partial
  differential equation with complex boundary conditions}, Computer Methods in
  Applied Mechanics and Engineering, 414 (2023), p.~116139.

\bibitem{Xu2023}
{\sc M.~Xu and C.~Shi}, {\em A {{Hessian}} recovery-based finite difference
  method for biharmonic problems}, Applied Mathematics Letters, 137 (2023),
  p.~108503.

\bibitem{Zang2020}
{\sc Y.~Zang, G.~Bao, X.~Ye, and H.~Zhou}, {\em Weak adversarial networks for
  high-dimensional partial differential equations}, Journal of Computational
  Physics, 411 (2020), p.~109409.

\bibitem{Zhang2015}
{\sc R.~Zhang and Q.~Zhai}, {\em A weak {{Galerkin}} finite element scheme for
  the biharmonic equations by using polynomials of reduced order}, Journal of
  Scientific Computing, 64 (2015), pp.~559--585.

\bibitem{Zhang2020b}
{\sc S.~Zhang}, {\em Minimal consistent finite element space for the biharmonic
  equation on quadrilateral grids}, IMA Journal of Numerical Analysis, 40
  (2020), pp.~1390--1406.

\end{thebibliography}
\end{document}